\documentclass[12pt]{amsart}
 
\usepackage{geometry}
\usepackage[utf8]{inputenc}
\usepackage[english]{babel}
\usepackage{amssymb}
\usepackage{amsmath}
\usepackage{amsthm}
\usepackage{amscd}
\usepackage{enumitem}
\usepackage{graphicx}
\usepackage{tikz}

\usepackage{relsize}

\usepackage{subcaption}

\newcommand{\eps}{\varepsilon}
\newcommand{\vect}{\text{span}}
\newcommand{\gen}{\text{Gen}}

\DeclareMathOperator{\per}{Per} 
\DeclareMathOperator{\prt}{par} 
\DeclareMathOperator{\Chi}{Chi}
\DeclareMathOperator{\capac}{cap}
\newcommand{\rt}{{\tt{root}}} 

\def\RR{\mathbb R}
\def\NN{\mathbb N}
\def\ZZ{\mathbb Z}
\def\CC{\mathbb C}
\def\KK{\mathbb K}

\theoremstyle{plain}
\newtheorem{theorem}{Theorem}[section]
\newtheorem{lemma}[theorem]{Lemma}
\newtheorem{corollary}[theorem]{Corollary}
\newtheorem{proposition}[theorem]{Proposition}

\theoremstyle{definition}
\newtheorem{definition}[theorem]{Definition}
\newtheorem{example}[theorem]{Example}

\theoremstyle{remark}
\newtheorem{remark}[theorem]{Remark}

\numberwithin{equation}{section}
\numberwithin{figure}{section}

\title{Chaotic weighted shifts on directed trees}
\author{Karl-G. Grosse-Erdmann and Dimitris Papathanasiou}
\address{Karl-G. Grosse-Erdmann, 
D\'epartement de Math\'ematique, Universit\'e de Mons, 20 Place du Parc, 7000 Mons, Belgium}
\email{kg.grosse-erdmann@umons.ac.be}
\address{Dimitris Papathanasiou, Sabanc{\i} University Tuzla Campus, Orta Mahalle, \"Universite Caddesi No:27 Tuzla, 34956 \.{I}stanbul, Turkey}
\email{d.papathanasiou@sabanciuniv.edu}

\thanks{The first author was supported by the Fonds de la Recherche Scientifique - FNRS under Grant n\textsuperscript{o} CDR J.0078.21.}

\keywords{Weighted shift operator, directed tree, chaotic operator, hypercyclic operator, mixing operator, fixed point, continued fraction, boundary of a tree, capacity}
\subjclass[2010]{Primary 47A16, 47B37; Secondary 05C05, 11A55, 31C20, 46A45}

\begin{document}

\begin{abstract}
We study the dynamical behaviour of weighted backward shift operators defined on sequence spaces over a directed tree. We provide a characterization of chaos on very general Fr\'echet sequence spaces in terms of the existence of a large supply of periodic points, or of fixed points. In the special case of the space $\ell^p$, $1\leq p<\infty$, or the space $c_0$ over the tree, we provide a characterization directly in terms of the weights of the shift operators. It has turned out that these characterizations involve certain generalized continued fractions that are introduced in this paper. Special attention is given to weighted backward shifts with symmetric weights, in particular to Rolewicz operators. In an appendix, we complement our previous work by characterizing hypercyclic and mixing weighted backward shifts on very general Fr\'echet sequence spaces over a tree. Also, some of our results have a close link with potential theory on flows over trees; the link is provided by the notion of capacity, as we explain in an epilogue.
\end{abstract}

\maketitle

\section{Introduction}\label{s-intro}

Shift operators play a fundamental r\^ole in operator theory. This is particularly true in linear dynamics, where they have been studied for a long time, see \cite{Rol}, \cite{Salas}, \cite{Gro00}, \cite{MaPe02}, and where any new notion is usually first tested on (weighted) shift operators. These operators have also been instrumental in distinguishing between various dynamical properties, see \cite{BaGr07}, \cite{BaRu15}, \cite{BMPP16}, and \cite{BoGE18}. Now, for any theory, a large supply of easy and flexible examples is important. Recently, shift-like composition operators have been investigated, see \cite{BDP18}, \cite{DDM22}, and \cite{DaPi21}. As have been shifts on trees.

Analysis on trees seems to have been initiated by the work of Cartier \cite{Car72}, \cite{Car73}, when he introduced harmonic functions on them. In recent years there has been an increased interest in the analysis on discrete structures like graphs, networks, and trees, see for example \cite{Ana11}, \cite{LyPe16}, \cite{PoGl16}, \cite{KLW21}, and \cite{JoPe23}.

Traditionally, shift operators act as unilateral shifts on the rooted tree $\NN_0$ of natural numbers, or as bilateral shifts on the unrooted tree $\ZZ$ of integers. It is therefore quite natural to study the action of shift operators on directed graphs, and in particular on trees. A first systematic study of shifts on trees from an operator theoretic point of view has been undertaken by Jab{\l}o\'nski, Jung and Stochel \cite{JJS12}.

This was then taken up by Mart\'{i}nez-Avenda\~{n}o \cite{Mar17}, who initiated the study of dynamical properties of shift operators on trees, see also \cite{MaRi20}. The present authors \cite{GrPa21} have continued his study: we have obtained a complete characterization of hypercyclic weighted backward shift operators on arbitrary trees if the underlying sequence space is of type $\ell^p$, $1\leq p<\infty$, or $c_0$. Recall that an operator is hypercyclic if it admits a dense orbit. As for the corresponding weighted forward shift operators, Mart\'{i}nez-Avenda\~{n}o \cite{Mar17} had already shown that they can only be hypercyclic if they are defined on one of the trees $\ZZ$ or $-\NN_0$, in which case they can be regarded as weighted backward shift operators. The hypercyclicity and universality of other (sequences of) operators on trees has been studied in \cite{Pav92}, and more recently in \cite{CoMa17}, \cite{ANP17}, \cite{BNS20}, \cite{BNN21}, \cite{ANP21}, and \cite{CPP22}.

The next step then is to study the property of chaos. As first proposed by Devaney \cite{Dev89}, and then adopted by Godefroy and Shapiro \cite{GoSh91} in linear dynamics, chaos demands hypercyclicity of the operator together with a dense set of periodic points. Chaos for weighted backward shift operators on the classical trees $\NN_0$ and $\ZZ$ was characterized by the first author \cite{Gro00}. The main aim of this paper is to characterize chaotic weighted backward shift operators on arbitrary (rooted or unrooted) trees, when the underlying sequence space is of type $\ell^p$, $1\leq p<\infty$, or $c_0$. 

Typically, the difficult part in the characterization of chaos for a class of operators is the discussion of hypercyclicity. Detecting, in addition a dense set of periodic points is rather a linear algebraic task that is often easily solved. This is particularly true for weighted backward shift operators, even when they are defined on trees, because it is quite obvious which sequences are periodic points for such an operator. We refer also to the notion of a backward invariant sequence in Definition \ref{d-backwinv}.

However, the problem comes to life again when one tries to determine which of the periodic points belong to the space under consideration. This problem turned out to be much more demanding than we initially expected. 

The paper is organized as follows. Section \ref{s-notation} presents the basic definitions and fixes notation. In the following two sections we study weighted backward shifts on Fr\'echet sequence spaces $X$ over an arbitrary tree $V$; our only restriction is that the canonical unit sequences should form an unconditional basis in $X$. We show in Section \ref{s-periodic} that such a shift is already chaotic if it has a dense set of periodic points; for rooted trees and for unrooted trees with a free left end it is already enough that every vertex $v\in V$ supports a fixed point $f$ with $f(v)=1$. In the main result of Section \ref{s-unrooted} we characterize chaos on arbitrary unrooted trees via the existence of certain fixed points (or, rather, backward invariant points) on some related rooted trees. 

The next five sections are devoted to weighted backward shifts on sequence spaces $\ell^p$, $1\leq p<\infty$, and $c_0$ over an arbitrary tree. It has turned out that certain constants $c_p$ and $r_p$, $1\leq p\leq \infty$, play a crucial r\^ole in the characterization of chaos on such spaces; these constants are defined on weighted rooted trees. The two types of constants are related, and their form resembles that of continued fractions. In Section \ref{s-constcprp} we introduce these constants and discuss some of their properties. In Section \ref{s-backwinv} we characterize the existence of backward invariant sequences on a weighted rooted tree in terms of the constants $c_p$ and $r_p$.
 This then allows us, in Section \ref{s-chaosrooted}, to characterize chaos for weighted backward shifts on spaces $\ell^p$, $1\leq p<\infty$, and $c_0$ over rooted trees. The case of unrooted trees is studied in Section \ref{s-chaosunrootedlc}. We end the main part of the paper with Section \ref{s-special_trees}, where we apply the previous results to some special weighted backward shift operators such as symmetric shifts on symmetric trees, symmetric shifts on arbitrary trees, and the so-called Rolewicz operators of type $\lambda B$.

We add two appendices. Section \ref{s-appendix} is a kind of afterthought to our previous paper \cite{GrPa21} that was suggested by our present work. It had been our intention in \cite{GrPa21} to characterize hypercyclic and mixing weighted backward shift operators  directly in terms of the weights, which was possible in spaces of type $\ell^p$, $1\leq p<\infty$, and $c_0$ thanks to the reverse H\"older inequalities. In this section we present characterizations for very general Fr\'echet sequence spaces on the tree. The conditions are less explicit, but will still be useful; one application is mentioned in Remark \ref{rem-GMM}. In Section \ref{s-chaosPoisson} we discuss how to interpret chaos for the backward shift operator in the framework of harmonic functions on trees. 

\subsection*{Acknowledgements}
After having finished a first version of this paper, we were kindly informed by Nikolaos Chalmoukis that, in the case of rooted trees, there is a close link between some of our results in Section \ref{s-backwinv} on backward invariant sequences and results from potential theory on flows over trees; the link is provided by the notion of capacity. In an epilogue to our paper we explain this link in Subsection \ref{subs-rootI}, and we show in Subsections \ref{subs-boundaryunrooted} - \ref{subs-rootII} that our results lead to some possibly new results in potential theory. We are very grateful to Nikos for many very useful discussions.

\section{Definitions and notation}\label{s-notation}

We will introduce the basic definitions and the notation used in this paper. Some more details may also be found in \cite{GrPa21} or \cite{JJS12}.

\subsection{Trees} A \textit{directed tree} $T=(V,E)$, or simply a \textit{tree} $V$, is a connected directed graph consisting of a countable set $V$ of \textit{vertices} and a set $E\subset (V\times V)\setminus \{(v,v):v\in V\}$ of \textit{directed edges} that has no cycles and for which each vertex $v\in V$ has at most one \textit{parent}, that is, a vertex $w\in V$ such that $(w,v)\in E$; the parent is denoted by $\prt(v)$. There is at most one vertex with no parent, called the \textit{root}; it is sometimes conveniently denoted as $v_0$. 

The set $\Chi(v)$ is the set of all \textit{children} of $v\in V$, that is, the vertices with $v$ as parent. More generally, for $n\geq 1$, 
\[
\Chi^n(v) = \Chi(\Chi(\ldots(\Chi(v)))),\quad \text{($n$ times)},
\]
with $\Chi^0(v)=\{v\}$. A vertex with no children is called a \textit{leaf}. A tree is \textit{locally finite} if each vertex only has a finite number of children.

In a rooted tree with root $v_0$, the set 
\[
\gen_n=\Chi^n(v_0),\ n\in\NN_0,
\]
is called the $n$-th \textit{generation}. In an unrooted tree, one first needs to fix some vertex, call it $v_0$. Then the $n$-th \textit{generation} $\gen_n$, $n\in\ZZ$, (with respect to $v_0$) consists of all vertices that belong to $\Chi^{k+n}(\prt^{k}(v_0))$ for some $k\geq \max(-n,0)$. It will sometimes be convenient to write
\[
\gen(v)\subset V
\]
for the generation that contains the vertex $v\in V$.

In this paper, \textit{subtrees} play a crucial r\^ole, that is subsets of a tree that are trees in their own right (without changing the parent-child relationship). The following subtrees are of importance. In a rooted tree we write
\[
V_m = \bigcup_{n=0}^m \gen_n.
\]
And for an arbitrary tree, rooted or not, we write for $v\in V$
\[
V(v)=\bigcup_{n\geq 0} \Chi^n(v);
\]
the vertices in $V(v)$ are called the \textit{descendants} of $v$ (of \textit{degree} $n$ when they are in $\Chi^n(v)$). This set is denoted by some authors as $\text{Des}(v)$. 

For a vertex $v$ in an unrooted tree $V$, we will also introduce $V_-(v)$ and $V_-^N(v)$, $N\geq 1$, in Section \ref{s-unrooted}; they are subsets but no subtrees of $V$.

The vertices of the form $\prt^n(v)$, $n\geq 0$, are the \textit{ancestors} of $v$ (of \textit{degree} $n$). We say that an unrooted tree has a \textit{free left end} if there is a vertex for which all ancestors only have one child.

A tree $V$ will be called \textit{symmetric} if $|\Chi(u)|=|\Chi(v)|$ whenever $u,v\in V$ belong to the same generation, where we use $|\cdot|$ for the cardinality of a set. We then define
\[
\gamma_n=|\Chi(v)|
\]
when $v\in \gen_n$, where the generations have been enumerated according to a fixed vertex $v_0$.

A \textit{path} is a sequence $(v_1,\ldots,v_n)$ in $V$ with $v_{k-1}=\prt(v_{k})$ for $k=2,\ldots,n$. A \textit{branch} starting from a given vertex $v_1$ is a sequence $(v_1,v_2,\ldots)$ of maximal length in $V$, where $v_{k-1}=\prt(v_{k})$ for $k\geq 2$. The \textit{length} of a tree, possibly infinite, is the supremum of the lengths of branches in $V$. Trees of finite length will play a special r\^ole in this paper. 

We often endow a tree $V$ with a \textit{weight} $\mu$, that is, a family $\mu=(\mu_v)_{v\in V}$ non-zero scalars. We then call $(V,\mu)$ a \textit{weighted tree}. We keep the notation $\mu$ even when we restrict it to a subset of $V$.

\subsection{Sequence spaces} Now let $V$ be an arbitrary finite or countable set. The space $\KK^V$ of all sequences $f=(f(v))_{v\in V}$ over $V$, where $\KK=\RR$ or $\CC$, is endowed with the product topology. The canonical unit sequences are denoted by $e_v=\chi_{\{v\}}$, $v\in V$. A subspace $X$ of $\KK^V$ is called a \textit{Fr\'echet (or Banach) sequence space over} $V$ if it is endowed with a Fr\'echet (resp. Banach) space topology for which the canonical embedding into $\KK^V$ is continuous. The topology of $X$ may be induced by an F-norm, which we will denote by $\|\cdot\|$; see \cite{KPR84}, \cite[Definition 2.9]{GrPe11}.

For us, the most important Banach sequence spaces are the following. Let $\mu=(\mu_v)_{v\in V}$ be a weight. Then we consider
\[
\ell^p(V,\mu) = \Big\{ f\in \KK^V : \|f\|^p:= \sum_{v\in V} |f(v)\mu_v|^p <\infty\Big\},\ 1\leq p< \infty,
\]
with the usual modification for $\ell^\infty(V,\mu)$, and 
\[
c_0(V,\mu)=\{ f\in \KK^V: \forall \eps >0, \exists F\subset V \mbox{ finite}, \forall v\in V\setminus F,  |f(v)\mu_v|<\eps\}, 
\]
which is endowed with the canonical norm of $\ell^\infty(V,\mu)$. If $\mu=1$ for all $v\in V$, the corresponding unweighted spaces are denoted by
\[
\ell^p(V), \, 1\leq p\leq \infty,\ \text{and}\  c_0(V).
\]

In this context, just like in \cite{GrPa21}, the reverse H\"older inequalities will turn out to be crucial, see \cite[Theorem 13]{HLP34}, \cite[Lemma 4.2]{GrPa21}. Let $J$ be a finite or countable set. Let $\mu=(\mu_j)_j \in (\KK\setminus \{0\})^J$. Then
     \begin{align*}   
		\inf_{\|x\|_1=1} \sum_{j\in J} |x_j \mu_j| &=  \inf_{j\in J} |\mu_j|,\\
		\inf_{\|x\|_1=1} \Big(\sum_{j\in J} |x_j \mu_j|^p\Big)^{1/p} &=  \Big(\sum_{j\in J} \frac{1}{|\mu_j|^{p^*}}\Big)^{-1/p^*},\, 1<p<\infty,\\
  	\inf_{\|x\|_1=1} \sup_{j\in J} |x_j \mu_j| &=  \Big(\sum_{j\in J} \frac{1}{|\mu_j|}\Big)^{-1},
     \end{align*}
where $x\in \KK^J$, $\|x\|_1=\sum_{j\in J}|x_j|$, $p^\ast$ is the conjugate exponent, and $\infty^{-1}=0$.

\subsection{Weighted backward shift operators} Let $V$ denote again a tree. Let $\lambda=(\lambda_v)_{v\in V}$ be a family of non-zero scalars, that is, a \textit{weight}. Then the \textit{weighted backward shift} $B_\lambda$ is defined formally on $\KK^V$ by
\[
(B_{\lambda}f)(v)=\sum_{u\in \Chi(v)}\lambda_uf(u), \quad v\in V,
\]
where, as usual, an empty sum is zero; see \cite{Mar17}, \cite{GrPa21}. If $\lambda=1$ for all $v\in V$, the corresponding unweighted backward shift is denoted by $B$. As mentioned in the Introduction, their is no point in studying the corresponding weighted forward shifts, see \cite{JJS12}, in this paper.

A weight $\lambda=(\lambda_v)_{v\in V}$ is called \textit{symmetric} if 
\[
\lambda_u=\lambda_v
\]
whenever $u, v\in V$ belong to the same generation. We then define
\[
\lambda_n=\lambda_v
\]
when $v\in \gen_n$, where the generations have been enumerated according to a fixed vertex $v_0$. A symmetric weight will therefore be indexed as $\lambda=(\lambda_n)_{n}$. By extension we then also call the weighted backward shift $B_\lambda$ \textit{symmetric}. Particular symmetric shifts are the \textit{Rolewicz operators} $\lambda B$, $\lambda\in\KK\setminus\{0\}$; see Subsection \ref{subs-rol}.

In the course of our work it has turned out that we need to distinguish two hypotheses on a weighted backward shift $B_\lambda$. Given a Fr\'echet sequence space $X$ over $V$, we say that $B_\lambda$ \textit{is defined on} $X$ if all the series defining $B_\lambda f$ converge unconditionally for any $f\in X$. By the Banach-Steinhaus theorem, $B_\lambda$ is then a (continuous linear) operator from $X$ to $\KK^V$. On the other hand, we say that $B_\lambda$ \textit{is an operator on} $X$ if, moreover, $B_\lambda f\in X$ for all $f\in X$. In that case, by the closed graph theorem, $B_\lambda$ is a (continuous linear) operator from $X$ to $X$. We also refer here to the notion of a backward invariant sequence as will be introduced in Definition \ref{d-backwinv} and the preceding discussion.

In connection with the iterates of weighted backward shifts the following notation is useful, see \cite{GrPa21}. If $\lambda=(\lambda_v)_{v\in V}$ is a weight and $(v_1,\ldots,v_n)$ is a path in $V$, then we set
\[
\lambda(v_1\to v_n) = \lambda_{v_2}\lambda_{v_3}\cdots\lambda_{v_n}.
\]
Indeed, we have for any $v\in V$ and $n\geq 1$,
\[
(B_\lambda^nf)(v) = \sum_{u\in\Chi^n(v)} \lambda(v\to u) f(u).
\]
Note also that
\begin{equation}\label{eq-lambda}
\lambda(v\to u)=\lambda(v\to w)\lambda(w\to u)
\end{equation}
if there is a path from $v$ to $u$ via $w$.

We recall here the well-known and very useful fact that every weighted backward shift is conjugate to an unweighted one if one modifies the underlying space, see \cite{GrPa21}. More precisely, let $X$ be a Fr\'echet sequence space over $V$ and $\mu=(\mu_v)_{v\in V}$ and $\lambda=(\lambda_v)_{v\in V}$ be weights. Define $X_\mu =\{ f\in \KK^V : (f(v)\mu_v)_{v\in V}\in X\}$, which turns canonically into a Fr\'echet sequence space. Then $\phi_\mu:X_\mu \to X$, $f\to (f(v)\mu_v)_{v\in V}$ is an isometric isomorphism. Now, the diagram
\[
\begin{CD}
X_\mu    @>B>>    X_\mu\\
@V{\phi_\mu}VV @VV{\phi_\mu}V \\
X    @>B_\lambda>> X
\end{CD}
\]
commutes if the weights $\mu$ and $\lambda$ are related via
\[
\lambda_v = \frac{\mu_{\prt(v)}}{\mu_v}
\]
for all $v\in V$ (unless $v$ is the root in a rooted tree). Equivalently, after fixing a vertex $v_0\in V$, conveniently taken as the root if the tree has one, we have 
\begin{equation}\label{eq-conjB2}
\mu_v = \frac{\prod_{k=0}^{m-1}\lambda_{\prt^k(v_0)}}{\prod_{k=0}^{n-1}\lambda_{\prt^k(v)}}\mu_{v_0} = \frac{\lambda(w\to v_0)}{\lambda(w\to v)}\mu_{v_0}, \, v\in V,
\end{equation}
where $w$ is a common ancestor of $v$ and $v_0$, that is, $w=\prt^n(v) = \prt^m(v_0)$ for some $n,m\geq 0$; note that there is always a minimal choice of $n$ and $m$.

Clearly, conjugate operators have the same dynamical properties.

\subsection{Linear dynamics} We finally recall the basic notions from linear dynamics; see the monographs \cite{BM09} and \cite{GrPe11} for introductions to that theory. 

An operator $T:X\to X$ on a separable Fr\'echet space $X$ is called \textit{hypercyclic} if there exists a vector $x\in X$, called \textit{hypercyclic} for $T$, whose orbit $\{T^nx : n\geq 0\}$ is dense in $X$. By the Birkhoff transitivity theorem, $T$ is hypercyclic if and only if it is \textit{topologically transitive}, that is, if for any non-empty open sets $U_1$ and $U_2$ in $X$ there is some $n\geq 0$ such that $T^n(U_1)\cap U_2\neq\varnothing$. If the latter holds for all sufficiently large $n$ (depending on $U_1$ and $U_2$), then $T$ is called \textit{mixing}. An intermediate notion is that of \textit{weak mixing}, which demands that the direct sum $T\oplus T$ is hypercyclic on $X\times X$.

Now, a vector $x\in X$ is called a \textit{periodic point} for $T$ if $T^Nx=x$ for some $N\geq 1$, and $N$ is called a \textit{period} of $x$. In the case of $N=1$ this is, of course, a \textit{fixed point} of $T$. The set of periodic points for $T$ is denoted by $\per(T)$. And the operator $T:X\to X$ is called \textit{chaotic} if it is hypercyclic and has a dense set of periodic points.

We will also introduce the notions of a \textit{fixed point over a subtree} and that of a ($\lambda$-)\textit{backward invariant sequence over a tree or a subtree}, see Definitions \ref{d-fpover} and  \ref{d-backwinv}.

\section{The importance of periodic points}\label{s-periodic}
Let $V$ be a tree, $X$ a separable Fr\'echet sequence space over $V$, and $B_\lambda$ a weighted backward shift that is an operator on $X$. In the classical situation of $V=\NN_0$ or $\ZZ$ it is well known that a single non-trivial periodic point for $B_\lambda$ makes this operator chaotic, provided that $(e_v)_{v\in V}$ is an unconditional basis in $X$, see \cite{Gro00}, \cite[Theorems 4.8, 4.12]{GrPe11}.  

Before attacking the problem of characterizing chaotic weighted backward shifts on arbitrary trees, let us first show that a single non-trivial periodic point no longer implies chaos. In fact, both parts of the definition of chaos may fail.

\begin{example}\label{ex-twobranched}
There is a rooted tree $V$ and a weighted backward shift $B_\lambda$ on $\ell^2(V)$ that has a non-trivial fixed point so that $B_\lambda$ is not hypercyclic and does not have a dense set of periodic points.   

Let $V$ be the rooted tree consisting of exactly two infinite branches starting from the root $v_0$, that is, the root has exactly two children, $v_1$ and $v_2$ say, and any other vertex has exactly one child. Let $\lambda$ be a weight so that $v_1$ and each of its descendants has weight 2, while $v_2$ and each of its descendants has weight 1. By \cite{GrPa21}, $B_\lambda$ is an operator on $\ell^2(V)$, and we have that
\[
f(v)=\begin{cases}
1, & \text{if }  v = v_0,\\
\frac{1}{2^{n+1}}, & \text{if }  \prt^n(v)=v_1, n\geq 0,\\
0, & \text{otherwise}
\end{cases}    
\]
is a non-trivial fixed point of $B_\lambda$. 

On the other hand, it follows from \cite[Theorem 4.4]{GrPa21} that $B_\lambda$ is not hypercyclic; alternatively just observe that, on the branch defined by $v_2$, $B_\lambda$ behaves like the unweighted backward shift on $\ell^2(\NN_0)$, which is not hypercyclic.

In addition, if the set of periodic points for $B_\lambda$ were dense then there would exist a periodic point $f$, of period $N$ say, such that $f(v_2)\neq 0$. But this implies that $f(v)=f(v_2)$ whenever $v\in \Chi^{nN}(v_2)$, $n\geq 1$, which prevents $f$ from being in $\ell^2(V)$, a contradiction.
\end{example}

The next example shows that even hypercyclicity together with the existence of a non-trivial periodic point does not imply chaos.

\begin{example}
Let $V$ be the two-branched rooted tree of the previous example with root $v_0$ and $\Chi(v_0)=\{v_1,v_2\}$. Let again $\lambda_{v}=2$ when $v$ is $v_1$ or one of its descendants, but now we put on the other branch a hypercyclic, non-chaotic unilateral weighted backward shift, for example by setting
\[
\lambda_{v}=\Big( \frac{n+2}{n+1}\Big)^{1/2}
\]
if $\prt^n(v)=v_2$, $n\geq 0$. One sees with the same arguments as in the previous example that $B_\lambda$ is an operator on $\ell^2(V)$ with a non-trivial fixed point (in fact, the same one as before) that is now hypercyclic by \cite[Theorem 4.4]{GrPa21} but still has no dense set of periodic points.
\end{example}

So, why does the presence of a non-trivial periodic point in the classical case of $\NN_0$ imply chaos? The point is that if a sequence $x=(x_n)_{n\geq 0}$ in $\ell^2(\NN_0)$, say, satisfies $B_\lambda^N x=x$, $N\geq 1$, and $x_j\neq 0$ for some $j\geq 0$ then $y=\sum_{k\in\ZZ, j+kN\geq 0} x_{j+kN} e_{j+kN}$ is also a periodic point, and $z=y+B_\lambda y +\ldots+B_\lambda^{N-1} y$ becomes a fixed point for $B_\lambda$ with all entries non-zero. From this fixed point one can easily construct a dense set of periodic points, so $z$ acts like a ``universal'' periodic point. Also, the presence of $z$ in $\ell^2(\NN_0)$ turns out to make $B_\lambda$ hypercyclic. See the proof of \cite[Theorem 4.6]{GrPe11}.

This motivates the following.

\begin{definition}
Let $V$ be a tree, $X$ a Fr\'echet sequence space over $V$, and $B_\lambda$ a weighted backward shift operator on $X$. A fixed point $f\in X$ for $B_\lambda$ is called \textit{universal} if $f(v)\neq 0$ for all $v\in V$.
\end{definition}

What distinguishes $V=\NN_0$ from more general trees, even from $\ZZ$, is also the fact that, for any vertex $v$, the set $V(v)$ of descendants of $v$ is cofinite in $\NN_0$, and thus creates no new behaviour. For arbitrary trees we need to consider $V(v)$ for all vertices $v$. Thus, let $X$ be a Fr\'echet sequence space over $V$ in which $(e_v)_{v\in V}$ is an unconditional basis. For a subtree $W$ of $V$, we write $X(W)$ for the space of restrictions of sequences $f\in X$ to $W$, endowed with the topology induced by $X$. Then $B_\lambda$ naturally defines an operator $B_\lambda : X(W)\to X(W)$. Note, however, that if $W$ has a root $w_0$ then $e_{w_0}$ is sent to 0.

\begin{definition}\label{d-fpover}
Let $V$ be a tree, $X$ a Fr\'echet sequence space over $V$ in which $(e_v)_{v\in V}$ is an unconditional basis,  $B_\lambda$ a weighted backward shift operator on $X$, and $W$ a subtree of $V$. Then a sequence $f\in X(W)$ that is a fixed point for $B_\lambda:X(W)\to X(W)$ is called \textit{a fixed point for $B_\lambda$ in $X$ over $W$}.
\end{definition}

The examples above show why the arguments for weighted backward shifts on $\NN_0$ break down on more general trees: one may encounter a periodic point that is zero along a whole branch, which then prevents the construction of a universal fixed point. In the case of rooted trees, the existence of a universal fixed point implies chaos. More generally, we have the following, which is the main result of this section.

\begin{theorem}\label{chaos-fixedpoint}
Let $V$ be a  tree and $\lambda=(\lambda_v)_{v\in V}$ a weight. Let $X$ be a Fr\'echet sequence space over $V$ in which $(e_v)_{v\in V}$ is an unconditional basis, and suppose that the weighted backward shift $B_\lambda$ is an operator on $X$. Consider the following assertions:
\begin{enumerate}
    \item[\rm (i)] for any $v\in V$ there is a fixed point $f$ for $B_\lambda$ in $X$ over $V(v)$ with $f(v)=1$, and 
    \begin{equation}\label{eq-fpleft}
    \sum_{n\geq 1} \lambda(\prt^n(v)\to v)e_{\prt^{n}(v)}\in X,
    \end{equation}
		where the sum extends over the $n\geq 1$ for which $\prt^n(v)$ is defined;
    \item[\rm (ii\textsubscript{a})] $B_{\lambda}$ is chaotic;
    \item[\rm (ii\textsubscript{b})] $\overline{\per(B_{\lambda})}=X$;
		\item[\rm (iii\textsubscript{a})] for any $v\in V$ there is some $f\in \per(B_\lambda)$ such that $f(v)=1$;
		\item[\rm (iii\textsubscript{b})] for any $v\in V$ there is a fixed point $f\in X$ for $B_\lambda$ such that $f(v)=1$;
		\item[\rm (iii\textsubscript{c})] $B_{\lambda}$ has a universal fixed point in $X$.
\end{enumerate}
Then
\[
(\emph{\text{i}}) \Longrightarrow
\big[(\emph{\text{ii\textsubscript{a}}}) \Longleftrightarrow
(\emph{\text{ii\textsubscript{b}}})\big] \Longrightarrow
\big[(\emph{\text{iii\textsubscript{a}}}) \Longleftrightarrow
(\emph{\text{iii\textsubscript{b}}}) \Longleftrightarrow
(\emph{\text{iii\textsubscript{c}}})\big].
\] 
If the tree is rooted or has a free left end, then all the assertions are equivalent.

Moreover, if $B_{\lambda}$ is chaotic then it is mixing.
\end{theorem}

\begin{proof}
(i) $\Longrightarrow$ (ii\textsubscript{b}). In view of the fact that $\per(B_{\lambda})$ is a linear subspace, it suffices by the basis assumption to approximate each $e_v$, $v\in V$, by a periodic point. Thus let $v\in V$ and $\varepsilon >0$. By (i) there is a fixed point $f$ for $B_\lambda$ in $X$ over $V(v)$ with $f(v)=1$. We can extend $f$ to all of $V$ by defining
$f(\prt^n(v))=\lambda(\prt^n(v)\to v)$ for any $n\geq 1$ for which $\prt^n(v)$ is defined, and $f(u)=0$ for all other vertices $u$ outside $V(v)$. By the hypothesis, $f$ belongs to $X$, and $f$ is clearly a fixed point for $B_\lambda$ (over $V$) with $f(v)=1$.

Let us enumerate the generations in the tree in such a way that $v\in \gen_0$. Let $\|\cdot\|$ denote an F-norm defining the topology of $X$. Then, by unconditionality of the basis, we can choose $N\geq 1$ such that
\[
\Big\| \sum_{n\in \ZZ\setminus\{0\}} f\chi_{\gen_{nN}}\Big\|<\varepsilon,
\]
where $\gen_n$ will be empty for large negative $n$ if the tree is rooted. Then
\[
g:=\sum_{n\in \ZZ} f\chi_{\gen_{nN}}
\]
is a periodic point for $B_\lambda$ with $\|g-e_v\|<\varepsilon$, as desired.

(ii\textsubscript{a}) $\Longleftrightarrow$ (ii\textsubscript{b}). It suffices to show that, under (ii\textsubscript{b}), $B_\lambda$ is hypercyclic. Using a recent result of Grivaux, Matheron and Menet we will even show that it is mixing, thereby also proving the final assertion of the theorem. In fact, by \cite[Corollary 5.4]{GMM21}, which also holds for separable Fr\'echet spaces, it suffices in view of (ii\textsubscript{b}) to show that, for every $v\in V$ and $\varepsilon>0$, there is some $m\geq 1$ such that, for all $n\geq m$, there is some $g\in X$ such that $\|g\|<\varepsilon$ and $\|B_\lambda^ng-e_v\|<\varepsilon$. Thus, let $v\in V$, and let $f$ be a periodic point for $B_\lambda$ with $f(v)=1$, which exists by (ii\textsubscript{b}); let $N\geq 1$ be a period of $f$. For $n\geq 1$ there are $l\geq 1$ and
$0\leq k\leq N-1$ with $n=lN-k$, and we define
\[
g_{n}:= B_\lambda^k (f\chi_{\Chi^{lN}(v)}).
\]
Then, by the basis assumption and since $f\in X$, we have that $g_n\in X$ and
\[
B_\lambda^n g_n = B_\lambda^{lN-k}B_\lambda^k \big(f\chi_{\Chi^{lN}(v)}\big)=e_v.
\]
Moreover, since $B_\lambda$ is continuous and $k$ ranges over a finite set, we have that $g_n\to 0$ as $n\to\infty$, which gives us what was needed.

The implication (ii\textsubscript{b}) $\Longrightarrow$ (iii\textsubscript{a}) is obvious. 

(iii\textsubscript{a}) $\Longrightarrow$ (iii\textsubscript{b}). Let $v\in V$, and let $f\in X$ be a periodic point of $B_\lambda$, of period $N\geq 1$ say, such that $f(v)=1$. Having enumerated the generations of the tree in such a way that $v\in \gen_0$, we consider
\[
g:=\sum_{n\in \ZZ} f\chi_{\gen_{nN}},
\]
where many generations will be empty if the tree is rooted. Then $g\in X$ by the basis assumption, and $g$ is periodic for $B_\lambda$ of period $N$, which implies that
\[
h:=g+B_\lambda g +\ldots + B_\lambda^{N-1} g \in X
\]
is a fixed point for $B_\lambda$. In addition,
\[
h(v)=g(v)=f(v)=1.
\]

(iii\textsubscript{b}) $\Longrightarrow$ (iii\textsubscript{c}). We first enumerate the elements of $V$ in some way, $V=\{v_0, v_1,\dots \}$. By (iii\textsubscript{b}) we can find a fixed point $f_0\in X$ for $B_{\lambda}$ such that $f_0(v_0)=1$. By rescaling we may assume that $f_0(v_0)\neq 0$ and 
\[
\|f_0\| <1.
\]
Proceeding inductively, for each $n\geq 0$, we can find a fixed point $f_n\in X$ for $B_{\lambda}$ such that $f_n(v_n)\neq 0$,
\[
\|f_n\| < \frac{1}{2^n},
\] 
\[
|f_n(v_k)| <  \frac{1}{2^{n-k+1}}|(f_0+\dots +f_k)(v_k)| \ \text{ for } k=0, \dots ,n-1,
\]
and
\[
f_n(v_n)\neq -(f_0+\dots +f_{n-1})(v_n);
\]
this last inequality also ensures that the next step in the induction process works.

We then set
\[
f=\sum_{n=0}^{\infty}f_n,
\]
which converges in $X$, and we observe that $f$ is a fixed point for $B_{\lambda}$. Now, let $v\in V$. Then $v=v_k$ for some $k\geq 0$, and  we have that
\begin{align*}
|f(v)|&=|f(v_k)|=\Big|\sum_{n=0}^{\infty}f_n(v_k)\Big| \geq  \Big|\sum_{n=0}^{k}f_n(v_k)\Big|-\sum_{n=k+1}^{\infty}|f_n(v_k)| \\
&\geq \Big|\sum_{n=0}^{k}f_n(v_k)\Big|-\Big(\sum_{n=k+1}^{\infty}\frac{1}{2^{n-k+1}}\Big)\Big|\sum_{n=0}^{k}f_n(v_k)\Big|=\frac{1}{2}\Big|\sum_{n=0}^{k}f_n(v_k)\Big|>0,
\end{align*}
proving the claim.

(iii\textsubscript{c}) $\Longrightarrow$ (iii\textsubscript{a}) is trivial.

To finish the proof we have to show that (iii\textsubscript{b}) $\Longrightarrow$ (i) if the tree is rooted or has a free left end. This is trivial for the first assertion in (i). Moreover, if the tree is rooted then the series in \eqref{eq-fpleft} is a finite sum and therefore belongs to $X$. On the other hand, if the tree has a free left end, then choose $v\in V$ such that $\gen(v)$ is a singleton. By (iii\textsubscript{b}), there is a fixed point $f\in X$ for $B_\lambda$ such that $f(v)=1$. But then $f(\prt^n(v))=\lambda(\prt^n(v)\to v)$ for all $n\geq 1$, so that \eqref{eq-fpleft} holds by the basis assumption since $f\in X$.
\end{proof}

\begin{remark}\label{rem-GMM}
(a) For the proof that chaotic weighted backward shifts are mixing we have used a very powerful and elegant general result of Grivaux, Matheron and Menet \cite{GMM21}. However, we could have just as well deduced the mixing property from Theorems \ref{t-charHCroot} and \ref{t-charHCunroot} in Appendix 1, which are specific to weighted backward shifts. Indeed, let $v\in V$. There is then a fixed point $f\in X$ for $B_\lambda$ with $f(v)=1$. Consider, for $n\geq 1$, 
\[
f_{v,n}=f\chi_{\Chi^{n}(v)}\in X.
\]
Then the hypothesis of (b)(iii) in Theorem \ref{t-charHCroot} and one half of the hypothesis of (b)(iii) in Theorem \ref{t-charHCunroot} are satisfied. In the unrooted case, let also $\varepsilon>0$, and choose a sequence $g\in X$ with $\|g\|<\varepsilon$ such that $e_v+g$ is a periodic point for $B_\lambda$. Since the sequence $(B^n_\lambda (e_v+g))_n$ only runs through a finite number of elements of $X$, we have that
\[
(B_\lambda^n (e_v+g))\chi_{\{\prt^n(v)\}} \to 0 \text{ in } X.
\]
A diagonal process produces a sequence $(g_{v,n})_n$ such that $g_{v,n}\to 0$ in $X$ and $(B_\lambda^n (e_v+g_{v,n})\chi_{\{\prt^n(v)\}}\to 0$ in $X$. It remains to restrict $g_{v,n}$ to $\Chi^n(\prt^n(v))$ to see that also the second half of the hypothesis of (b)(iii) in Theorem \ref{t-charHCunroot} is satisfied.

In both cases we therefore have that $B_\lambda$ is mixing.

(b) We note that, in (i), it suffices to have condition \eqref{eq-fpleft} only for some $v\in V$. This is due to the fact that, for any two vertices, the series in \eqref{eq-fpleft} have a common tail up to some multiplicative constant, see \eqref{eq-lambda}.

(c) The theorem recovers the known characterizations in the case of the classical trees $\NN_0$ and $\ZZ$, see \cite[Theorems 4.8, 4.13]{GrPe11}.

(d) Let us mention for clarity that none of the conditions in the theorem can hold if the tree has a leaf. Indeed, by definition of $B_\lambda$, if $f$ is a periodic point for $B_\lambda$ then $f(v)=0$ for any leaf $v$.
\end{remark}

We next show that, in general, one cannot improve Theorem \ref{chaos-fixedpoint} in that there are trees in which the conditions (iii) do not imply the conditions (ii), see Example \ref{ex-comb}, and there are trees in which the conditions (ii) do not imply (i), see Example \ref{ex-comb2}. In fact, the tree is the same, and it will serve us at various places in the paper as a counter-example.

\begin{example}[Comb tree]\label{ex-comb}
There is a tree $V$, a Fr\'echet sequence space $X$ over $V$  in which $(e_v)_{v\in V}$ is an unconditional basis, and a weighted backward shift $B_\lambda$ on $X$ for which every vertex $v\in V$ admits a fixed point $f\in X$ with $f(v)=1$ but so that $B_\lambda$ is not chaotic. Note that $B_\lambda$ admits a universal fixed point by Theorem \ref{chaos-fixedpoint}.

To see this, we consider the set $\ZZ$ of integers, from every point of which emanates an infinite branch. In other words, $V=\ZZ\times\NN_0$ so that every vertex $(n,0)$ has two children, $(n+1,0)$ and $(n+1,1)$, while any $(n,k)\in V$, $k\geq 1$, has only $(n+1,k+1)$ as a child. This tree looks like an infinite comb (with inclined teeth in order to align the generations). As for the space, we choose $\mu_{(n,k)}= 1$ if $n\leq 0$, $k\geq 0$, and $\mu_{(n,k)}= \frac{1}{2^n}$ if $n> 0$, $k\geq 0$, and we consider $X=\ell^1(V,\mu)=\{ f\in \KK^V : \sum_{v\in V}|f(v)|\mu_v<\infty\}$. Then the unweighted backward shift $B$ is an operator on $X$, see \cite[Proposition 2.3]{GrPa21}. 

It is easy to find a rich supply of fixed points. Just fix $m\in\ZZ$, and consider $f_m\in X$ defined by
\[
f_m(n,k)= \begin{cases} 
1,&\text{if $n\geq m$, $k=0$,}\\
-1,&\text{if $n\geq m$, $k=n-m+1$,}\\
0,&\text{otherwise.}
\end{cases}
\]
Then for every $v\in V$ there is some $m$ such that $f_m(v)\neq 0$. It remains to normalize $f_m$.

On the other hand, let $f\in X$ be a periodic point for $B$, of period $N\geq 1$, say. Then
\[
f(-N,0) = (B^N f)(-N,0) = \sum_{k=0}^N f(0,k),
\]                                                                                                           
and therefore
\begin{align*}
\|f-e_{(0,0)}\| &\geq |f(-N,0)| + |f(0,0)-1| + \sum_{k=1}^\infty |f(0,k)|\\
&\geq |f(0,0)| - \Big|\sum_{k=1}^N f(0,k)\Big| + 1 - |f(0,0)| + \sum_{k=1}^N |f(0,k)|\geq 1,
\end{align*}
so that the periodic points for $B$ cannot be dense in $X$, and $B$ is not chaotic.
\end{example}

\begin{example}\label{ex-comb2}
There is a tree $V$, a Fr\'echet sequence space $X$ over $V$  in which $(e_v)_{v\in V}$ is an unconditional basis, and a chaotic weighted backward shift $B_\lambda$ on $X$ for which \eqref{eq-fpleft} does not hold.

To see this we consider again the comb tree of the previous example, but this time we take $X=\ell^1(V,\mu)$ with $\mu_{(n,k)}= \frac{1}{2^k}$ if $n\in \ZZ$, $k\geq 0$. The unweighted backward shift $B$ is an operator on $X$. To show that it is chaotic, let $v=(n_0,k_0)\in V$, $N> k_0$, and define
\[
f_N(n,k)= \begin{cases} 
1,&\text{if $n=n_0+lN$, $k=k_0+lN$, $l\geq 0$,}\\
-1,&\text{if $n=n_0+lN$, $k=N+lN$, $l\geq 0$,}\\
0,&\text{otherwise.}
\end{cases}
\]
Then $f_N\in X$ is a periodic point of $B$ of period $N$ with $\|f_N-e_v\|\to 0$ as $N\to\infty$, so that $B$ is chaotic. On the other hand, $\sum_{n=-\infty}^{-1} e_{(n,0)} \notin X$, so that \eqref{eq-fpleft} does not hold for $v=(0,0)$.

Let us note for the sequel that, while $B$ has a universal fixed point since it is chaotic, it cannot possess a positive universal fixed point $f$; otherwise we would have that $f(-n,0)\geq f(0,0)>0$ for all $n\geq 1$, which contradicts the fact that $f\in X$. 
\end{example}

Nonetheless, for certain weighted backward shifts on certain unrooted trees without a free left end, the conditions (iii) do imply the conditions (ii), see Theorem \ref{t-chunrootedsimplified}. The equivalence of the conditions (ii) with (i) seems to be an even more delicate matter, see Remark \ref{rem-sym}(c).

When the weight of a chaotic weighted backward shift is positive one might wonder whether a universal fixed point can also be positive. This is indeed the case when the tree is rooted or has a free left end, but not in general, as we have seen in the previous example. Conversely, while by Example \ref{ex-comb} the existence of a universal fixed point for the unweighted backward shift does not necessarily imply chaos, the existence of a positive universal fixed point does. In fact we have the following.

\begin{theorem}\label{chaos-positive}
Let $V$ be a  tree and $\lambda=(\lambda_v)_{v\in V}$ a positive weight. Let $X$ be a Fr\'echet sequence space over $V$ in which $(e_v)_{v\in V}$ is an unconditional basis, and suppose that the weighted backward shift $B_\lambda$ be an operator on $X$. Consider the following assertions:
\begin{enumerate}
    \item[\rm (i\textsubscript{a})] $B_\lambda$ has a positive universal fixed point in $X$;
		\item[\rm (i\textsubscript{b})] for any $v\in V$ there is a positive universal fixed point for $B_\lambda$ in $X$ over $V(v)$ and
          \begin{equation}\label{eq-univpos}
          \sum_{n\geq 1} \lambda(\prt^n(v)\to v)e_{\prt^{n}(v)}\in X,
          \end{equation}
		      where the sum extends over the $n\geq 1$ for which $\prt^n(v)$ is defined;
    \item[\rm (ii)] $B_\lambda$ is chaotic;
    \item[\rm (iii)] $B_\lambda$ has a universal fixed point in $X$.
\end{enumerate}
Then
\[
[\emph{\text{(i\textsubscript{a})}} \Longleftrightarrow \emph{\text{(i\textsubscript{b})}}] \Longrightarrow \emph{\text{(ii)}} \Longrightarrow \emph{\text{(iii)}}.
\] 
If the tree is rooted or has a free left end, then the three assertions are equivalent.
\end{theorem}

\begin{proof} 
(i\textsubscript{a}) $\Longrightarrow$ (i\textsubscript{b}). Let $f\in X$ be a positive universal fixed point for $B_\lambda$, and let $v\in V$. It suffices to show 
\eqref{eq-univpos}. But since
\[
f(\prt^n(v)) = \sum_{u\in\Chi^n(\prt^n(v))} \lambda(\prt^n(v)\to u) f(u) \geq  \lambda(\prt^n(v)\to v) f(v)
\]
whenever $\prt^n(v)$ is defined, this follows from the unconditionality of the basis and the fact that $f\in X$ is strictly positive.

(i\textsubscript{b}) $\Longrightarrow$ (i\textsubscript{a}). For any $v\in V$, let $f_{v,+}$ be a positive universal fixed point for $B_\lambda$ in $X$ over $V(v)$. We may assume that $f_{v,+}(v)=1$, and we extend $f_{v,+}$ to $V$ by setting it zero outside $V(v)$. Then, by \eqref{eq-univpos},
\[
f_v:=f_{v,+}+ \sum_{n\geq 1} \lambda(\prt^n(v)\to v) e_{\prt^{n}(v)}\in X,
\]
and $f_v$ is a positive fixed point of $B_\lambda$ with $f_v(v)=1$. Choosing now numbers $\alpha_v>0$, $v\in V$, such that $f:=\sum_{v\in V}\alpha_vf_v$ converges in $X$, we have that $f$ is the desired positive universal fixed point for $B_\lambda$.

The implications (i\textsubscript{b}) $\Longrightarrow$ (ii) and (ii) $\Longrightarrow$ (iii) are given by Theorem \ref{chaos-fixedpoint}.

It remains to show that (iii) implies (i\textsubscript{a}) or (i\textsubscript{b}) when $V$ is a rooted tree or a tree with a free left end. We first consider the rooted case. Let $f\in X$ be a universal fixed point for $B_\lambda$. We define for $v\in V$,
\[
\alpha_v = \frac{|f(v)|}{\sum_{u\in \Chi(v)}\lambda_u|f(u)|},
\]
which is well-defined and strictly positive by the definition of the operator $B_\lambda$ and by the assumption. Also, since $f$ is a fixed point for $B_\lambda$, $|\alpha_v|\leq 1$. Now define for any $v\in V$
\[
g(v)= \Big(\prod_{n=1}^m \alpha_{\prt^n(v)}\Big)|f(v)|,
\]
where $m\geq 0$ is such that $v\in \gen_m$. It follows that, for all $v\in V$,
\[
0<g(v)\leq |f(v)|,
\]
which implies that $g\in X$, and
\begin{align*}
\sum_{u\in\Chi(v)} \lambda_u g(u) &= \Big(\prod_{n=0}^m \alpha_{\prt^n(v)}\Big)\sum_{u\in\Chi(v)} \lambda_u |f(u)|\\
 &= \Big(\prod_{n=0}^m \alpha_{\prt^n(v)}\Big)\frac{|f(v)|}{\alpha_v}= g(v),
\end{align*}
so that $g$ is a strictly positive fixed point of $B_\lambda$, giving us (i\textsubscript{a}).

In the unrooted case we have a free left end. Then the previous argument gives the first assertion in (i\textsubscript{b}), while the second one follows from Theorem \ref{chaos-fixedpoint}.
\end{proof}

As already announced, the converses of the two simple implications in the theorem both fail for arbitrary trees, as follows from Examples \ref{ex-comb} and \ref{ex-comb2}.

Incidentally, as noted in Remark \ref{rem-GMM}(b), it suffices to have \eqref{eq-univpos} only for some $v\in V$. But it is not enough to have all the condition (i\textsubscript{b}) only for some $v\in V$, as witnessed by Example \ref{ex-twobranched}.

We finish this section with a natural sufficient condition for chaos in terms of the weights. In the classical case of $V=\NN_0$, there is up to a multiplicative constant only one universal fixed point for $B_\lambda$, namely
\[
\sum_{n=1}^\infty \frac{1}{\prod_{k=1}^n \lambda_k} e_n.
\]
Thus, $B_\lambda$ is chaotic if and only if this element belongs to $X$. A similar statement holds on the tree $V=\ZZ$. See also \cite[Theorems 4.8 and 4.13]{GrPe11}. This motivates the following.

\begin{corollary}
Let $V$ be a leafless tree and $\lambda=(\lambda_v)_{v\in V}$ a weight. Let $X$ be a Fr\'echet sequence space over $V$ in which $(e_v)_{v\in V}$ is an unconditional basis, and suppose that the weighted backward shift $B_\lambda$ is an operator on $X$. If, for any $v\in V$, there is a branch $(v_0, v_1,v_2,\ldots)$ starting from $v_0:=v$ such that
\[
 \sum_{n\geq 1} \lambda(\prt^n(v)\to v)e_{\prt^{n}(v)} + \sum_{n=0}^\infty \frac{1}{\lambda(v\to v_n)} e_{v_n} \in X,
\]
where the first sum extends over the $n\geq 1$ for which $\prt^n(v)$ is defined, then $B_\lambda$ is chaotic.
\end{corollary}

This follows from the fact that $\sum_{n=0}^\infty \frac{1}{\lambda(v\to v_n)} e_{v_n}$ is a fixed point for $B_\lambda$ in $X$ over $V(v)$ with value 1 at $v$. Thus condition (i) in Theorem \ref{chaos-fixedpoint} is satisfies.

It is not surprising that the condition in the corollary is in general not necessary.

\begin{example}\label{non-uniform}
Let $V$ be the binary rooted tree, that is, the rooted tree in which every vertex has exactly two children. Then the unweighted backward shift $B$ is an operator on $\ell^2(V)$ which has a universal fixed point $f$ in the space and is therefore chaotic; indeed, it suffices to let $f(v)=\frac{1}{2^n}$ if $v\in \gen_n$, $n\geq 0$. But there is no branch $(v_0, v_1,v_2,\ldots)$ for which $\sum_{n=0}^\infty e_{v_n}$ belongs to $X$.
\end{example}

\section{Chaos on unrooted trees}\label{s-unrooted}
It is a well-known phenomenon that the dynamical behaviour of bilateral weighted backward shifts on the tree $\mathbb{Z}$ is, in general, considerably more complicated than that for unilateral weighted backward shifts on the tree $\mathbb{N}_0$. 

On the other hand, chaotic bilateral weighted backward shifts are as easy to treat as unilateral ones, see \cite[Theorems 4.6, 4.12]{GrPe11}. This is due to the fact that the tree $\ZZ$ has a free left end. However, when we do not have a free left end, Theorem \ref{chaos-fixedpoint} and the subsequent examples of the comb tree show that chaos on unrooted trees will be a more difficult notion than chaos on rooted trees. It may therefore be surprising that, even then, the unrooted case can be reduced to chaos on rooted trees. This reduction, which is nonetheless rather technical, comes form a simple but very powerful observation.

We begin with the unweighted backward shift $B$, and we illustrate the basic idea in the special case of fixed points. Suppose that we search for a fixed point $f$ that is non-zero at a given vertex $v$. This involves conditions ``to the left'' and ``to the right'' of $v$. The latter simply means that the restriction $f_+$ of $f$ to the subtree
\[
V_+(v)=V(v) = \bigcup_{n\geq 0}\Chi^n(v)
\]
of descendants of $v$ is a fixed point on it, and $V(v)$ is a rooted tree with root $v$. 

For the other part, let us write
\[
v_{-n}=\prt^n(v), \quad \text{$n\geq 0$},
\]
in particular $v_0=v$. Then the remaining defining conditions for the fixed point $f$ are
\begin{equation}\label{eq-fp}
f(v_{-n})=f(v_{-n+1}) + \sum_{\substack{u\in\Chi(v_{-n})\\u\neq v_{-n+1}}} f(u),\quad n\geq 1,
\end{equation}
and, for any vertex $w$ that is neither one of the $v_{-n}$, $n\geq 0$, nor a descendant of $v$,
\begin{equation}\label{eq-inv1}
f(w)=\sum_{u\in\Chi(w)} f(u).
\end{equation}
Now, all we have to do is to write \eqref{eq-fp} as 
\begin{equation}\label{eq-inv2}
f(v_{-n+1})= f(v_{-n})+ \sum_{\substack{u\in\Chi(v_{-n})\\u\neq v_{-n+1}}} (-f(u)),\quad n\geq 1.
\end{equation}
Then \eqref{eq-inv1} and \eqref{eq-inv2} say exactly that if we consider
\[
V_-(v) = V\setminus \bigcup_{n\geq 1}\Chi^n(v)= \big(V\setminus V(v)\big) \cup\{v\},
\]
then the sequence $f_-$ on $V_-(v)$ given by
\begin{equation}\label{f-}
f_-(w) =
\begin{cases} 
f(w),&\text{if } w=v_{-n} \text{ for some $n\geq 0$},\\
-f(w),&\text{otherwise},
\end{cases}
\end{equation}
is a fixed point for the backward shift on the rooted tree $V_-(v)$, when we need to define a new parent-child relationship on $V_-(v)$. Indeed, for $n\geq 0$, the children of $v_{-n}$ are $v_{-n-1}$ and all the children of $v_{-n-1}$, \textit{except} $v_{-n}$, while all other relationships remain unchanged. That is, for any $w\in V_-(v)$,
\[
\Chi(w) =\begin{cases} 
\{v_{-n-1}\}\cup \Chi_V(v_{-n-1})\setminus \{v_{-n}\},&\text{if } w=v_{-n} \text{ for some $n\geq 0$},\\
\Chi_V(w),&\text{otherwise},
\end{cases}
\]
where $\Chi_V$ denotes the child mapping in $V$. See Figures \ref{fig-v_1a} and \ref{fig-v_1b}, where $W_{-n} = V(v_{-n})\setminus V(v_{-n+1})$ for $n\geq 1$.

\begin{figure}
\centering
\begin{minipage}[t]{.5\textwidth}
\centering
\begin{tikzpicture}[scale=1.5]
\draw[fill] (0,0) circle (.5pt);

\draw[->,>=latex] (-.5,0.25) -- (0,0);

\draw[->,>=latex] (0,0) -- (1,-.5);\draw[fill] (1,-.5) circle (.5pt);

\draw[->,>=latex] (0,0) -- (1,.6);\draw[fill] (1,.6) circle (.5pt);
\draw[->,>=latex] (0,0) -- (1,.8);\draw[fill] (1,.8) circle (.5pt);

\draw[->,>=latex] (1,-.5) -- (2,.1);\draw[fill] (2,.1) circle (.5pt);
\draw[->,>=latex] (1,-.5) -- (2,-.05);\draw[fill] (2,-.05) circle (.5pt);
\draw[->,>=latex] (1,-.5) -- (2,-.2);\draw[fill] (2,-.2) circle (.5pt);

\draw[->,>=latex] (1,.8) -- (2,1.2);\draw[fill] (2,1.2) circle (.5pt);
\draw[->,>=latex] (1,.6) -- (2,1);\draw[fill] (2,1) circle (.5pt);
\draw[->,>=latex] (1,.6) -- (2,0.8);\draw[fill] (2,0.8) circle (.5pt);

\draw[->,>=latex] (2,.8) -- (3,0.9);\draw[fill] (3,0.9) circle (.5pt);
\draw[->,>=latex] (2,1) -- (3,1.1);\draw[fill] (3,1.1) circle (.5pt);

\draw[->,>=latex] (2,1.2) -- (3,1.3);\draw[fill] (3,1.3) circle (.5pt);
\draw[->,>=latex] (2,1.2) -- (3,1.5);\draw[fill] (3,1.5) circle (.5pt);

\draw[->,>=latex] (2,.1) -- (3,.5);\draw[fill] (3,.5) circle (.5pt);
\draw[->,>=latex] (2,.1) -- (3,.35);\draw[fill] (3,.35) circle (.5pt);

\draw[->,>=latex] (2,-.05) -- (3,.2);\draw[fill] (3,.2) circle (.5pt);

\draw[->,>=latex] (2,-.2) -- (3,.05);\draw[fill] (3,.05) circle (.5pt);
\draw[->,>=latex] (2,-.2) -- (3,-.1);\draw[fill] (3,-.1) circle (.5pt);

\draw[->,>=latex] (1,-.5) -- (2,-1);\draw[fill] (2,-1) circle (.5pt);

\draw[->,>=latex] (2,-1) -- (3,-1.5);\draw[fill] (3,-1.5) circle (.5pt);

\draw[->,>=latex] (2,-1) -- (3,-1.3);\draw[fill] (3,-1.3) circle (.5pt);
\draw[->,>=latex] (2,-1) -- (3,-1.1);\draw[fill] (3,-1.1) circle (.5pt);
\draw[->,>=latex] (2,-1) -- (3,-.9);\draw[fill] (3,-.9) circle (.5pt);
\draw[->,>=latex] (2,-1) -- (3,-.7);\draw[fill] (3,-.7) circle (.5pt);
\draw[->,>=latex] (2,-1) -- (3,-.5);\draw[fill] (3,-.5) circle (.5pt);

\node at (3.5,1.3) {\footnotesize{$W_{-2}$}};
\node at (3.5,0.3) {\footnotesize{$W_{-1}$}};
\node at (3.5,-1.0) {\footnotesize{$V(v)$}};

\node at (.0,-.2) {\footnotesize{$v_{-2}$}};
\node at (1,-.7) {\footnotesize{$v_{-1}$}};
\node at (1.7,-1.2) {\footnotesize{$v_{0}=v$}};

\end{tikzpicture}
\caption{The tree $V$ and the subtree $V(v)$ (detail)}%
\label{fig-v_1a}
\end{minipage}%
\begin{minipage}[t]{.5\textwidth}
\begin{tikzpicture}[scale=1.5]
\draw[fill] (0,0) circle (.5pt);

\draw[->,>=latex] (0,0) -- (1,-.5);\draw[fill] (1,-.5) circle (.5pt);
\draw[->,>=latex] (0,0) -- (1,.6);\draw[fill] (1,.6) circle (.5pt);
\draw[->,>=latex] (0,0) -- (1,.45);\draw[fill] (1,.45) circle (.5pt);
\draw[->,>=latex] (0,0) -- (1,.3);\draw[fill] (1,.3) circle (.5pt);

\draw[->,>=latex] (1,-.5) -- (2,-1);\draw[fill] (2,-1) circle (.5pt);

\draw[->,>=latex] (1,-.5) -- (2,-.4);\draw[fill] (2,-.4) circle (.5pt);
\draw[->,>=latex] (1,-.5) -- (2,-.2);\draw[fill] (2,-.2) circle (.5pt);

\draw[->,>=latex] (1,.6) -- (2,1.0);\draw[fill] (2,1.0) circle (.5pt);
\draw[->,>=latex] (1,.6) -- (2,.85);\draw[fill] (2,.85) circle (.5pt);
\draw[->,>=latex] (1,.45) -- (2,.7);\draw[fill] (2,.7) circle (.5pt);
\draw[->,>=latex] (1,.3) -- (2,.55);\draw[fill] (2,.55) circle (.5pt);
\draw[->,>=latex] (1,.3) -- (2,.4);\draw[fill] (2,.4) circle (.5pt);

\draw[->,>=latex] (2,-1) -- (3,-1.5);\draw[fill] (3,-1.5) circle (.5pt);

\draw[->,>=latex] (2,-1) -- (3,-1.1);\draw[fill] (3,-1.1) circle (.5pt);
\draw[->,>=latex] (2,-1) -- (3,-1);\draw[fill] (3,-1) circle (.5pt);
\draw[->,>=latex] (2,-1) -- (3,-.9);\draw[fill] (3,-.9) circle (.5pt);

\draw[->,>=latex] (2,-.2) -- (3,-.1);\draw[fill] (3,-.1) circle (.5pt);
\draw[->,>=latex] (2,-.4) -- (3,-.3);\draw[fill] (3,-.3) circle (.5pt);
\draw[->,>=latex] (2,-.4) -- (3,-.5);\draw[fill] (3,-.5) circle (.5pt);

\draw[->,>=latex] (2,1.0) -- (3,1.3);\draw[fill] (3,1.3) circle (.5pt);
\draw[->,>=latex] (2,1.0) -- (3,1.2);\draw[fill] (3,1.2) circle (.5pt);
\draw[->,>=latex] (2,1.0) -- (3,1.1);\draw[fill] (3,1.1) circle (.5pt);

\draw[->,>=latex] (2,.85) -- (3,.95);\draw[fill] (3,.95) circle (.5pt);

\draw[->,>=latex] (2,.7) -- (3,.8);\draw[fill] (3,.8) circle (.5pt);

\draw[->,>=latex] (2,.55) -- (3,.65);\draw[fill] (3,.65) circle (.5pt);

\draw[->,>=latex] (2,.4) -- (3,.3);\draw[fill] (3,.3) circle (.5pt);
\draw[->,>=latex] (2,.4) -- (3,.4);\draw[fill] (3,.4) circle (.5pt);
\draw[->,>=latex] (2,.4) -- (3,.5);\draw[fill] (3,.5) circle (.5pt);

\node at (3.5,.8) {\footnotesize{$W_{-1}$}};
\node at (3.5,-.3) {\footnotesize{$W_{-2}$}};
\node at (3.5,-1) {\footnotesize{$W_{-3}$}};

\node at (-.4,0) {\footnotesize{$v_0=v$}};
\node at (1,-.7) {\footnotesize{$v_{-1}$}};
\node at (2,-1.2) {\footnotesize{$v_{-2}$}};
\node at (3,-1.65) {\footnotesize{$v_{-3}$}};

\end{tikzpicture}
\caption{The tree $V_-(v)$ (detail)}%
\label{fig-v_1b}
\end{minipage}
\end{figure}

However, a problem arises when we add a sequence space $X$ over $V$. We assume that the backward shift $B$ is an operator on $X$. Its restriction to $V(v)$ will remain continuous under the usual assumptions on $X$. On the other hand, the new parent-child relationship in $V_-(v)$ may destroy the continuity of $B$ on $V_-(v)$. Just think of the case when $V=\ZZ$ with weights $\mu_n=1/2^n$ for $n\geq 0$ and $\mu_n=1/|n|!$ if $n<0$. Then $B$ is a chaotic operator on $\ell^1(V,\mu)$, but $V_-(0)$ can be identified with $\NN_0$ and $B$ is not continuous on $\ell^1(\NN_0,(1/n!)_n)$. We therefore introduce the notion of a backward invariant sequence that requires no continuity of $B$. We take this opportunity to define additional notions that will be useful in the sequel.

\begin{definition}\label{d-backwinv}
Let $V$ be a tree and $\lambda=(\lambda_v)_{v\in V}$ a weight. Let $X$ be a Fr\'echet sequence space over V. Then a sequence $f\in \mathbb{K}^V$ is called \textit{$\lambda$-backward invariant} if, for any $v\in V$ that is not a leaf,
\begin{equation}\label{eq-ser}
f(v) = \sum_{u\in \Chi(v)} \lambda_u f(u),
\end{equation}
where the series are supposed to be unconditionally convergent. If $f\in X$ we call it \textit{$\lambda$-backward invariant in $X$ (over $V$)}. 
In the special case where $\lambda_v=1$ for all $v\in V$, we call $f$ \textit{backward invariant}.
\end{definition}

\begin{remark}
We stress two important aspects of this definition. First, we do not assume anything on the corresponding weighted backward shift $B_\lambda$. It need not be an operator on $X$, and in fact the series \eqref{eq-ser} might not even be defined for other sequences $f$. Secondly, unlike for fixed points, equation \eqref{eq-ser} need not apply to leaves $v$, which adds some flexibility that we will need later, see Subsection \ref{s-backinv}. In addition, as in Definition \ref{d-fpover}, the sequence space $X$ might initially be defined over a larger set than $V$.
\end{remark}

Returning to our discussion, we thus see that the sequences $f_+$ and $f_-$ are backward invariant on $V_+(v)=V(v)$ and $V_-(v)$, respectively. 

A gratifying, and in view of the examples on the comb tree also necessary, feature of this construction is that when we repeat it for periodic points of an arbitrary period $N\geq 1$ then the new rooted trees can no longer be obtained from the trees for fixed points.

Indeed, let again $v\in V$ be an arbitrary vertex. We enumerate the generations of $V$ in such a way that $v_0:=v\in\gen_0$. Then we define
\[
V^N_-(v)= \bigcup_{n\in\ZZ} \gen_{nN} \setminus \bigcup_{n\geq 1}\Chi^{nN}(v);
\]
in other words, we keep all vertices whose generation is a multiple of $N$ after deleting all descendants of $v$ (but keeping $v$). In addition, we have to define a new parent-child relationship on $V^N_-(v)$. More specifically, for any $w\in V^N_-(v)$, we define
\begin{equation}\label{eq-childN}
\Chi(w) =\begin{cases} 
\{v_{-nN-N}\}\cup \Chi^N_V(v_{-nN-N})\setminus \{v_{-nN}\},&\text{if } w=v_{-nN} \text{ for some $n\geq 0$},\\
\Chi_V^N(w),&\text{otherwise},
\end{cases}
\end{equation}
where $\Chi_V$ denotes again the child mapping in $V$. See Figure \ref{fig-v_N}, where 
\begin{equation}\label{eq-WnNN}
W_{-nN}^N = \big(V(v_{-nN})\setminus V(v_{-nN+N})\big) \cap \bigcup_{k\in\ZZ} \gen_{kN}, \ n\geq 1. 
\end{equation}

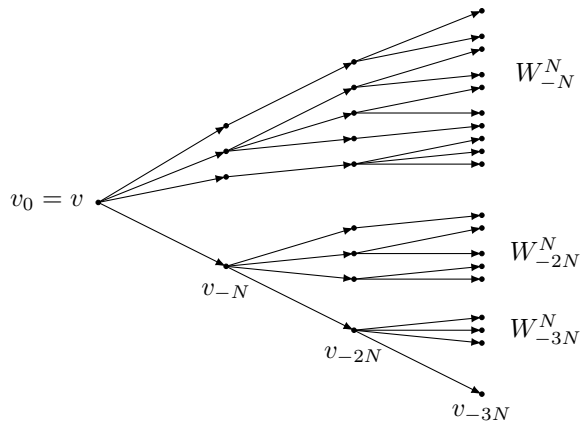
\begin{figure}
\begin{tikzpicture}[scale=1.7]
\draw[fill] (0,0) circle (.5pt);

\draw[->,>=latex] (0,0) -- (1,-.5);\draw[fill] (1,-.5) circle (.5pt);
\draw[->,>=latex] (0,0) -- (1,.6);\draw[fill] (1,.6) circle (.5pt);
\draw[->,>=latex] (0,0) -- (1,.4);\draw[fill] (1,.4) circle (.5pt);
\draw[->,>=latex] (0,0) -- (1,.2);\draw[fill] (1,.2) circle (.5pt);

\draw[->,>=latex] (1,-.5) -- (2,-1);\draw[fill] (2,-1) circle (.5pt);

\draw[->,>=latex] (1,-.5) -- (2,-.6);\draw[fill] (2,-.6) circle (.5pt);
\draw[->,>=latex] (1,-.5) -- (2,-.4);\draw[fill] (2,-.4) circle (.5pt);
\draw[->,>=latex] (1,-.5) -- (2,-.2);\draw[fill] (2,-.2) circle (.5pt);

\draw[->,>=latex] (1,.6) -- (2,1.1);\draw[fill] (2,1.1) circle (.5pt);

\draw[->,>=latex] (1,.4) -- (2,.9);\draw[fill] (2,.9) circle (.5pt);
\draw[->,>=latex] (1,.4) -- (2,.7);\draw[fill] (2,.7) circle (.5pt);
\draw[->,>=latex] (1,.4) -- (2,.5);\draw[fill] (2,.5) circle (.5pt);

\draw[->,>=latex] (1,.2) -- (2,.3);\draw[fill] (2,.3) circle (.5pt);

\draw[->,>=latex] (2,-1) -- (3,-1.5);\draw[fill] (3,-1.5) circle (.5pt);

\draw[->,>=latex] (2,-1) -- (3,-1.1);\draw[fill] (3,-1.1) circle (.5pt);
\draw[->,>=latex] (2,-1) -- (3,-1);\draw[fill] (3,-1) circle (.5pt);
\draw[->,>=latex] (2,-1) -- (3,-.9);\draw[fill] (3,-.9) circle (.5pt);

\draw[->,>=latex] (2,-.2) -- (3,-.1);\draw[fill] (3,-.1) circle (.5pt);
\draw[->,>=latex] (2,-.4) -- (3,-.2);\draw[fill] (3,-.2) circle (.5pt);
\draw[->,>=latex] (2,-.4) -- (3,-.4);\draw[fill] (3,-.4) circle (.5pt);
\draw[->,>=latex] (2,-.6) -- (3,-.5);\draw[fill] (3,-.5) circle (.5pt);
\draw[->,>=latex] (2,-.6) -- (3,-.6);\draw[fill] (3,-.6) circle (.5pt);

\draw[->,>=latex] (2,1.1) -- (3,1.5);\draw[fill] (3,1.5) circle (.5pt);
\draw[->,>=latex] (2,1.1) -- (3,1.3);\draw[fill] (3,1.3) circle (.5pt);

\draw[->,>=latex] (2,.9) -- (3,1.2);\draw[fill] (3,1.2) circle (.5pt);
\draw[->,>=latex] (2,.9) -- (3,1);\draw[fill] (3,1) circle (.5pt);

\draw[->,>=latex] (2,.7) -- (3,.9);\draw[fill] (3,.9) circle (.5pt);
\draw[->,>=latex] (2,.7) -- (3,.7);\draw[fill] (3,.7) circle (.5pt);

\draw[->,>=latex] (2,.5) -- (3,.6);\draw[fill] (3,.6) circle (.5pt);

\draw[->,>=latex] (2,.3) -- (3,.3);\draw[fill] (3,.3) circle (.5pt);
\draw[->,>=latex] (2,.3) -- (3,.4);\draw[fill] (3,.4) circle (.5pt);
\draw[->,>=latex] (2,.3) -- (3,.5);\draw[fill] (3,.5) circle (.5pt);

\node at (3.5,1) {\footnotesize{$W^N_{-N}$}};
\node at (3.5,-.4) {\footnotesize{$W^N_{-2N}$}};
\node at (3.5,-1) {\footnotesize{$W^N_{-3N}$}};

\node at (-.4,0) {\footnotesize{$v_0=v$}};
\node at (1,-.7) {\footnotesize{$v_{-N}$}};
\node at (2,-1.2) {\footnotesize{$v_{-2N}$}};
\node at (3,-1.65) {\footnotesize{$v_{-3N}$}};
\end{tikzpicture}
\caption{The tree $V_-^N(v)$, $N\geq 1$ (detail)}%
\label{fig-v_N}
\end{figure}

In the proof of the following theorem, the corresponding tree $V^N_+(v)$ will also appear, but not in its statement because, just like in the proof of Theorem \ref{chaos-fixedpoint}, any fixed point on $V(v)$  will provide us with suitable periodic points. 

Recall that in our setting, by \cite[Remark 4.1]{GrPa21}, no weighted backward shift can by hypercyclic if the underlying tree has a leaf. We may therefore assume that the tree is leafless. 

In the sequel, $\|\cdot\|$ will denote an F-norm defining the topology of the Fr\'echet sequence space $X$. 

\begin{theorem}\label{t-chunrooted}
Let $V$ be a leafless unrooted tree. Let $X$ be a Fr\'echet sequence space over $V$ in which $(e_v)_{v\in V}$ is an unconditional basis, and suppose that the backward shift $B$ is an operator on $X$. Then $B$ is chaotic if and only if, for any $v\in V$,
\begin{enumerate}
\item[\rm (i)] there is a fixed point $f$ for $B$ in $X$ over the rooted tree $V(v)$ with $f(v)=1$;
\item[\rm (ii)] for any $\eps>0$ there is some $N\geq 1$ and a backward invariant sequence $f$ in $X$ over the rooted tree $V^N_-(v)$ with $f(v)=1$ such that
\[
\|f-e_v\|<\eps.
\]
\end{enumerate}
\end{theorem}

\begin{proof}
Throughout this proof, if $v_0:=v\in V$ we let $v_{-n}=\prt^n(v)$, $n\geq 1$. And we note that, by unconditionality of the basis, there exists, for any $\eps>0$, some $\delta_\eps>0$ such that $\|f\|<\delta_\eps$ implies that $\|(\alpha_uf(u))_u\|<\eps$ for any sequence $(\alpha_u)_{u\in V}$ on $V$ with $\sup_{u\in V}|\alpha_u|\leq 1$.

We first prove necessity. Thus suppose that $B$ is chaotic on $X$, and let $v\in V$.  

(i) By Theorem \ref{chaos-fixedpoint}, $B$ admits a fixed point $f\in X$ for $B$ with $f(v)=1$. It then suffices to consider its restriction to $V(v)$.

(ii) Now, let $\eps>0$. Since $B$ is chaotic, there is a periodic point $f\in X$ for $B$ with $\|f-e_v\|<\delta_\eps$, and we can assume that $f(v)=1$. Let $N$ be its period. Then $f$ satisfies
\begin{equation}\label{eq-backw1}
f(v_{-nN}) = f(v_{-nN+N}) + \sum_{\substack{u\in\Chi^N(v_{-nN})\\ u\neq v_{-nN+N}}} f(u),\quad n\geq 1,
\end{equation}
as well as
\begin{equation}\label{eq-backw2}
f(w) = \sum_{u\in \Chi^N(w)} f(u)\quad \text{if } w\in V^N_-(v)\setminus \{v_{-nN}: n\geq 0\}.
\end{equation}
This shows that the sequence $f_-^N$ on $V_-^N(v)$ given by
\[
f^N_-(w) =
\begin{cases} 
f(w),&\text{if } w=v_{-nN} \text{ for some $n\geq 0$},\\
-f(w),&\text{otherwise},
\end{cases}
\]
is backward invariant on $V^N_-(v)$ with $f^N_-(v)=1$. Moreover, $\|f^N_- - e_v\|<\eps$.

We turn to the proof of sufficiency. We need to show that, under the assumptions (i) and (ii), $B$ is chaotic on $X$. Since the set of periodic points is a vector space and since $(e_v)_v$ is a basis in $X$ it suffices, by Theorem \ref{chaos-fixedpoint}, to show that, for any $v\in V$ and $\eps>0$ there exists a periodic point $f$ for $B$ such that
\[
\|f-e_v\|<2\eps.
\]

Thus, let $v_0:=v\in V$ and $\eps>0$. First, by (i), there exists a fixed point $f_+$ for $B$ in $X$ over $V_+(v)=V(v)$ with $f_+(v)=1$. Then, for any $N\geq 1$, the restriction $f_{1,N}$ to $V^N_+(v):=\bigcup_{n\geq 0}\Chi^{nN}(v)$ satisfies
\begin{equation}\label{ex-backw0}
f_{1,N}(w) = \sum_{u\in \Chi^N(w)} f(u), \quad w\in V^N_+(v).
\end{equation}
Note that $f_{1,N}(v)=1$. It follows from the unconditionality of the basis that there is some $N_0\geq 1$ such that, for any $N> N_0$,
\[
\|f_{1,N}-e_v\|<\eps.
\]
We extend $f_{1,N}$ to $X$ by setting it zero outside $V^N_+(v)$.

Now, for any $N\geq 1$, let $P_N$ be the set of all sequences $f$ on $X$ that satisfy
\begin{equation}\label{ex-backw}
f(v_{0}) = f(v_{-N}) + \sum_{\substack{u\in\Chi^N(v_{-N})\\ u\neq v_{0}}} f(u).
\end{equation}
By continuity of $B$ these are closed subsets of $X$. Thus $P:=\bigcup_{N=1}^{N_0}P_N$ is a closed set in $X$ that does not contain $e_v$. Let
\[
\eta=\min(\text{dist}(P,e_v),\delta_\eps),
\]
where $\delta_\eps$ is as given above. By (ii) there is some $N\geq 1$ and a backward invariant sequence $f_-$ in $X$ over $V_-^N(v)$ with $f_-(v)=1$ such that 
\[
\|f_--e_v\|<\eta.
\]
It follows by backward invariance on $V_-^N(v)$ that $f_-$ satisfies \eqref{ex-backw}, hence belongs to $P_N$. On the other hand, since $\|f_--e_v\|<\text{dist}(P,e_v)$ it follows that $N>N_0$. Now define $f_2$ on $V_-^N(v)$ by
\[
f_2(w) =
\begin{cases} 
f_-(w),&\text{if } w=v_{-nN} \text{ for some $n\geq 0$},\\
-f_-(w),&\text{otherwise}.
\end{cases}
\]
Note that $f_2(v)=1$. Since $\|f_--e_v\|<\delta_\eps$, it follows that
\[
\|f_2-e_v\|<\eps.
\]
We extend $f_2$ to $X$ by setting it zero outside $V_-^N(v)$.

Finally, we define the sequence $f\in X$ by
\[
f=f_{1,N}+f_2-e_v
\]
with $N$ as fixed in the construction of $f_2$. Then $f(v)=1$, and we have that
\[
\|f-e_v\|<2\eps.
\]
Moreover, since $f_-$ is backward invariant on $V^N_-(v)$ it follows that $f_2$ and therefore $f$ satisfies \eqref{eq-backw1} and \eqref{eq-backw2}; we use here that $V$ has no leaves. And $f_{1,N}$ and therefore $f$ satisfies \eqref{ex-backw0}. These three equations tell us that $f$ is a periodic point of $B$ of order $N$. This had to be shown.
\end{proof}

Now it is easy but rather tedious to transfer the theorem to arbitrary weighted backward shifts. One possibility is to repeat the arguments above for $B_\lambda$. Let us just indicate the crucial difference. For a fixed point $f$ for $B_\lambda$, equations \eqref{eq-fp} turn into
\[
f(v_{-n})=\lambda_{v_{-n+1}}f(v_{-n+1}) + \sum_{\substack{u\in\Chi(v_{-n})\\ u\neq v_{-n+1}}} \lambda_u f(u),\quad n\geq 1,
\]
which is then rewritten as  
\begin{equation}\label{eq-fp3}
f(v_{-n+1})= \frac{1}{\lambda_{v_{-n+1}}}f(v_{-n})+ \sum_{\substack{u\in\Chi(v_{-n})\\ u\neq v_{-n+1}}} \frac{\lambda_u}{\lambda_{v_{-n+1}}} (-f(u)),\quad n\geq 1.
\end{equation}
More generally, for a periodic point $f$ for $B_\lambda$ of period $N\geq 1$, the equation \eqref{eq-backw1} now becomes
\begin{equation}\label{eq-fp4}
f(v_{-nN}) = \lambda(v_{-nN}\to v_{-nN+N})f(v_{-nN+N}) + \sum_{\substack{u\in\Chi^N(v_{-nN})\\u\neq v_{-nN+N}}}\lambda(v_{-nN}\to u) f(u),\quad n\geq 1,
\end{equation}
which turns into
\begin{equation}\label{eq-fp5}
\begin{split}
f(v_{-nN+N}) = \frac{1}{\lambda(v_{-nN}\to v_{-nN+N})}&f(v_{-nN})\, + \\
+&\sum_{\substack{u\in\Chi^N(v_{-nN})\\u\neq v_{-nN+N}}}\frac{\lambda(v_{-nN}\to u)}{\lambda(v_{-nN}\to v_{-nN+N})} (-f(u)),\quad n\geq 1.
\end{split}
\end{equation}
The equations \eqref{eq-inv1} and \eqref{eq-backw2} have obvious modifications. 

It is thus natural to define a new weight $\lambda_-^N$ on $V_-^N(v)$ as follows:
\begin{equation}\label{eq-lambdan-}
\begin{split}
\lambda^N_{-,w}=
\begin{cases}
\displaystyle\frac{1}{\lambda(v_{-nN}\to v_{-nN+N})},& \text{if } w=v_{-nN}, n\geq 1,\\[1em]
\displaystyle\frac{\lambda(v_{-nN}\to w)}{\lambda(v_{-nN}\to v_{-nN+N})},&\text{if } w\in\Chi(v_{-nN+N})\setminus \{v_{-nN}\}, n\geq 1,\\[1em]
\lambda(\prt^N(w)\to w),& \text{otherwise}.
\end{cases}
\end{split}
\end{equation}
It is important to note that in \eqref{eq-fp4} and \eqref{eq-fp5} the sets $\Chi$ refer to the tree $V$, while in \eqref{eq-lambdan-} the set $\Chi$ refers to the tree $V_-^N(v)$; in particular, the set $\Chi^N(v_{-nN})\setminus\{v_{-nN+N}\}$ in $V$ is equal to the set $\Chi(v_{-nN+N})\setminus \{v_{-nN}\}$ in $V_-^N(v)$.

Let us also spell out the weight $\lambda_-=\lambda^1_-$ on $V_-(v)=V^1_-(v)$: we have that
\begin{equation}\label{eq-lambda-}
\begin{split}
\lambda_{-,w}=
\begin{cases}
\displaystyle\frac{1}{\lambda_{v_{-n+1}}},& \text{if } w=v_{-n}, n\geq 1,\\[1em]
\displaystyle\frac{\lambda_w}{\lambda_{v_{-n+1}}},& \text{if } w\in\Chi(v_{-n+1})\setminus \{v_{-n}\}, n\geq 1,\\[1em]
\lambda_w,& \text{otherwise}.
\end{cases}
\end{split}
\end{equation}

We remark in passing that for the weight $\lambda=(1)_{v\in V}$ the weights $\lambda_-^N$ are again of this form.

A second way to generalize Theorem \ref{t-chunrooted} is to use a conjugacy to replace $B_\lambda$ on $X$ by the unweighted backward shift $B$ on some weighted space $X_\mu$. However, for the definition of the weight $\mu$, a distinguished vertex was fixed in $V$. One has to show that it can be replaced by the vertex $v$ appearing in the conditions (i) and (ii). And one finally has to use two more conjugacies, separately on $V(v)$ and $V_-^N(v)$. 

In both ways one arrives at the main result of this section.

\begin{theorem}\label{t-chunrootedgen}
Let $V$ be a leafless unrooted tree and $\lambda=(\lambda_v)_{v\in V}$ a weight. Let $X$ be a Fr\'echet sequence space over $V$ in which $(e_v)_{v\in V}$ is an unconditional basis, and suppose that the weighted backward shift $B_\lambda$ is an operator on $X$. Then $B_\lambda$ is chaotic if and only if, for any $v\in V$,
\begin{enumerate}
\item[\rm (i)] there is a fixed point $f$ for $B_\lambda$ in $X$ over the rooted tree $V(v)$ with $f(v)=1$;
\item[\rm (ii)] for any $\eps>0$ there is some $N\geq 1$ and a $\lambda^N_-$-backward invariant sequence $f$ in $X$ over the rooted tree $V^N_-(v)$ with $f(v)=1$ such that
\[
\|f-e_v\|<\eps.
\]
\end{enumerate}
\end{theorem}

Using the same ideas we can also give a characterization of the existence of universal fixed points, complementing our results in Theorems \ref{chaos-fixedpoint} and \ref{chaos-positive}. Again, the tree has to be leafless because $(B_\lambda f)(v)=0$ for any leaf $v$. Recall the definition of the weight $\lambda_-$, see \eqref{eq-lambda-}. And we will say that a backward invariant sequence is \textit{universal} if it is non-zero on the whole tree.

\begin{theorem}\label{chaos-positive3}
Let $V$ be a leafless unrooted tree and $\lambda=(\lambda_v)_{v\in V}$ a weight. Let $X$ be a Fr\'echet sequence space over $V$ in which $(e_v)_{v\in V}$ is an unconditional basis, and suppose that the weighted backward shift $B_\lambda$ is an operator on $X$. 

Then $B_\lambda$ has a universal fixed point in $X$ if and only if, for some $v\in V$ and then for all $v\in V$, there is a universal fixed point for $B_\lambda$ in $X$ over the rooted tree $V(v)$ and a universal $\lambda_-$-backward invariant sequence in $X$ over the rooted tree $V_-(v)$.
\end{theorem}

Indeed, for the unweighted backward shift $B$, this is immediate from the discussion at the beginning of this section, which showed how to decompose and to reassemble fixed points over the trees $V(v)$ and $V_-(v)$ for any $v\in V$. In the general case one has to note \eqref{eq-fp3}.

As for Theorem \ref{t-chunrootedgen}, one might wonder whether one can drop condition (ii) on the trees $V_-^N(v)$ that is rather difficult to control. We know from Example \ref{ex-comb} that this is not always the case. But under an additional condition on the weight and the space, this very useful simplification is possible. For an application of this result, see Theorem \ref{t-charlambdabil_simple} and then Theorem \ref{t-symsym}. Note that the condition necessarily implies that the tree has no free left end.

We recall that $\gen(v)$ denotes the generation that contains a vertex $v$.

\begin{theorem}\label{t-chunrootedsimplified}
Let $V$ be an unrooted tree and $\lambda=(\lambda_v)_{v\in V}$ a weight. Let $X$ be a Fr\'echet sequence space over $V$ in which $(e_v)_{v\in V}$ is an unconditional basis, and suppose that the weighted backward shift $B_\lambda$ is an operator on $X$. Suppose that, for every 
$v\in V$, there is some $f\in X$ such that
\begin{equation}\label{eq-simpcond}
\sum_{u\in\emph{\gen}(v)} \Big|\frac{\lambda(w\to u)}{\lambda(w\to v)}f(u)\Big|=\infty,
\end{equation}
where $w$ is a common ancestor of $u$ and $v$. Then $B_\lambda$ is chaotic if and only if, for any $v\in V$, there is a fixed point $f$ for $B_\lambda$ in $X$ over the rooted tree $V(v)$ with $f(v)=1$.
\end{theorem}

\begin{proof}
By conjugacy, it suffices to prove the result for the unweighted backward shift $B$; note that the conjugacy \eqref{eq-conjB2} would give the factor $\tfrac{\lambda(w\to u)}{\lambda(w\to v_0)}\frac{1}{\mu_{v_0}}$ in \eqref{eq-simpcond}, where $v_0$ is a fixed vertex in $V$ and $w$ is a common ancestor of $v_0$, $u$ and $v$. However, since
\begin{equation}\label{eq-conjug}
\frac{\lambda(w\to u)}{\lambda(w\to v_0)} = \frac{\lambda(w\to v)}{\lambda(w\to v_0)}\frac{\lambda(w\to u)}{\lambda(w\to v)}, 
\end{equation}
we may replace the factor by $\tfrac{\lambda(w\to u)}{\lambda(w\to v)}$.

By Theorem \ref{chaos-fixedpoint}, it suffices to prove sufficiency. Thus let $v\in V$ and $\eps>0$. Then there is a fixed point $f_0$ for $B$ in $X$ over the rooted tree $V(v)$ with $f_0(v)=1$. 

Also, by assumption, there is some $f\in X$ such that
\[
\sum_{u\in \gen(v)} |f(u)|=\infty.
\]
We fix an enumeration $\{v_1,v_2,\ldots\}$ of $\gen(v)$. By unconditionality of the basis, there is some $n\geq 1$ such that
\[
\Big\|\frac{1}{\sum_{j=1}^n |f(v_j)|}\sum_{k=1}^n |f(v_k)| e_{v_k} \Big\| <\eps.
\]
Setting, for $1\leq k\leq n$,
\[
a_k = \frac{|f(v_k)|}{\sum_{j=1}^n |f(v_j)|}
\]
we have that $\sum_{k=1}^n a_k=1$ and
\begin{equation}\label{eq-aka}
\Big\|\sum_{k=1}^n a_ke_{v_k}\Big\|<\eps.
\end{equation}
Then, again by the hypothesis, there are fixed points $f_k$ for $B$ in $X$ over the rooted tree $V(v_k)$ with $f_k(v_k)=1$, $k=1,\ldots,n$. 

Let $N_0\geq 1$ be such that $v,v_1,\ldots,v_n$ have a common ancestor of degree $N_0$. Then, by unconditionality of the basis, we can choose $N\geq N_0$ such that
\begin{equation}\label{eq-aka2}
\Big\|\Big(f_0-\sum_{k=1}^n f_k\Big)\chi_{\bigcup_{j=1}^{\infty}\gen_{jN}}\Big\|<\varepsilon,
\end{equation}
where we have assumed without loss of generality that $v\in \gen_0$. We set
\[
f=\Big(f_0-\sum_{k=1}^n a_kf_k\Big)\chi_{\bigcup_{j=0}^{\infty}\gen_{jN}}.
\]
Since $\sum_{k=1}^n a_k = 1$, $f_k(v_k)=1$ for $k=0,\ldots,n$, and $v,v_1,\ldots,v_n$ have a common ancestor of degree $N$, we have that $f$ is a periodic point for $B$ of period $N$. Moreover, 
\[
f-e_v = -\sum_{k=1}^n a_k e_{v_k} + \Big(f_0-\sum_{k=1}^n a_kf_k\Big)\chi_{\bigcup_{j=1}^{\infty}\gen_{jN}}
\]
and thus, by \eqref{eq-aka} and \eqref{eq-aka2},
\[
\|f-e_v\| < 2\eps.
\]
Altogether we have shown that the periodic points for $B$ are dense in $X$, which implies by Theorem \ref{chaos-fixedpoint} that $B$ is chaotic.
\end{proof}

\begin{remark}\label{rem-sym}
(a) The existence of $f\in X$ that satisfies \eqref{eq-simpcond} can be expressed as saying that the sequence $\big(\tfrac{\lambda(w\to u)}{\lambda(w\to v)}\big)_{u\in\gen(v)}$ does not belong to the $\alpha$-dual (also called the K\"othe-dual) of the sequence space $X$ over $\gen(v)$; see \cite[Section 7.1]{Boo00}. By H\"older's inequality and its converse, the $\alpha$-dual of the classical sequence spaces $\ell^p$, $1\leq p<\infty$, and $c_0$ is $\ell^{p^\ast}$, with $p^\ast=1$ for the space $c_0$.

(b) We note that if the weight $\lambda$ is symmetric, see Section \ref{s-notation}, then equation \eqref{eq-simpcond} reduces to
\[
\sum_{u\in\gen(v)}|f(u)|=\infty.
\]

(c) The theorem shows that, under \eqref{eq-simpcond}, the conditions (iii) in Theorem \ref{chaos-fixedpoint} imply the conditions (ii). We remark, however, that even under this additional assumption, the conditions (ii) in Theorem \ref{chaos-fixedpoint} do not imply (i). Indeed, consider the situation of Example \ref{ex-comb2}. Then, for any $v=(n,k)\in V=\ZZ\times\NN_0$, the sequence $f=\chi_{\gen(v)}$ belongs to $X=\ell^1(V,\mu)$, while $\sum_{u\in \gen(v)} |f(u)|=\infty$, so that \eqref{eq-simpcond} holds for the unweighted backward shift $B$. But it was shown in that example that $B$ is a chaotic operator that does not satisfy condition (i) of Theorem \ref{chaos-fixedpoint}.
\end{remark}

\section{Two constants associated with a weighted rooted tree}\label{s-constcprp} 

We will now turn to the special case in which the weighted backward shift acts on a space $\ell^p(V)$ or $c_0(V)$ over a rooted tree $V$. As already discussed, chaos of such an operator is equivalent to chaos for the unweighted backward shift $B$ on a weighted space $\ell^p(V,\mu)$ or $c_0(V,\mu)$. 

By Theorem \ref{chaos-fixedpoint} we then have to find, for any $v\in V$, a fixed point $f$ of $B$ with $f(v)=1$, that is, we want that
\[
f(v) = \sum_{u\in\Chi(v)} f(u),
\]
and similarly for all descendants of $v$. The initial strategy is clear: starting from $v$ we need to successively distribute the mass  1 over the following generations in such a way that the norm of the resulting vector stays bounded in order for the limit vector to belong to the underlying space. But this is a daunting task. At first it seems a good idea to distribute the mass in an optimal way according to the weights $\mu_u$ of the children; yet it could be that along one branch the generations first have few children, and all with big weights, so that one is inclined to pass on only a small part of the mass along this branch, while much later on this branch there may appear generations with many children of small weights, which could have allowed a bigger share along this branch right from the beginning. We are therefore led to perform a kind of backward-forward induction: we first fix $m$, find an optimal fixed point for the first $m$ generations by starting form the last generation backwards to the root, and then let $m$ tend to infinity.

While we know from \cite{GrPa21} that hypercyclic weighted backward shifts can only live on trees without leaves, the strategy that we have just described forces us to look also at trees that terminate with generation $n$.

\vspace{\baselineskip}
\textit{We will consider in this and the following section arbitrary rooted trees, that is, rooted trees that may or may not have leaves.}
\vspace{\baselineskip} 

It turns out that, in the characterization of chaotic weighted backward shifts on spaces of type $\ell^p$ and $c_0$, two closely related constants appear crucially. We ask the reader to first accept the expressions for these constants at face value. Their proper interpretations and certain subtleties will be explained in the subsequent remarks.

In the sequel we will be working in $[0,\infty]$, and all the constants possibly take the value infinity. We usually interpret $\infty^{-1}=0$ and $0^{-1}=\infty$, but we will mention exceptions to this rule, see for example Remark \ref{rem-contfr}(a).

\begin{definition}\label{d-const}
Let $V$ be a tree with root $v_0$ and $\mu=(\mu_v)_{v\in V}$ a weight on $V$. 

(a) For $1\leq p\leq \infty$, the \textit{continued fraction $c_p(V,\mu)$ of order $p$} associated to $V$ and $\mu$ is defined as follows:

$-$ for $p=1$,
\[
c_1(V,\mu)=\inf_{(v_0,v_1,v_2,\ldots)}\sum_{n\geq 1} |\mu_{v_n}|,
\]
where $(v_0,v_1,v_2,\ldots)$ represents an arbitrary branch in $V$ that starts from $v_0$;

$-$ for $1<p<\infty$,
\[
c_p(V,\mu)=\frac{1}{\Bigg( \sum\limits_{v_1\in \Chi(v_0)}\frac{1}{\Bigg( |\mu_{v_1}|^p+\frac{1}{\bigg( \mathlarger{\mathlarger{\sum\limits}_{v_2\in \Chi(v_1)}}\frac{1}{\Big( |\mu_{v_2}|^p+\frac{1}{\big( \mathlarger{\sum\limits}_{v_3\in \Chi(v_2)} \frac{1}{\big(|\mu_{v_{3}}|^p+\ldots \big)^{p^{\ast}/p}}\big)^{p/p^{\ast}}} \Big)^{p^{\ast}/p}}\bigg)^{p/p^{\ast}}} \Bigg)^{p^{\ast}/p}}\Bigg)^{1/p^{\ast}}};
\]

$-$ for $p=\infty$,
\[
c_\infty(V,\mu)=\frac{1}{ \sum\limits_{v_1\in \Chi(v_0)}\frac{1}{\max\Big( |\mu_{v_1}|,\frac{1}{\big( \mathlarger{\sum\limits}_{v_2\in \Chi(v_1)}\frac{1}{\mathlarger{\max}\big( |\mu_{v_2}|,\frac{1}{\big( \mathlarger{\sum\limits}_{v_3\in \Chi(v_2)} \frac{1}{\max\big(|\mu_{v_{3}}|,\ldots\big)}\big)} \big)}\big)} \Big)}}.
\]

(b) For $1\leq p\leq \infty$, the \textit{resistance $r_p(V,\mu)$ of order $p$} of $V$ and $\mu$ is defined as follows:

$-$ for $p=1$,
\begin{align*}
r_1(V,\mu) &= |\mu_{v_0}|+c_1(V,\mu)\\
&=\inf_{(v_0,v_1,v_2,\ldots)}\sum_{n\geq 0} |\mu_{v_n}|,
\end{align*}
where $(v_0,v_1,v_2,\ldots)$ represents an arbitrary branch in $V$ that starts from $v_0$;

$-$ for $1<p<\infty$,
\begin{align*}
r_p(V,\mu) &= (|\mu_{v_0}|^p + c_p(V,\mu)^p)^{1/p}\\
&=\Bigg(|\mu_{v_0}|^p+\frac{1}{\Bigg( \sum\limits_{v_1\in \Chi(v_0)}\frac{1}{\Bigg( |\mu_{v_1}|^p+\frac{1}{\big( \mathlarger{\sum\limits}_{v_2\in \Chi(v_1)}\frac{1}{\big( |\mu_{v_2}|^p+\ldots \big)^{p^{\ast}/p}}\big)^{p/p^{\ast}}} \Bigg)^{p^{\ast}/p}}\Bigg)^{p/p^{\ast}}}\Bigg)^{1/p};
\end{align*}

$-$ for $p=\infty$,
\begin{align*}
r_\infty(V,\mu) &= \max(|\mu_{v_0}|, c_\infty(V,\mu))\\
&= \max\Bigg(|\mu_{v_0}|,\frac{1}{ \sum\limits_{v_1\in \Chi(v_0)}\frac{1}{\max\Big( |\mu_{v_1}|,\frac{1}{\big( \mathlarger{\sum\limits}_{v_2\in \Chi(v_1)}\frac{1}{\mathlarger{\max}\big( |\mu_{v_2}|,\ldots \big)}\big)} \Big)}}
\Bigg).
\end{align*}
\end{definition}

The name of the constants $c_p(V,\mu)$ derives from their resemblance with the classical continued fractions. And we call the constants $r_p(V,\mu)$ the resistance of the weighted tree by analogy with the resistance in electrical circuits. We will discuss both aspects at the end of this section.

As already said, the definition of the constants $c_p$ and $r_p$ contains some subtleties (unless for $p=1$). The following remark and the theorem after it are therefore particularly important for a proper understanding of the constants.

\begin{remark}\label{rem-contfr}
Let $1<p\leq \infty$.

(a) For any $v\in V$, the expressions
\begin{equation}\label{eq-series}
\frac{1}{\sum\limits_{u\in \Chi(v)} \frac{1}{\big(|\mu_u|^p+\ldots\big)^{p^\ast/p}}},\quad 1<p<\infty, \quad\text{and}\quad\frac{1}{\sum\limits_{u\in \Chi(v)} \frac{1}{\max(|\mu_u|,\ldots)}}
\end{equation}
appear in $c_p(V,\mu)$. If $v$ is a leaf of $V$, these expressions have to be interpreted as 0 (and not as $\infty$ as might be expected). In particular
\[
c_p(\{v_0\},\mu) = 0.
\]
And if every branch starting from the root $v_0$ has length $m\geq 1$, then, for $1<p<\infty$, 
\begin{align*}
&c_p(V,\mu)=\\
& \frac{1}{\Bigg( \sum\limits_{v_1\in \Chi(v_0)}\frac{1}{\Bigg( |\mu_{v_1}|^p+\frac{1}{\bigg( \mathlarger{\sum\limits}_{v_2\in \Chi(v_1)}\frac{1}{\Big( |\mu_{v_2}|^p+\frac{1}{\big( \ddots \big(|\mu_{v_{m-1}}|^p+\frac{1}{\big( \mathlarger{\sum\limits}_{v_m\in \Chi(v_{m-1})}\frac{1}{|\mu_{v_m}|^{p^{\ast}}}\big)^{p/p^{\ast}}} \big)^{p^{\ast}/p}\big)^{p/p^{\ast}}} \Big)^{p^{\ast}/p}}\bigg)^{p/p^{\ast}}} \Bigg)^{p^{\ast}/p}}\Bigg)^{1/p^{\ast}}}
\end{align*}
and
\[
c_\infty(V,\mu)=\frac{1}{ \sum\limits_{v_1\in \Chi(v_0)}\frac{1}{ \max\Big(|\mu_{v_1}|,\frac{1}{ \mathlarger{\sum\limits}_{v_2\in \Chi(v_1)}\frac{1}{\mathlarger{\max}\big(|\mu_{v_2}|,\frac{1}{\big( \ddots \max\big(|\mu_{v_{m-1}}|,\frac{1}{ \sum\limits_{v_m\in \Chi(v_{m-1})}\frac{1}{|\mu_{v_m}|}} \big)\big)} \big)}}\Big)}}.
\]

(b) In addition, it may well be that the series in \eqref{eq-series} are infinite, in which case the quotient is to be interpreted as 0. This will, however, not arise if the backward shift $B$ is defined on $\ell^p(V,\mu)$ or $c_0(V,\mu)$ (equivalently, on $\ell^\infty(V,\mu)$) because then, for all $v\in V$, $\sum_{u\in\Chi(v)} |\mu_u|^{-p^\ast}<\infty$ and $\sum_{u\in\Chi(v)} |\mu_u|^{-1}<\infty$, respectively.

(c) Now, under the conventions of (a) and (b), we have the following recursive definition of the values $c_p(V,\mu)$. Let us consider the subtree $V_m = \bigcup_{n=0}^m \gen_n$, $m\geq 0$. If $1<p<\infty$, then $c_p(V_0,\mu)=0$,
\[
c_p(V_1,\mu)= \frac{1}{\big(\sum\limits_{v_1\in \Chi(v_0)}\frac{1}{|\mu_{v_1}|^{p^{\ast}}}\big)^{1/p^{\ast}}},
\]
and, for any $m\geq 1$,
\begin{equation}\label{eq-formulam}
c_p(V_{m+1},\mu)= c_p(V_m,\mu'),
\end{equation}
where $\mu'$ coincides with $\mu$ on $V_{m-1}$, while for any $v\in \gen_m$, $\mu'_v$ is such that
\begin{equation}\label{eq-formula}
|\mu'_{v}|=\Big(|\mu_{v}|^p+\frac{1}{\big(\sum\limits_{u\in \Chi(v)}\frac{1}{|\mu_{u}|^{p^{\ast}}}\big)^{p/p^{\ast}}}\Big)^{1/p};
\end{equation}
note that if $v\in\gen_m$ is a leaf in $V$ then, by (a), $|\mu_v'|=|\mu_v|$. Incidentally, \eqref{eq-formulam} does not hold for $m=0$.
 
It follows that the sequence $(c_p(V_m,\mu))_{m\geq 0}$ is increasing, which, by the way, distinguishes our continued fractions from the classical ones. Then $c_p(V,\mu)$ is defined as
\[
c_p(V,\mu) = \lim_{m\to\infty} c_p(V_m,\mu) =  \sup_{m\geq 0} c_p(V_m,\mu).
\]

The definition of $c_\infty(V,\mu)$ is similar, where now 
\[
c_\infty(V_1,\mu)= \frac{1}{\sum\limits_{v_1\in \Chi(v_0)}\frac{1}{|\mu_{v_1}|}}
\]
and $\mu'$ is given for $v\in \gen_m$ by
\[
|\mu'_{v}|=\max\Big(|\mu_{v}|,\frac{1}{\sum\limits_{u\in \Chi(v)}\frac{1}{|\mu_{u}|}}\Big).
\]

This definition of the constants $c_p$, which uses a kind of backward-forward recurrence, reflects their use in the characterization of the existence of fixed points; see our remarks at the beginning of the section and the proofs in Subsection \ref{s-backinvarb}.

(d) As for the values $r_p(V,\mu)$ we have accordingly, for $1<p\leq \infty$, $r_p(V_0,\mu)=|\mu_{v_0}|$,
\[
r_p(V_{m+1},\mu)= r_p(V_m,\mu'),
\]
which holds here for all $m\geq 0$, and 
\[
r_p(V,\mu)= \lim_{m\to\infty} r_p(V_m,\mu) =  \sup_{m\geq 0} r_p(V_m,\mu).
\]

(e) We note that all constants $c_p(V,\mu)$ and $r_p(V,\mu)$, $1\leq p\leq \infty$, are finite if the tree has finite length.
\end{remark} 

\begin{example}\label{ex-unbranched}
In the special case of a rooted tree $V$ with root $v_0$ in which each vertex has at most one child, which therefore has a single finite or infinite branch $(v_0,v_1,v_2,\ldots)$, we have that
\[
c_p(V,\mu)= \Big(\sum_{n\geq 1}|\mu_{v_n}|^p\Big)^{1/p}\quad \text{and}\quad r_p(V,\mu)= \Big(\sum_{n\geq 0}|\mu_{v_n}|^p\Big)^{1/p}, \ 1\leq p \leq \infty,
\]
where for $p=\infty$ the $p$-norm has to be replaced by the sup-norm. 
\end{example}

An obvious feature of the constants is their positive homogeneity with respect to the weight $\mu$. On closer inspection one also perceives some recursive structure, see also Figure \ref{fig-rec1}. Recall that $V(v)$ denotes the subtree of descendants of $v\in V$.

\begin{theorem}\label{t-rec1}
Let $V$ be a tree with root $v_0$ and $\mu=(\mu_v)_{v\in V}$ a weight on $V$.

\emph{(a)} For $1\leq p\leq \infty$, $c_p(V,\mu)$ and $r_p(V,\mu)$ are positively homogeneous in the weight.

\emph{(b)} Let $v\in\Chi(v_0)$. Then,

\noindent $-$ for $1\leq p<\infty$,
\[
c_p(V,\mu)\leq \big(|\mu_{v}|^p+ c_p(V(v),\mu)^p\big)^{1/p};
\]
$-$ for $p=\infty$,
\[
c_\infty(V,\mu) \leq \max(|\mu_{v}|,c_\infty(V(v),\mu));
\]
$-$ for $1\leq p<\infty$,
\[
r_p(V,\mu)\leq \big(|\mu_{v_0}|^p+ r_p(V(v),\mu)^p\big)^{1/p};
\]
$-$ for $p=\infty$,
\[
r_\infty(V,\mu) \leq \max(|\mu_{v_0}|,r_\infty(V(v),\mu)).
\]

\emph{(c)} Suppose that $V$ has finite length, or that $f\to \sum_{v\in \Chi(v_0)} f(v)$ is continuous on $\ell^p(\Chi(v_0),\mu)$, $1<p\leq\infty$. Then,

\noindent $-$  for $1<p<\infty$,
\begin{equation}\label{eq-cp}
c_p(V,\mu)=\frac{1}{\Big( \sum\limits_{v\in \Chi(v_0)}\frac{1}{\big( |\mu_{v}|^p+ c_p(V(v),\mu)^p\big)^{p^{\ast}/p}}\Big)^{1/p^{\ast}}};
\end{equation}
$-$ for $p=\infty$,
\[
c_\infty(V,\mu)=\frac{1}{ \sum\limits_{v\in \Chi(v_0)}\frac{1}{\max\big( |\mu_{v}|, c_\infty(V(v),\mu)\big)}};
\]
$-$ for $1<p<\infty$,
\[
r_p(V,\mu)=\Big(|\mu_{v_0}|^p + \frac{1}{\Big( \sum\limits_{v\in \Chi(v_0)}\frac{1}{r_p(V(v),\mu)^{p^{\ast}}}\Big)^{p/p^{\ast}}}\Big)^{1/p};
\]
$-$ for $p=\infty$,
\[
r_\infty(V,\mu)=\max\Big(|\mu_{v_0}|,\frac{1}{ \sum\limits_{v\in \Chi(v_0)}\frac{1}{r_\infty(V(v),\mu)}}\Big).
\]
\end{theorem}

\begin{proof} For $1<p<\infty$, we find that, for $m\geq 1$,
\[
c_p(V_m,\mu)=\frac{1}{\Big( \sum\limits_{v\in \Chi(v_0)}\frac{1}{\big( |\mu_{v}|^p+ c_p(V(v)_{m-1},\mu)^p\big)^{p^{\ast}/p}}\Big)^{1/p^{\ast}}},
\]
where $V_m$ cuts off the tree $V$ at generation $m$. From this, the assertion in (b) is obvious after letting $m\to\infty$. In the same way, the assertion in (c) follows if the tree is of finite length. But if the tree has infinite length we need to interchange a limit with a summation. If $f\to \sum_{v\in \Chi(v_0)} f(v)$ is continuous on $\ell^p(\Chi(v_0),\mu)$, then $\sum_{v\in \Chi(v_0)} |\mu_v|^{-p^\ast}<\infty$, so that the dominated convergence theorem allows to conclude.

The remaining assertions are either proved similarly or are obvious.
\end{proof}

When one looks at the expression for the constants $c_p$, Assertion (c) might at first seem obvious and require no assumptions. The following example shows, however, that this is not the case. The second assumption in (c) will turn out to be natural in our context, see the discussion before Theorem \ref{t-charfinpmin}. 

\begin{example}\label{ex-cuttree}
Consider the rooted tree $V$ where the root $v_0$ has infinitely many children and any other descendant of $v_0$ has exactly one child. Let $\mu(v)=1$ for all $v\in V$, and let $1<p<\infty$. Then, for $m\geq 1$, $c_p(V_m,\mu)=0$ because $c_p(V_{m}(v),\mu)=(m-1)^{1/p}$ for all $v\in \Chi(v_0)$, see Example \ref{ex-unbranched}. Hence $c_p(V,\mu)=0$. On the other hand, $c_p(V(v),\mu)=\infty$ for all $v\in \Chi(v_0)$, so that the right-hand side in \eqref{eq-cp} is infinite.
\end{example}

The finiteness of $c_p$ will play a decisive r\^ole for chaos, see Section \ref{s-chaosrooted}. Thus the following is of interest.

\begin{corollary}\label{c-subtrees}
Let $V$ be a tree with root $v_0$ that has at least one child, and let $\mu=(\mu_v)_{v\in V}$ be a weight on $V$. Let $1\leq p\leq\infty$.

\emph{(a)} If there is a child $v$ of $v_0$ such that $c_p(V(v),\mu)<\infty$, then $c_p(V,\mu)<\infty$.

\emph{(b)} Suppose that $p=1$, or else that $1<p\leq \infty$ and either $V$ is of finite length or $f\to \sum_{v\in \Chi(v_0)} f(v)$ is continuous on $\ell^p(\Chi(v_0),\mu)$. Then $c_p(V,\mu)<\infty$ if and only if there is a child $v$ of $v_0$ such that $c_p(V(v),\mu)<\infty$.
\end{corollary}

The previous example shows again that we cannot drop the additional assumptions in (b).

We have seen in Remark \ref{rem-contfr}(c) that the proper definition of the constants $c_p$ is by some backward-forward recurrence. Theorem \ref{t-rec1}(c) revealed a kind of recursive structure of the $c_p$. We will now discuss a second type of recursiveness.

For this, let $V$ be a tree with root $v_0$. Let us consider a branch $(v_0,v_1,v_2,\ldots)$ starting from the root. We will then regroup all descendants of $v_n$ except those that are also descendants of $v_{n+1}$ into a new tree
\[
W_n = V(v_n)\setminus V(v_{n+1}), n\geq 0,
\]
see Figure \ref{fig-rec2}. If $v_m$ is a leaf then the list ends with $W_m=\{v_m\}$.

For the following formulas we need to modify slightly the constants $c_p$. We set
\begin{equation}\label{eq-cp'}
\frac{1}{c_p'(V,\mu)}=\begin{cases}
0, &\text{if $V$ is a singleton},\\
\frac{1}{c_p(V,\mu)}, & \text{otherwise.}
\end{cases}
\end{equation}

Interestingly, the formulas that we obtain are essentially classical continued fractions (apart from some additional exponents). The proof shows that these continued fractions have to be taken as the limit of the \textit{even-numbered} convergents.

\begin{figure}
\begin{minipage}[t]{.5\textwidth}
\centering
\begin{tikzpicture}
\draw[fill] (0,0) circle (.5pt);

\draw[->,>=latex] (0,0) -- (1,.6);\draw[fill] (1,.6) circle (.5pt);
\draw[->,>=latex] (0,0) -- (1,.2);\draw[fill] (1,.2) circle (.5pt);
\draw[->,>=latex] (0,0) -- (1,-.2);\draw[fill] (1,-.2) circle (.5pt);
\draw[->,>=latex] (0,0) -- (1,-.6);\draw[fill] (1,-.6) circle (.5pt);

\draw[->,>=latex] (1,-.6) -- (2,-1.1);\draw[fill] (2,-1.1) circle (.5pt);
\draw[->,>=latex] (1,-.6) -- (2,-.9);\draw[fill] (2,-.9) circle (.5pt);
\draw[->,>=latex] (1,-.6) -- (2,-.7);\draw[fill] (2,-.7) circle (.5pt);
\draw[->,>=latex] (1,-.6) -- (2,-.5);\draw[fill] (2,-.5) circle (.5pt);

\draw[->,>=latex] (1,.6) -- (2,.9);\draw[fill] (2,.9) circle (.5pt);

\draw[->,>=latex] (1,.2) -- (2,.6);\draw[fill] (2,.6) circle (.5pt);
\draw[->,>=latex] (1,.2) -- (2,.4);\draw[fill] (2,.4) circle (.5pt);
\draw[->,>=latex] (1,.2) -- (2,.2);\draw[fill] (2,.2) circle (.5pt);

\draw[->,>=latex] (1,-.2) -- (2,-.2);\draw[fill] (2,-.2) circle (.5pt);

\draw[->,>=latex] (2,-1.1) -- (3,-1.5);\draw[fill] (3,-1.5) circle (.5pt);
\draw[->,>=latex] (2,-1.1) -- (3,-1.4);\draw[fill] (3,-1.4) circle (.5pt);
\draw[->,>=latex] (2,-1.1) -- (3,-1.3);\draw[fill] (3,-1.3) circle (.5pt);
\draw[->,>=latex] (2,-1.1) -- (3,-1.2);\draw[fill] (3,-1.2) circle (.5pt);

\draw[->,>=latex] (2,-.5) -- (3,-.45);\draw[fill] (3,-.45) circle (.5pt);
\draw[->,>=latex] (2,-.7) -- (3,-.6);\draw[fill] (3,-.6) circle (.5pt);
\draw[->,>=latex] (2,-.7) -- (3,-.7);\draw[fill] (3,-.7) circle (.5pt);
\draw[->,>=latex] (2,-.9) -- (3,-.9);\draw[fill] (3,-.9) circle (.5pt);
\draw[->,>=latex] (2,-.9) -- (3,-1);\draw[fill] (3,-1) circle (.5pt);

\draw[->,>=latex] (2,.9) -- (3,1);\draw[fill] (3,1) circle (.5pt);
\draw[->,>=latex] (2,.9) -- (3,1.1);\draw[fill] (3,1.1) circle (.5pt);

\draw[->,>=latex] (2,.6) -- (3,.7);\draw[fill] (3,.7) circle (.5pt);
\draw[->,>=latex] (2,.6) -- (3,.6);\draw[fill] (3,.6) circle (.5pt);

\draw[->,>=latex] (2,.4) -- (3,.4);\draw[fill] (3,.4) circle (.5pt);
\draw[->,>=latex] (2,.4) -- (3,.3);\draw[fill] (3,.3) circle (.5pt);

\draw[->,>=latex] (2,.2) -- (3,.15);\draw[fill] (3,.15) circle (.5pt);

\draw[->,>=latex] (2,-.2) -- (3,-.3);\draw[fill] (3,-.3) circle (.5pt);
\draw[->,>=latex] (2,-.2) -- (3,-.2);\draw[fill] (3,-.2) circle (.5pt);
\draw[->,>=latex] (2,-.2) -- (3,-.1);\draw[fill] (3,-.1) circle (.5pt);

\node at (3.5,1) {\footnotesize{$V(v_1)$}};
\node at (3.5,.4) {\footnotesize{$V(v_2)$}};
\node at (3.5,-0.23) {\footnotesize{$V(v_3)$}};
\node at (3.5,-.95) {\footnotesize{$V(v_4)$}};

\node at (-.3,0) {\footnotesize{$v_0$}};
\node at (1,.8) {\footnotesize{$v_1$}};
\node at (1,0.37) {\scriptsize{$v_2$}};
\node at (1,-.38) {\scriptsize{$v_3$}};
\node at (1,-.8) {\footnotesize{$v_4$}};
\end{tikzpicture}
\caption{Recursive structure for Theorem \ref{t-rec1}(c)}%
\label{fig-rec1}
\end{minipage}%
\begin{minipage}[t]{.5\textwidth}
\centering
\begin{tikzpicture}
\draw[fill] (0,0) circle (.5pt);

\draw[->,>=latex] (0,0) -- (1,.2);\draw[fill] (1,.2) circle (.5pt);
\draw[->,>=latex] (0,0) -- (1,-.5);\draw[fill] (1,-.5) circle (.5pt);
\draw[->,>=latex] (0,0) -- (1,.6);\draw[fill] (1,.6) circle (.5pt);
\draw[->,>=latex] (0,0) -- (1,.4);\draw[fill] (1,.4) circle (.5pt);

\draw[->,>=latex] (1,-.5) -- (2,-1);\draw[fill] (2,-1) circle (.5pt);

\draw[->,>=latex] (1,-.5) -- (2,-.6);\draw[fill] (2,-.6) circle (.5pt);
\draw[->,>=latex] (1,-.5) -- (2,-.4);\draw[fill] (2,-.4) circle (.5pt);
\draw[->,>=latex] (1,-.5) -- (2,-.2);\draw[fill] (2,-.2) circle (.5pt);

\draw[->,>=latex] (1,.6) -- (2,1.1);\draw[fill] (2,1.1) circle (.5pt);

\draw[->,>=latex] (1,.4) -- (2,.9);\draw[fill] (2,.9) circle (.5pt);
\draw[->,>=latex] (1,.4) -- (2,.7);\draw[fill] (2,.7) circle (.5pt);
\draw[->,>=latex] (1,.4) -- (2,.5);\draw[fill] (2,.5) circle (.5pt);

\draw[->,>=latex] (1,.2) -- (2,.3);\draw[fill] (2,.3) circle (.5pt);

\draw[->,>=latex] (2,-1) -- (3,-1.5);\draw[fill] (3,-1.5) circle (.5pt);

\draw[->,>=latex] (2,-1) -- (3,-1.1);\draw[fill] (3,-1.1) circle (.5pt);
\draw[->,>=latex] (2,-1) -- (3,-1);\draw[fill] (3,-1) circle (.5pt);
\draw[->,>=latex] (2,-1) -- (3,-.9);\draw[fill] (3,-.9) circle (.5pt);

\draw[->,>=latex] (2,-.2) -- (3,-.1);\draw[fill] (3,-.1) circle (.5pt);
\draw[->,>=latex] (2,-.4) -- (3,-.25);\draw[fill] (3,-.25) circle (.5pt);
\draw[->,>=latex] (2,-.4) -- (3,-.35);\draw[fill] (3,-.35) circle (.5pt);
\draw[->,>=latex] (2,-.6) -- (3,-.5);\draw[fill] (3,-.5) circle (.5pt);
\draw[->,>=latex] (2,-.6) -- (3,-.6);\draw[fill] (3,-.6) circle (.5pt);

\draw[->,>=latex] (2,1.1) -- (3,1.4);\draw[fill] (3,1.4) circle (.5pt);
\draw[->,>=latex] (2,1.1) -- (3,1.3);\draw[fill] (3,1.3) circle (.5pt);

\draw[->,>=latex] (2,.9) -- (3,1.15);\draw[fill] (3,1.15) circle (.5pt);
\draw[->,>=latex] (2,.9) -- (3,1.05);\draw[fill] (3,1.05) circle (.5pt);

\draw[->,>=latex] (2,.7) -- (3,.9);\draw[fill] (3,.9) circle (.5pt);
\draw[->,>=latex] (2,.7) -- (3,.8);\draw[fill] (3,.8) circle (.5pt);

\draw[->,>=latex] (2,.5) -- (3,.65);\draw[fill] (3,.65) circle (.5pt);

\draw[->,>=latex] (2,.3) -- (3,.3);\draw[fill] (3,.3) circle (.5pt);
\draw[->,>=latex] (2,.3) -- (3,.4);\draw[fill] (3,.4) circle (.5pt);
\draw[->,>=latex] (2,.3) -- (3,.5);\draw[fill] (3,.5) circle (.5pt);

\node at (3.5,.8) {\footnotesize{$W_0$}};
\node at (3.5,-.4) {\footnotesize{$W_1$}};
\node at (3.5,-1.1) {\footnotesize{$W_2$}};

\node at (-.3,0) {\footnotesize{$v_0$}};
\node at (1,-.7) {\footnotesize{$v_1$}};
\node at (2,-1.2) {\footnotesize{$v_2$}};
\node at (3,-1.7) {\footnotesize{$v_3$}};
\end{tikzpicture}
\caption{Recursive structure for Theorem \ref{t-rec2}}%
\label{fig-rec2}
\end{minipage}
\end{figure}

\begin{theorem}\label{t-rec2}
Let $V$ be a tree with root $v_0$ and $\mu=(\mu_v)_{v\in V}$ a weight on $V$. Let $(v_0,v_1,v_2,\ldots)$ be a branch starting from the root, and let $W_n$, $n\geq 0$, be defined as above. Then we have the following:

$-$ for $1<p<\infty$,
\[
c_p(V,\mu)= \frac{1}{\Bigg( \frac{1}{c'_p(W_0,\mu)^{p^*}} + \frac{1}{\Bigg( |\mu_{v_1}|^p+\frac{1}{\bigg( \frac{1}{c'_p(W_1,\mu)^{p^*}}+\frac{1}{\Big( |\mu_{v_2}|^p+\frac{1}{\big( \frac{1}{c'_p(W_2,\mu)^{p^*}} + \ldots \big)^{p/p^{\ast}}} \Big)^{p^{\ast}/p}}\bigg)^{p/p^{\ast}}} \Bigg)^{p^{\ast}/p}}\Bigg)^{1/p^{\ast}}};
\]

$-$ for $p=\infty$,
\[
c_\infty(V,\mu)= \frac{1}{ \frac{1}{c'_\infty(W_0,\mu)} + \frac{1}{\max\big( |\mu_{v_1}|,\frac{1}{ \frac{1}{c'_\infty(W_1,\mu)}+\frac{1}{ \max\big(|\mu_{v_2}|,\frac{1}{\big( \frac{1}{c'_\infty(W_2,\mu)} + \ldots \big)} \big)}} \big)}}.
\]
\end{theorem}

\begin{proof}
Let again $V_m=\bigcup_{n=0}^m \gen_m$ be the tree up to generation $m\geq 1$, and set
\[
W_{n,m} = W_n\cap V_m, \ 0\leq n\leq m-1.
\]

Let $1<p<\infty$, and write the right-hand side of the claimed formula for $c_p(V,\mu)$ as $d$. It is easy to see that
\[
c_p(V_m,\mu)= \frac{1}{\Bigg( \frac{1}{c_p'(W_{0,m},\mu)^{p^*}} + \frac{1}{\Bigg( |\mu_{v_1}|^p+\frac{1}{\bigg( \frac{1}{c_p'(W_{1,m},\mu)^{p^*}}+\frac{1}{\Big(\ddots \frac{1}{\big( \frac{1}{c_p'(W_{m-1,m},\mu)^{p^*}} + \frac{1}{|\mu_{v_m}|^{p^\ast}} \big)^{p/p^{\ast}}} \Big)^{p^{\ast}/p}}\bigg)^{p/p^{\ast}}} \Bigg)^{p^{\ast}/p}}\Bigg)^{1/p^{\ast}}}.
\]

It is here that the curious reinterpretation \eqref{eq-cp'} intervenes: the contribution of the subtree $W_n$, $n\leq m-1$, inside $c_p(V_m,\mu)$ is
\[
\sum_{\substack{u\in \Chi(v_n)\\u\neq v_{n+1}}} \ldots,
\]
which equals 
\[
\frac{1}{c_p(W_{n,m},\mu)^{p^*}}
\]
unless when $W_n$ is a singleton, in which case the sum disappears, so that $\frac{1}{c_p'(W_{n,m},\mu)^{p^*}}$ needs to be taken as zero. 

Since $c_p'(W_{n,m},\mu)\leq c_p'(W_n,\mu)$ for all $n\leq m-1$, we have that
\[
c_p(V_m,\mu)\leq d_m,
\]
where
\[
d_m = \frac{1}{\Bigg( \frac{1}{c_p'(W_{0},\mu)^{p^*}} + \frac{1}{\Bigg( |\mu_{v_1}|^p+\frac{1}{\bigg( \frac{1}{c_p'(W_{1},\mu)^{p^*}}+\frac{1}{\Big(\ddots \frac{1}{\big( \frac{1}{c_p'(W_{m-1},\mu)^{p^*}} + \frac{1}{|\mu_{v_m}|^{p^\ast}} \big)^{p/p^{\ast}}} \Big)^{p^{\ast}/p}}\bigg)^{p/p^{\ast}}} \Bigg)^{p^{\ast}/p}}\Bigg)^{1/p^{\ast}}}.
\]
The $d_m$ are the even-numbered convergents of the above almost-classical continued fractions.

The sequence $(d_{m})_m$ is clearly increasing, and we interpret $d$ as
\[
d:=\lim_{m\to\infty}d_{m}.
\]
Thus we have that
\begin{equation}\label{eq-cf2}
c_p(V,\mu) = \lim_{m\to\infty} c_p(V_m,\mu) \leq d.
\end{equation}

On the other hand, let $1\leq l\leq m$. Then we observe that 
\[
\frac{1}{\Bigg( \frac{1}{c_p'(W_{0,m},\mu)^{p^*}} + \frac{1}{\Bigg( |\mu_{v_1}|^p+\frac{1}{\bigg( \frac{1}{c_p'(W_{1,m},\mu)^{p^*}}+\frac{1}{\Big(\ddots \frac{1}{\big( \frac{1}{c_p'(W_{l-1,m},\mu)^{p^*}} + \frac{1}{|\mu_{v_l}|^{p^\ast}} \big)^{p/p^{\ast}}} \Big)^{p^{\ast}/p}}\bigg)^{p/p^{\ast}}} \Bigg)^{p^{\ast}/p}}\Bigg)^{1/p^{\ast}}}
\]
is less or equal than $c_p(V_m,\mu)$. Letting first $m$ and then $l$ tend to infinity, we see that
\[
d\leq  c_p(V,\mu),
\]
which together with \eqref{eq-cf2} implies the result.

The proof in the case of $p=\infty$ is similar.
\end{proof}

\begin{example}\label{ex-comb3}
We consider half the comb tree of Example \ref{ex-comb}; more precisely, let $V=\{(n,k) : n\in \NN_0, 0\leq k\leq n\}$ so that, for $n\geq 0$, every vertex $(n,0)$ has two children, $(n+1,0)$ and $(n+1,1)$, while every vertex $(n,k)$, $k\geq 1$, has a single child $(n+1,k+1)$. When we define $\mu_{(n,k)}=\frac{1}{2^{k/2}}$ for $(n,k)\in V$, then $\mu_{(n,0)} = c_2'(W_n,\mu)=1$ for all $n\geq 0$, and hence the continued fraction for $c_2(V,\mu)^2$ becomes the continued fraction for the golden ratio $\phi$, that is, $c_2(V,\mu)=\sqrt{\phi}$. 
\end{example}

Let us also illustrate the curious modification of $c_p$ into $c_p'$.

\begin{example}\label{ex-comb4}
Let $1<p<\infty$, and consider the tree $V=\NN_0$ with weights $\mu_n$, $n\geq 0$. By Example \ref{ex-unbranched},
\begin{equation}\label{eq-comb4}
c_p(V,\mu)=\Big(\sum_{n=1}^\infty |\mu_n|^p\Big)^{1/p}.
\end{equation}
We obtain the same value by the formula in the result above because all $W_n$ are singletons and thus $\frac{1}{c_p'(W_n,\mu)}=0$, $n\geq 0$.

Now add infinite branches to each vertex of $\NN_0$, so that we arrive at the half-comb tree of the previous example. For $n\geq 0$, let $\mu_{(n,0)}=\mu_n$, and $\mu_{n,k}=1$ for $k\geq 1$. Then this time all $c_p(W_n,\mu)$ are infinite, so that again $\frac{1}{c_p'(W_n,\mu)}=0$, $n\geq 0$, and the new tree has the same constant $c_p$, given by \eqref{eq-comb4}. 

Finally, if we modify the tree $\NN_0$ by only attaching the tree of Example \ref{ex-cuttree} at the root $0$, with weights 1, then $c_p(W_0,\mu)=0$, so that the constant $c_p$ for this tree is zero. 
\end{example}

We end the section with some comments on the constants $c_p$ and $r_p$. The structure of the constants $c_p$ resembles that of the classical continued fractions. Generalized continued fractions of the type that we encounter here do not seem to have been studied in the literature yet (but see Subsection \ref{subs-rootI}). In the case when $p=2$, hence $p/p^\ast=1$, our expressions are similar to so-called branched continued fractions, see \cite{Sie83}, \cite{BoKu98}, which take the form
\[
a_0 + \sum_{k_0=1}^N  
\frac{a_{k_0}}
{ a_{k_0,0}+ \sum\limits_{k_1=1}^N 
\frac{a_{k_0,k_1}}
{a_{k_0,k_1,0}+  \sum\limits_{k_2=1}^N 
\frac{a_{k_0,k_1,k_2}} {a_{k_0,k_1,k_2,0}+ \ldots}}}.
\]
But the number of terms at each step is invariably $N$. And in our continued fractions the sums $\sum$ only appear after another fraction, which leads to the positive homogeneity observed above.

A noteworthy feature of our continued fractions is that they reflect the structure of the tree. In fact, knowing for each $v\in V$ the number of terms that appear successively in the denominators allows to reconstruct the tree. Fractions that reflect a tree structure in a different way have recently also come up in the work of Spier \cite{Spi20}.

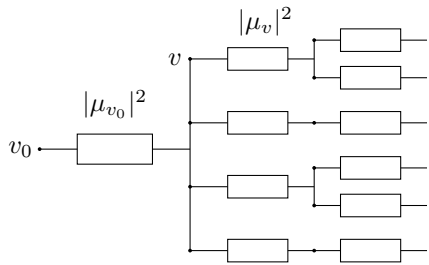
\begin{figure}
\begin{tikzpicture}
\draw[fill] (-2,0) circle (.5pt);
\draw (-2,0) -- (-1.5,0);
\draw (-1.5,-.2) rectangle (-.5,.2);
\draw (-.5,0) -- (0,0);

\draw (0,-1.35) -- (0,1.2);

\draw[fill] (0,1.2) circle (.5pt);
\draw (0,1.2) -- (0.5,1.2);
\draw (.5,1.05) rectangle (1.3,1.35);
\draw (1.3,1.2) -- (1.65,1.2);

\draw[fill] (0,.35) circle (.5pt);
\draw (0,.35) -- (0.5,.35);
\draw (.5,.2) rectangle (1.3,.5);
\draw (1.3,.35) -- (2,.35);

\draw[fill] (0,-.5) circle (.5pt);
\draw (0,-.5) -- (0.5,-.5);
\draw (.5,-.65) rectangle (1.3,-.35);
\draw (1.3,-.5) -- (1.65,-.5);

\draw[fill] (0,-1.35) circle (.5pt);
\draw (0,-1.35) -- (0.5,-1.35);
\draw (.5,-1.5) rectangle (1.3,-1.2);
\draw (1.3,-1.35) -- (2,-1.35);

\draw (1.65,.95) -- (1.65,1.45);

\draw (1.65,-.25) -- (1.65,-.75);

\draw[fill] (1.65,1.45) circle (.5pt);
\draw (1.65,1.45) -- (2,1.45);
\draw (2,1.3) rectangle (2.8,1.6);
\draw (2.8,1.45) -- (3.15,1.45);

\draw[fill] (1.65,.95) circle (.5pt);
\draw (1.65,.95) -- (2,.95);
\draw (2,.8) rectangle (2.8,1.1);
\draw (2.8,.95) -- (3.15,.95);

\draw[fill] (1.65,.35) circle (.5pt);
\draw (2,.2) rectangle (2.8,.5);
\draw (2.8,.35) -- (3.15,.35);
 
\draw[fill] (1.65,-.25) circle (.5pt);
\draw (1.65,-.25) -- (2,.-.25);
\draw (2,-.4) rectangle (2.8,-.1);
\draw (2.8,-.25) -- (3.15,.-.25);
 
\draw[fill] (1.65,-.75) circle (.5pt);
\draw (1.65,-.75) -- (2,-.75);
\draw (2,-.9) rectangle (2.8,-.6);
\draw (2.8,-.75) -- (3.15,-.75);

\draw[fill] (1.65,-1.35) circle (.5pt);
\draw (2,-1.5) rectangle (2.8,-1.2);
\draw (2.8,-1.35) -- (3.15,-1.35);

\node at (-2.25,0) {\footnotesize{$v_0$}};
\node at (-.2,1.2) {\footnotesize{$v$}};

\node at (-1,.6) {\footnotesize{$|\mu_{v_0}|^2$}};
\node at (1,1.7) {\footnotesize{$|\mu_{v}|^2$}};
\end{tikzpicture}
\caption{A weighted tree as an electrical circuit}%
\label{fig-circuit}
\end{figure}

As for the resistance constants, let us assume that the tree is locally finite. Then Theorem \ref{t-rec1} tells us that, for $p=2$,
\[
r_2(V,\mu)^2=|\mu_{v_0}|^2 + \frac{1}{\sum\limits_{v\in \Chi(v_0)}\frac{1}{r_2(V(v),\mu)^{2}}}.
\]
This is nothing but the formula for the resistance in an electrical circuit. More specifically, if we interpret, for any vertex $v\in V$, $|\mu_v|^2$ as the value of a resistor \textit{after} the vertex, see Figure \ref{fig-circuit}, then the formula tells us that $r_2(V,\mu)^2$ is exactly the total resistance of our circuit -- where we need to reconnect the various branches of the circuit ``at infinity'', that is, at the Poisson boundary in the language of harmonic functions on trees, see Section \ref{s-chaosPoisson}. This interpretation has motivated our terminology.

\section{Backward invariance on $\ell^p$ and $c_0$}\label{s-backwinv} 

We now return to the study of the action of the backward shift operator $B$. 

\subsection{Backward invariance on rooted trees of finite length}\label{s-backinv} 
Motivated by the strategy explained at the beginning of the previous section we will first only consider rooted trees of finite length, that is, trees where each branch starting form the root has a finite length (and therefore ends in a leaf). Note that for such trees the constants $c_p(V,\mu)$ and $r_p(V,\mu)$, $1\leq p\leq \infty$, are always finite.

For the results in this section we recall the notion of a backward invariant sequence in Definition \ref{d-backwinv}. We stress that $f$ is in general not a fixed point of the backward shift operator $B$ on $V$. In fact, since $(Bf)(v)=0$ for any leaf of $v$ of $V$, the zero-sequence is the only fixed point of $B$.

For the following result, the reverse H\"older inequality will be essential. 

\begin{theorem}\label{t-charfinp}
Let $V$ be a rooted tree of finite length with root $v_0$ and $\mu=(\mu_v)_{v\in V}$ a weight on $V$. Let $\mathcal{J}$ be the set of all backward invariant sequences $f$ on $V$ with $f(v_0)=1$. 

\emph{(a)} If $1\leq p<\infty$ then
\[
\inf_{f\in \mathcal{J}} \Big(\sum_{v\in V} |f(v)\mu_v|^p\Big)^{1/p}=r_p(V,\mu)\ \text{ and }\ \inf_{f\in \mathcal{J}} \Big(\sum_{\substack{v\in V\\v\neq v_0}} |f(v)\mu_v|^p\Big)^{1/p}=c_p(V,\mu).
\]

\emph{(b)} If $p=\infty$, then 
\[
\inf_{f\in \mathcal{J}} \sup_{v\in V} |f(v)\mu_v|=r_\infty(V,\mu)\ \text{ and } \
\inf_{f\in \mathcal{J}} \sup_{\substack{v\in V\\v\neq v_0}} |f(v)\mu_v|=c_\infty(V,\mu).
\] 
\end{theorem}

\begin{proof}
First, by the definition of the constants $r_p$ it suffices to show the result for the $c_p$.

(a) Let $1<p<\infty$. We proceed by induction on $m$, the length of $V$. For $m=0$ the assertion is trivial, see Remark \ref{rem-contfr}(a). Consider $m=1$. Since we are only interested in sequences $f$ of minimal norm we may assume that $f(v)\geq 0$ for all $v\in \Chi(v_0)$, so that $\|f\chi_{\Chi(v_0)}\|_1=\sum_{v\in \Chi(v_0)}f(v)=f(v_0)=1$. It thus follows from the reverse H\"older inequality, see Section \ref{s-notation}, that
\begin{equation}\label{eq-hoel0}
\begin{split}
\inf_{f\in \mathcal{J}} \Big(\sum_{\substack{v\in V\\v\neq v_0}} |f(v)\mu_v|^p\Big)^{1/p}&=\inf_{\|f\chi_{\Chi(v_0)}\|_1=1} \Big(\sum_{v\in \Chi(v_0)}|f(v)\mu_v|^p\Big)^{1/p}\\
&= \frac{1}{\big(\sum\limits_{v\in \Chi(v_0)}\frac{1}{|\mu_v|^{p^{\ast}}}\big)^{1/p^\ast}} = c_p(V,\mu).
\end{split}
\end{equation}

Let now $m\geq 1$, and suppose that we have obtained the result for all rooted trees of length $m$. Let $V$ be a rooted tree of length $m+1$.
Let $f$ be a backward invariant sequence on $V$ with $f(v_0)=1$. Since we are only interested in sequences of minimal norm, we may assume that if $f(v)=0$ for some $v\in V$ then $f(u)=0$ for all $u\in \Chi(v)$. Thus we have that
\begin{align*}
\sum_{\substack{v\in V\\v\neq v_0}} |f(v)\mu_v|^p &= \sum_{n=1}^{m-1} \Big(\sum_{v\in \gen_n} |f(v)\mu_v|^p\Big)+ \sum_{v\in \gen_{m}}|f(v)\mu_v|^p+\sum_{v\in \gen_{m+1}}|f(v)\mu_v|^p\\
&=  \sum_{n=1}^{m-1} \Big(\sum_{v\in \gen_n} |f(v)\mu_v|^p\Big)+ \sum_{\substack{v\in \gen_{m}\\v \text{ a leaf}}}|f(v)\mu_v|^p\\
&\hspace{13em}+\sum_{\substack{v\in \gen_{m}\\v \text{ not a leaf}}}\Big(|f(v)\mu_v|^p+\sum_{u\in \Chi(v)}|f(u)\mu_u|^p\Big).
\end{align*}

Consider $v\in \gen_{m}$ that is not a leaf. Since $f$ is backward invariant, we have that
\[
f(v)=\sum_{u\in \Chi(v)}f(u).
\]
Hence there are scalars $\alpha_u$ such that, for all $u\in \Chi(v)$,
\[
f(u) = \alpha_u f(v)\quad\text{and}\quad \sum_{u\in\Chi(v)} \alpha_u=1,
\]
which, by our assumption on $f$, is also true if $f(v)=0$. Thus,
\[
\sum_{u\in \Chi(v)}|f(u)\mu_u|^p = |f(v)|^p \sum_{u\in \Chi(v)}|\alpha_u\mu_u|^p.
\]
Once more, since we are only interested in sequences $f$ of minimal norm we may assume that $\alpha_u\geq 0$ for all $u\in \Chi(v)$, so that $\|(\alpha_u)_u\|_1=1$. On the other hand, it follows from another application of the reverse H\"older inequality that
\begin{equation}\label{eq-hoel}
\inf_{\|(\alpha_u)_u\|_1=1} \sum_{u\in \Chi(v)}|\alpha_u\mu_u|^p = \frac{1}{\Big(\sum\limits_{u\in \Chi(v)}\frac{1}{|\mu_u|^{p^{\ast}}}\Big)^{p/p^\ast}}.
\end{equation}

Altogether we have shown that
\begin{align*}
\inf_{f\in \mathcal{J}}\Big(\sum_{\substack{v\in V\\v\neq v_0}} |f(v)\mu_v|^p\Big)^{1/p} &= \inf_{f\in \mathcal{J}_m}\bigg(\sum_{n=1}^{m-1} \Big(\sum_{v\in \gen_n} |f(v)\mu_v|^p\Big)+
\sum_{\substack{v\in \gen_{m}\\v \text{ a leaf}}}|f(v)\mu_v|^p\\
&\hspace{5em}+ \sum_{\substack{v\in \gen_{m}\\v \text{ not a leaf}}}|f(v)|^p\Big(|\mu_v|^p+\frac{1}{\Big(\sum\limits_{u\in \Chi(v)}\frac{1}{|\mu_u|^{p^{\ast}}}\Big)^{p/p^\ast}}\Big)\bigg)^{1/p}\\
&= \inf_{f\in \mathcal{J}_m}\bigg(\sum_{n=1}^{m-1} \Big(\sum_{v\in \gen_n} |f(v)\mu_v|^p\Big)\\
&\hspace{5em}+\sum_{v\in \gen_{m}}|f(v)|^p\Big(|\mu_v|^p+\frac{1}{\Big(\sum\limits_{u\in \Chi(v)}\frac{1}{|\mu_u|^{p^{\ast}}}\Big)^{p/p^\ast}}\Big)\bigg)^{1/p},
\end{align*}
where $\mathcal{J}_m$ denotes the set of backward invariant sequences on $V_m=\bigcup_{n=0}^m \gen_n$ and where we have used the convention of Remark \ref{rem-contfr}(a).

Now let $\mu'$ be the weight on $V_m$ defined in Remark \ref{rem-contfr}(c). It follows with the induction hypothesis and that remark that
\[
\inf_{f\in \mathcal{J}}\Big(\sum_{\substack{v\in V\\v\neq v_0}} |f(v)\mu_v|^p\Big)^{1/p} = c_p(V_m,\mu') = c_p(V,\mu), 
\]
as had to be shown.

For $p=1$ one can proceed similarly. But a simple reflection shows the result to be obvious: the best way to spread the mass 1 is to pass it completely on along a branch with smallest possible $\ell^1$-norms of the weights. 

(b) The proof for $p=\infty$ is the same as that for $p<\infty$, where we now replace the second infimum in \eqref{eq-hoel0} by
\[
\inf_{\|f\chi_{\Chi(v_0)}\|_1=1} \sup_{v\in \Chi(v_0)}|f(v)\mu_v| = \frac{1}{\sum\limits_{v\in \Chi(v_0)}\frac{1}{|\mu_v|}},
\]
and similarly for \eqref{eq-hoel}. 
\end{proof}

We give two simple examples to see that the infimum in the formulas is not necessarily attained. 

\begin{example}\label{ex-att}
Let $V$ be the tree of length 1 given by the root and an infinite number of children.

(a) Let $p=2$ and $\mu_v=1$ for all $v\in V$. Sending the mass 1 uniformly to $N$ children leads to a backward invariant sequence of norm $(1+\frac{1}{N})^{1/2}$, which tends to 1, while there is clearly no backward invariant sequence of norm 1 that gives the root the value 1. Note, however, that the backward shift $B$ is not defined on $\ell^p(V,\mu)$, see Section \ref{s-notation}.

(b) Let $p=1$ and choose weights $\mu_v$ on the children $v$ of the root with $\inf_v|\mu_v|>0$ such that the infimum is not attained. Then there is no backward invariant sequence of norm  $r_1(V,\mu)$, while here $B$ is indeed an operator on $\ell^1(V,\mu)$.
\end{example} 

If $p\neq 1$, it is in fact the rather weak continuity assumption on $B$ to be defined on $\ell^p(V,\mu)$ or $c_0(V,\mu)$ that ensures the existence of a backward invariant sequence of minimal norm. Moreover, in $\ell^p(V,\mu)$, the minimal backward invariant sequence is unique, and one can state its form explicitly. 

\begin{theorem}\label{t-charfinpmin}
Let $V$ be a rooted tree of finite length with root $v_0$ and $\mu=(\mu_v)_{v\in V}$ a weight on $V$. 

\emph{(a)} Let $1< p<\infty$, and suppose that $B$ is defined on $\ell^p(V,\mu)$. Then there exists a unique backward invariant sequence $f$ on $V$ with $f(v_0)=1$ such that
\[
\Big(\sum_{v\in V} |f(v)\mu_v|^p\Big)^{1/p}=r_p(V,\mu),
\]
or such that
\[
\Big(\sum_{\substack{v\in V\\v\neq v_0}} |f(v)\mu_v|^p\Big)^{1/p}=c_p(V,\mu).
\]
For any $v=v_n\in \emph{\gen}_n$, $n\geq 1$, $f$ is given by
\[
f(v) = \prod_{k=1}^n\frac{r_p(V(v_k),\mu)^{-p^\ast}}{\sum\limits_{u\in\Chi(\prt(v_{k}))}r_p(V(u),\mu)^{-p^\ast}}>0,
\]
where $(v_0,v_1,\ldots,v_n)$ is the path from $v_0$ to $v_n$.

\emph{(b)} Suppose that $B$ is defined on $c_0(V,\mu)$ or, equivalently, on $\ell^\infty(V,\mu)$. Then there exists a backward invariant sequence $f$ on $V$ with $f(v_0)=1$ such that
\[
\sup_{v\in V} |f(v)\mu_v|=r_\infty(V,\mu)\ \text{ and } \ \sup_{\substack{v\in V\\v\neq v_0}} |f(v)\mu_v|=c_\infty(V,\mu).
\]
One such sequence $f$ is given, for any $v=v_n\in \emph{\gen}_n$, $n\geq 1$, by
\[
f(v) = \prod_{k=1}^n\frac{r_\infty(V(v_k),\mu)^{-1}}{\sum\limits_{u\in\Chi(\prt(v_{k}))}r_\infty(V(u),\mu)^{-1}}>0,
\]
where $(v_0,v_1,\ldots,v_n)$ is the path from $v_0$ to $v_n$.
\end{theorem}

We note here for a better understanding that, in (a), $\sum_{u\in\Chi(\prt(v_{k}))}r_p(V(u),\mu)^{-p^\ast}<\infty$ and $r_p(V(v_k),\mu)^{-p^\ast}>0$ for all $k$ because $r_p(V(u),\mu)\geq |\mu_u|$ for all $u$ and since $B$ is defined on $\ell^p(V,\mu)$; see Remark \ref{rem-contfr}(c). And similarly for (b).

\begin{proof} The result follows from an analysis of the proof of Theorem \ref{t-charfinp}. 

(a) We first show existence, where it suffices to consider the constant $c_p$. We proceed by induction on the length $m$ of the tree. For $m=0$ there is nothing to prove. Suppose that the claim has been proved for $m\geq 0$. Let $V$ be a tree of length $m+1$, and let $\mu$ be a weight on $V$ such that $B$ is defined on $\ell^p(V,\mu)$. We associate to $\mu$ a new weight $\mu'$ on $V_m$ as in Remark \ref{rem-contfr}(c), see \eqref{eq-formula}. Since $|\mu'_v| \geq |\mu_v|$ for all $v\in V_{m}$, $B$ is also defined on $\ell^p(V_m,\mu')$. Thus, by the induction hypothesis, a backward invariant sequence $f_m$ on $V_m$ with $f_m(v_0)=1$ and of minimal norm $(\sum_{v\in V_m,v\neq v_0}|f_m(v)\mu'_v|^p)^{1/p}$ is given by the following: for any $v=v_n\in \gen_n$, $1\leq n\leq m$, 
\[
f_{m}(v)=  \prod_{k=1}^n\frac{r_p(V_m(v_k),\mu')^{-p^\ast}}{\sum\limits_{u\in\Chi(\prt(v_{k}))}r_p(V_m(u),\mu')^{-p^\ast}}>0,
\]
where $(v_0,v_1,\ldots,v_n)$ is the path from $v_0$ to $v_n$; to avoid confusion note here that, for any $v\in V_m$ with $v\in\gen_n$, we have that
\begin{equation}\label{eq-Vmv}
V_m(v) = V_m\cap {\textstyle\bigcup_{k\geq 0}} \Chi^k(v)=V(v)_{m-n}
\end{equation}
is the set of descendants of $v$ in $V_m$. 

Let $\mathcal{J}_m$ and $\mathcal{J}$ denote the set of all backward invariant sequences $f$ on $V_{m}$ and $V$, respectively, with $f(v_0)=1$. Then we have seen in the proof of Theorem \ref{t-charfinp} that, for $m\geq 1$,
\begin{equation}\label{eq-backinv}
\begin{split}
\inf_{f\in \mathcal{J}}\Big(\sum_{\substack{v\in V\\v\neq v_0}} |f(v)\mu_v|^p\Big)^{1/p}&= \inf_{f\in \mathcal{J}_m}\bigg(\sum_{n=1}^{m-1} \Big(\sum_{v\in \gen_n} |f(v)\mu_v|^p\Big)\\
&\hspace{3em}+\sum_{v\in \gen_{m}}|f(v)|^p\Big(|\mu_v|^p+\frac{1}{\Big(\sum\limits_{u\in \Chi(v)}\frac{1}{|\mu_u|^{p^{\ast}}}\Big)^{p/p^\ast}}\Big)\bigg)^{1/p}.
\end{split}
\end{equation}
By the definition of $\mu'$, the infimum on the right-hand side is attained at $f=f_{m}$. 

Now, for any $v=v_m\in\gen_m$ that is not a leaf of $V$, the reverse H\"older inequality implies that there is a positive sequence $(\alpha_u)_{u\in \Chi(v_m)}$ with $\sum_{u\in \Chi(v_m)}\alpha_u=1$ such that
\begin{equation}\label{eq-hoel2}
\sum_{u\in \Chi(v_m)}|\alpha_u\mu_u|^p = \frac{1}{\Big(\sum\limits_{u\in \Chi(v_m)}\frac{1}{|\mu_u|^{p^{\ast}}}\Big)^{p/p^\ast}}.
\end{equation}
One such sequence is given by
\[
\alpha_u = \frac{|\mu_u|^{-p^\ast}}{\sum\limits_{w\in \Chi(v_m)}|\mu_w|^{-p^\ast}}>0,\quad u\in \Chi(v_m)
\]
(and in fact it is the only such sequence). See \cite[Theorem 13 and p. 119]{HLP34}, \cite[Lemma 4.2]{GrPa21} for these statements; and it is here that we have used the fact that $B:\ell^p(V,\mu)\to\KK^V$ is an operator: it implies that $\sum_{u\in \Chi(v_m)}\frac{1}{|\mu_u|^{p^{\ast}}}<\infty$.

Now let us first assume that $m\geq 1$. Then, as a consequence of the above, the proof of Theorem \ref{t-charfinp} shows that the infimum on the left-hand side of \eqref{eq-backinv} is attained at the following sequence $f$. If $v=v_n\in \gen_{n}$, $1\leq n\leq m$, then
\begin{align*}
f(v) := f_{m}(v) &= \prod_{k=1}^n\frac{r_p(V_m(v_k),\mu')^{-p^\ast}}{\sum\limits_{u\in\Chi(\prt(v_{k}))}r_p(V_m(u),\mu')^{-p^\ast}}\\
&= \prod_{k=1}^n\frac{r_p(V(v_k),\mu)^{-p^\ast}}{\sum\limits_{u\in\Chi(\prt(v_{k}))}r_p(V(u),\mu)^{-p^\ast}}>0,
\end{align*}
where we have used Remark \ref{rem-contfr}(d). 

If $v=v_{m+1}\in \gen_{m+1}$, which implies that $v_m:=\prt(v_{m+1})$ is not a leaf of $V$, then, by the proof of Theorem \ref{t-charfinp},
\begin{align*}
f(v) &:= \alpha_{v_{m+1}} f_{m}(v_m) = \frac{|\mu_{v_{m+1}}|^{-p^\ast}}{\sum\limits_{u\in \Chi(v_m)}|\mu_u|^{-p^\ast}}\prod_{k=1}^m\frac{r_p(V_m(v_k),\mu')^{-p^\ast}}{\sum\limits_{u\in\Chi(\prt(v_{k}))}r_p(V_m(u),\mu')^{-p^\ast}}\\
&= \frac{r_p(V(v_{m+1}),\mu)^{-p^\ast}}{\sum\limits_{u\in\Chi(\prt(v_{m+1}))}r_p(V(u),\mu)^{-p^\ast}}\prod_{k=1}^m\frac{r_p(V(v_k),\mu)^{-p^\ast}}{\sum\limits_{u\in\Chi(\prt(v_{k}))}r_p(V(u),\mu)^{-p^\ast}}\\
&= \prod_{k=1}^{m+1}\frac{r_p(V(v_k),\mu)^{-p^\ast}}{\sum\limits_{u\in\Chi(\prt(v_{k}))}r_p(V(u),\mu)^{-p^\ast}}>0,
\end{align*}
where we have used Remark \ref{rem-contfr}(d) again.

Finally, if $m=0$, only the case of $v=v_{m+1}\in \gen_{m+1}$ is relevant, and the previous equations provide $f(v)$ when we take empty products as 1.

Altogether we have proved the existence part of (a). Now, since the set $\mathcal{J}$ of backward invariant sequences $f$ on $V$ with $f(v_0)=1$ is convex, the uniqueness of a backward invariant sequence of minimal norm follows from the fact that the norm in $\ell^p(V\setminus\{v_0\},\mu)$ is strictly convex.

(b)  The proof of existence for $p = \infty$ is the same as that for $p < \infty$, where we replace \eqref{eq-hoel2} by
\[
\sup_{u\in \Chi(v_m)}|\alpha_u\mu_u| = \frac{1}{\sum\limits_{u\in \Chi(v_m)}\frac{1}{|\mu_u|}},
\]
where $\sum_{u\in \Chi(v_m)}\frac{1}{|\mu_u|}<\infty$ since $B:\ell^\infty(V,\mu)\to\KK^V$ is an operator. A positive 
sequence $(\alpha_u)_{u\in \Chi(v_m)}$ with $\sum_{u\in \Chi(v_m)}\alpha_u=1$ such that \eqref{eq-hoel2} holds is given by 
\[
\alpha_u = \frac{|\mu_u|^{-1}}{\sum\limits_{w\in \Chi(v_m)}|\mu_w|^{-1}}>0,\quad u\in \Chi(v_m).
\]
\end{proof}

Not surprisingly, for $1<p<\infty$, the minimal backward invariant sequence has all its entries non-zero: only by distributing the initial mass 1 over all vertices can one achieve a minimal norm. 

\begin{remark}\label{r-decr}
The backward invariant sequences $f$ of Theorem \ref{t-charfinpmin}(a) and (b) have the property that $(f(v_n))_{n=0,\ldots,m}$ is decreasing for every branch $(v_0,\ldots,v_m)$ starting from $v_0$. This follows from backward invariance and the positivity of the values.
\end{remark}

\begin{example}\label{ex-problemc0}
We stress that, in the case of $p=\infty$, a minimal sequence $f$ might not belong to $c_0(V,\mu)$. Let $V$ be the tree of length 1 where the root $v_0$ has infinitely many children $u_n$, $n\geq 1$, and let $\mu_{v_0}=1$, $\mu_{u_n}=2^{n}$, $n\geq 1$. Note that $B$ is defined on $c_0(V,\mu)$. Moreover, $r_\infty(V,\mu)=c_\infty(V,\mu)=1$. Then there is a unique backward invariant sequence $f$ with $f(v_0)=1$ 
such that $\sup_{v\in V} |f(v)\mu_v|=1$, and there is a unique backward invariant sequence $f$ with $f(v_0)=1$ such that $\sup_{v\in V,v\neq v_0} |f(v)\mu_v|=1$, and both are given by $f(u_n)=\frac{1}{2^n}$, $n\geq 1$. However, $f\notin c_0(V,\mu)$.

In addition, backward invariant sequences of minimal norm might not be unique. To see this, consider the binary tree $V$ of length 2, where the root $v_0$ and both its children have two children, and let $\mu$ be identically 1. Then $c_\infty(V,\mu)=\frac{1}{2}$, and there are infinitely many backward invariant sequences $f$ with $f(v_0)=1$ and $\sup_{v\in V, v\neq v_0} |f(v)\mu_v|=\frac{1}{2}$: they may differ in their values at generation 2. The same then also holds for $r_\infty$.
\end{example}

\subsection{Backward invariance on arbitrary rooted trees}\label{s-backinvarb}

We will now extend the previous results to rooted trees of arbitrary length, with or without leaves, where we start with the case of $1<p<\infty$.

\begin{theorem}\label{t-charpinv}
Let $V$ be a rooted tree with root $v_0$ and $\mu=(\mu_v)_{v\in V}$ a weight on $V$. Let $1< p<\infty$, and suppose that $B$ is defined on $\ell^p(V,\mu)$. Let $\mathcal{J}$ be the set of all backward invariant sequences $f$ on $V$ with $f(v_0)=1$.

\emph{(a)} Then
\[
\inf_{f\in \mathcal{J}}\Big(\sum_{v\in V} |f(v)\mu_v|^p\Big)^{1/p}= r_p(V,\mu)\ \text{ and } \ \inf_{f\in \mathcal{J}}\Big(\sum_{\substack{v\in V\\v\neq v_0}} |f(v)\mu_v|^p\Big)^{1/p}= c_p(V,\mu).
\]

\emph{(b)} If $r_p(V,\mu)<\infty$ or, equivalently, $c_p(V,\mu)<\infty$, then there exists a unique backward invariant sequence $f\in\ell^p(V,\mu)$ with $f(v_0)=1$ that attains the infima.
This sequence $f$ is given by 
\[
f(v)=\begin{cases} 
\displaystyle\prod_{k=1}^n\frac{r_p(V(v_k),\mu)^{-p^\ast}}{\sum\limits_{u\in\Chi(\prt(v_{k}))}r_p(V(u),\mu)^{-p^\ast}}>0,&\text{if } r_p(V(v),\mu)<\infty,\\
0, &\text{if } r_p(V(v),\mu)=\infty,\\
\end{cases}
\]
where $n\geq 1$ is such that $v=v_n\in \emph{\gen}_n$ and $(v_0,v_1,\ldots,v_n)$ is the path leading from $v_0$ to $v_n$.
\end{theorem}

\begin{proof} (a) By the definition of $r_p$ we need only prove the assertion for $c_p$. We start with a preliminary inequality. Let $f$ be a backward invariant sequence on $V$ with $f(v_0)=1$. Since, for any $m\geq 1$, the restriction of $f$ to $V_m$ is backward invariant on $V_m$, Theorem \ref{t-charfinp} implies that
\[
c_p(V_m,\mu) \leq \Big(\sum_{\substack{v\in V_m\\v\neq v_0}} |f(v)\mu_v|^p\Big)^{1/p}.
\]
In view of Remark \ref{rem-contfr}(c) this implies that
\begin{equation}\label{eq-half}
c_p(V,\mu)\leq  \Big(\sum_{\substack{v\in V\\v\neq v_0}} |f(v)\mu_v|^p\Big)^{1/p}.
\end{equation}

Thus, if $c_p(V,\mu)=\infty$ nothing more needs to be proved. So suppose now that $c_p(V,\mu)<\infty$. Applying Theorem \ref{t-charfinp}(a) to the trees $V_m$, $m\geq 1$, we obtain backward invariant sequences $f_{m}$ on $V_m$ with $f_{m}(v_0)=1$ and 
\[
\Big(\sum_{\substack{v\in V_m\\v\neq v_0}} |f_m(v)\mu_v|^p\Big)^{1/p} \leq c_p(V_m,\mu)+\tfrac{1}{m}\leq c_p(V,\mu)+\tfrac{1}{m}.
\]
We extend each $f_{m}$ to a sequence of $\ell^p(V,\mu)$ by assigning to it the value 0 outside $V_m$, calling it again $f_{m}$. Then $(f_{m})_{m}$ is a bounded sequence in $\ell^p(V,\mu)$, which we regard as the dual of the separable space $\ell^{p*}(V,1/\mu)$. Therefore, by Banach-Alaoglu, $(f_{m})_m$ has a weak$^\ast$-convergent subsequence $(f_{m_k})_k$ with limit $f$, say. Then $f(v_0)=1$ and
\begin{equation}\label{eq-fixpt2}
\Big(\sum_{\substack{v\in V\\v\neq v_0}} |f(v)\mu_v|^p\Big)^{1/p} \leq c_p(V,\mu)<\infty,
\end{equation}
so that $f\in\ell^p(V,\mu)$. 

In view of \eqref{eq-half}, it remains to show that $f$ is backward invariant on $V$. For this, let $v\in V$ be not a leaf of $V$. By hypothesis on $B$, the functional $\varphi(g):= \sum_{u\in\Chi(v)} g(u)$ is defined and therefore continuous on $\ell^{p}(\Chi(v),\mu)$, hence weak$^\ast$-continuous. We also have that $\varphi(f_{m_k}) = f_{m_k}(v)$ for all sufficiently large $k$. Therefore $\sum_{u\in\Chi(v)} f(u) = f(v)$, which had to be shown.

(b) The uniqueness follows again from the strict convexity of $\ell^p$-norms.

For the formula for $f$, we distinguish two cases. First assume that $r_p(V(v),\mu)<\infty$ for all $v\in V$. Applying Theorem \ref{t-charfinpmin}(a) to the subtree $V_m$ of finite length, $m\geq 1$, we see that there exists a unique backward invariant sequence $f_m$ on $V_m$ with $f_m(v_0)=1$ such that
\begin{equation}\label{eq-fpB1}
\Big(\sum_{\substack{v\in V_m\\v\neq v_0}} |f_m(v)\mu_v|^p\Big)^{1/p}=c_p(V_m,\mu) \leq c_p(V,\mu),
\end{equation}
and for any $v=v_n\in\gen_n$, $1\leq n\leq m$, $f_m$ satisfies
\[
f_m(v) = \prod_{k=1}^{n}\frac{r_p(V_m(v_k),\mu)^{-p^\ast}}{\sum\limits_{u\in\Chi(\prt(v_{k}))}r_p(V_m(u),\mu)^{-p^\ast}},
\]
where $(v_0,\ldots,v_n)$ is the path from $v_0$ to $v_{n}$. Extend $f_m$ to $V$ by defining it to be zero outside $V_m$.

Now, $\sum_{u\in\Chi(\prt(v_{k}))}r_p(V_m(u),\mu)^{-p^\ast}<\infty$ for $k=1,\ldots,n$; see the comment after the statement of Theorem \ref{t-charfinpmin}. If we let $m\to\infty$, by Remark \ref{rem-contfr}(d), this sum decreases. Since, by assumption,  $r_p(V(u),\mu)<\infty$ for all $u\in V$, the sum tends to a strictly positive limit, as does $r_p(V_m(v_k),\mu)^{-p^\ast}$. Altogether, for any $v\in V$,
\[
f(v):=\lim_{m\to\infty} f_m(v) = \prod_{k=1}^{n}\frac{r_p(V(v_k),\mu)^{-p^\ast}}{\sum\limits_{u\in\Chi(\prt(v_{k}))}r_p(V(u),\mu)^{-p^\ast}}>0
\]
exists. Then $f(v_0)=1$, and the special form of the formula for $f$ implies that it is backward invariant on $V$. Fatou's lemma, applied to \eqref{eq-fpB1}, shows that \eqref{eq-fixpt2} holds, so that $f\in \ell^p(V,\mu)$ has minimal norm in $\mathcal{J}$, which proves (b) in the first case.

Secondly, suppose that $r_p(V(v),\mu)=\infty$ for certain $v\in V$. By Corollary \ref{c-subtrees}(b) we also have that $r_p(V(u),\mu)=\infty$ for all descendants $u$ of $v$. It follows from (a) that the only backward invariant sequence on the subtree $V(v)$ and in $\ell^p(V(v),\mu)$ is the zero sequence. Define a new tree $\widetilde{V}$ by deleting from $V$ all vertices $v$ with $r_p(V(v),\mu)=\infty$ (and therefore also all its descendants); then $\widetilde{V}$ also has $v_0$ as root. It follows that the backward invariant sequence $f$ on $V$ with $f(v_0)=1$ of minimal norm vanishes on $V\setminus \widetilde{V}$, while on $\widetilde{V}$ it coincides with the minimal backward invariant sequence on $\widetilde{V}$ with value 1 at $v_0$.

Now, by Theorem \ref{t-rec1}(b), we have for any $v\in \widetilde{V}$,
\[
r_p(\widetilde{V}(v),\mu)= r_p(V(v),\mu)<\infty,
\]
where we have used that, by Corollary \ref{c-subtrees}(b), if $r_p(V(v),\mu)<\infty$ then either $v$ is a leaf of $V$ or there is at least one child $u$ of $v$ with $r_p(V(u),\mu)<\infty$. Thus we can apply the first case to the tree $\widetilde{V}$, and thereby obtain (b) also in the second case.
\end{proof}

\begin{example}\label{ex-inffin}
Consider the two-branched rooted tree $V$ of Example \ref{ex-twobranched}, where along the upper branch we choose as weights $2^{-n/2}$ in generation $n\geq 1$, while on the lower branch all weights are 1. Then $r_2(V,\mu)=\sqrt{2}.$ The unique backward invariant sequence $f$ with $f(v_0)=1$ of minimal norm is given by the value 1 along the upper branch, while the lower branch carries the value 0. 
\end{example} 

For the space $X=\ell^1(V,\mu)$, the analogue of the previous theorem is rather obvious, see the comment in the proof of Theorem \ref{t-charfinp}(a) for $p=1$. Note, however, that the infimum need not be attained even if $B$ is an operator on $\ell^1(V,\mu)$, see Example \ref{ex-att}.

\begin{theorem}\label{t-char1inv}
Let $V$ be a rooted tree with root $v_0$ and $\mu=(\mu_v)_{v\in V}$ a weight on $V$. Let $\mathcal{J}$ be the set of all backward invariant sequences $f$ on $V$ with $f(v_0)=1$.

\emph{(a)} Then
\[
\inf_{f\in \mathcal{J}}\sum_{v\in V} |f(v)\mu_v| = r_1(V,\mu) \ \text{ and } \
\inf_{f\in \mathcal{J}}\sum_{\substack{v\in V\\v\neq v_0}} |f(v)\mu_v| = c_1(V,\mu).
\]

\emph{(b)} If $r_1(V,\mu) < \infty$ or, equivalently, $c_1(V,\mu) < \infty$, then a backward invariant sequence $f$ in $\ell^1(V,\mu)$ with $f(v_0)=1$ is given by
\[
f=\sum_{n=0}^\infty e_{v_n},
\]
where $(v_0,v_1,v_2,\ldots)$ is a branch starting from $v_0$ with $\sum_{n=0}^\infty |\mu_{v_n}| <\infty$.
\end{theorem}

The remaining case is that of $c_0(V,\mu)$. Here, however, we cannot guarantee that the backward invariant sequences are in $c_0(V,\mu)$; they might end up in $\ell^\infty(V,\mu)$. Just consider the tree $\NN_0$ with unit weights. And a backward invariant sequence of minimal norm might not be unique, see Example \ref{ex-problemc0}. But otherwise the proof of the following theorem proceeds along the same lines as that of Theorem \ref{t-charpinv}.

\begin{theorem}\label{t-charinfinv}
Let $V$ be a rooted tree with root $v_0$ and $\mu=(\mu_v)_{v\in V}$ a weight on $V$. Suppose that $B$ is defined on $c_0(V,\mu)$ or, equivalently, on $\ell^\infty(V,\mu)$. Let $\mathcal{J}$ be the set of all backward invariant sequences $f$ on $V$ with $f(v_0)=1$.

\emph{(a)} Then
\[
\inf_{f\in \mathcal{J}} \sup_{v\in V} |f(v)\mu_v| = r_\infty(V,\mu)\ \text{ and } \
\inf_{f\in \mathcal{J}} \sup_{\substack{v\in V\\v\neq v_0}} |f(v)\mu_v| = c_\infty(V,\mu).
\]

\emph{(b)} If $r_\infty(V,\mu)<\infty$ or, equivalently, $c_\infty(V,\mu)<\infty$, then there exists a backward invariant sequence $f\in\ell^\infty(V,\mu)$ with $f(v_0)=1$ that attains the infimum for $c_\infty(V,\mu)$ and therefore that for $r_\infty(V,\mu)$. One such sequence $f$ is given by 
\[
f(v)=\begin{cases} 
\displaystyle\prod_{k=1}^n\frac{r_\infty(V(v_k),\mu)^{-1}}{\sum\limits_{u\in\Chi(\prt(v_{k}))}r_\infty(V(u),\mu)^{-1}}>0,&\text{if } r_\infty(V(v),\mu)<\infty,\\
0, &\text{if } r_\infty(V(v),\mu)=\infty,\\
\end{cases}
\]
where $n\geq 1$ is such that $v=v_n\in \emph{\gen}_n$ and $(v_0,v_1,\ldots,v_n)$ is the path leading from $v_0$ to $v_n$.
\end{theorem}

\begin{remark}\label{r-decr2}
The backward invariant sequences $f$ mentioned in parts (b) of the three theorems of this subsection have the property that $(f(v_n))_{n\geq 0}$ is decreasing for every branch $(v_0,v_1,v_2,\ldots)$ starting from $v_0$; see Remark \ref{r-decr}.
\end{remark}

\section{A characterization of chaos on $\ell^p$ and $c_0$ over rooted trees}\label{s-chaosrooted}

The previous section allows us rather easily to characterize chaotic weighted backward shifts on spaces of type $\ell^p$ and $c_0$. Since chaotic operators are in particular hypercyclic, the underlying tree cannot have leaves, see \cite[Remark 4.1]{GrPa21}.

\vspace{\baselineskip}
\textit{Throughout the remainder of this paper (Sections \ref{s-chaosrooted} to \ref{s-chaosPoisson}) we will assume that the trees have no leaves.}
\vspace{\baselineskip}

Also, it turns out that only the finiteness of the constants $c_p$ or $r_p$ play a r\^ole, not their values. We will therefore, from now on, give preference to the constant $c_p$.

We first consider rooted trees $V$. As usual, we start by considering the unweighted backward shift $B$ on weighted sequence spaces.

\begin{theorem}\label{t-charpmu}
Let $V$ be a rooted tree and $\mu=(\mu_v)_{v\in V}$ a weight on $V$. Suppose that the backward shift $B$ is an operator on $\ell^p(V,\mu)$, $1< p<\infty$. Then $B$ is chaotic on $\ell^p(V,\mu)$ if and only if, for every $v\in V$,
\begin{equation}\label{eq-chaos}
c_p(V(v),\mu) < \infty.
\end{equation}
\end{theorem}

\begin{proof}
By Theorem \ref{chaos-fixedpoint}, $B$ is chaotic on $\ell^p(V,\mu)$ if and only if, for any $v\in V$, there is a fixed point $f$ for $B$ in $\ell^p(V(v),\mu)$ with $f(v)=1$; note that \eqref{eq-fpleft} is automatically satisfied for rooted trees. Since $V$ has no leaves, being a fixed point for $B$ is equivalent to being backward invariant. Thus the result follows from Theorem \ref{t-charpinv}(a).
\end{proof}

\begin{remark}\label{rem-alt}
Example \ref{ex-inffin} tells us that if a weighted tree has a finite continued fraction then there still may exist a vertex whose subtree has an infinite continued fraction. Thus it is not sufficient to demand \eqref{eq-chaos} only for the root.
\end{remark}

In the same way, Theorem \ref{t-char1inv} implies the following.

\begin{theorem}\label{t-char1mu}
Let $V$ be a rooted tree and $\mu=(\mu_v)_{v\in V}$ a weight on $V$. Suppose that the backward shift $B$ is an operator on $\ell^1(V,\mu)$. 
Then $B$ is chaotic on $\ell^1(V,\mu)$ if and only if, for every $v\in V$, there exists a branch $(v,v_1,v_2,\ldots)$ starting from $v$ such that
\[
\sum_{n=1}^\infty |\mu_{v_n}| <\infty.
\]
\end{theorem}

In the case of $X=c_0(V,\mu)$ one cannot obtain fixed points of $B$ via Theorem \ref{t-charinfinv}: the only fixed points for $B$ on the tree $\NN_0$ with unit weights are the constant sequences. As a consequence we obtain a slightly more involved characterization.

\begin{theorem}\label{t-char0mu}
Let $V$ be a rooted tree and $\mu=(\mu_v)_{v\in V}$ a weight on $V$. Suppose that the backward shift $B$ is an operator on $c_0(V,\mu)$. Then the following assertions are equivalent:
\begin{enumerate}
\item[\rm (i)] $B$ is chaotic on $c_0(V,\mu)$;
\item[\rm (ii)] there is a weight $\widetilde{\mu}$ on $V$ with $(\mu_v/\widetilde{\mu}_v)_v\in c_0(V)$ such that, for every $v\in V$, 
\[
c_\infty(V(v),\widetilde{\mu}) < \infty;
\]
\item[\rm (iii)] for every $v\in V$ there is a weight $\widetilde{\mu}$ on $V(v)$ with $(\mu_u/\widetilde{\mu}_u)_u\in c_0(V(v))$ such that
\[
c_\infty(V(v),\widetilde{\mu}) < \infty.
\]
\end{enumerate}
\end{theorem}

\begin{proof} Suppose that (i) holds. Since $B$ is chaotic on $c_0(V,\mu)$, it admits by Theorem \ref{chaos-fixedpoint} a universal fixed point $f$. Since $f\in c_0(V,\mu)$, there exists an increasing sequence $(F_n)_{n\geq 1}$ of finite subsets of $V$ with $\bigcup_{n=1}^\infty F_n=V$ such that, for all $v\notin F_n$, $n\geq 1$, 
\[
|f(v)\mu_v|\leq \frac{1}{n}.
\]
Let $\widetilde{\mu}_v = (n+1) \mu_v$ for $v\in F_{n+1}\setminus F_{n}$, $n\geq 0$, where $F_0=\varnothing$. Then $(\mu_v/\widetilde{\mu}_v)_v\in c_0(V)$ and $f\in \ell^\infty(V,\widetilde{\mu})$.

Now let $v\in V$. Let $m\geq 1$ be large enough so that $v\in V_m$. Then the restriction of $\frac{1}{f(v)}f$ to $V_m(v)$, see \eqref{eq-Vmv}, is a backward invariant sequence on $V_m(v)$ with value 1 at $v$, which implies by Theorem \ref{t-charfinp}(b) that
\[
c_\infty(V_m(v),\widetilde{\mu}) \leq r_\infty(V_m(v),\widetilde{\mu}) \leq \frac{1}{|f(v)|}\sup_{u\in V_m(v)}|f(u)\widetilde{\mu}(u)|\leq \frac{1}{|f(v)|}\|f\|_{\ell^\infty(V,\widetilde{\mu})}.
\]
Letting $m\to\infty$ we obtain (ii), see Remark \ref{rem-contfr}(c).

Since (ii) trivially implies (iii), it remains to show that (iii) implies (i).

Thus, let $v\in V$, and let a weight $\widetilde{\mu}$ on $V(v)$ be given by (iii) so that $c_\infty(V(v),\widetilde{\mu}) < \infty$. Since $B$ is defined on $c_0(V,\mu)$ and $(\mu_u/\widetilde{\mu}_u)_u\in c_0(V(v))$, we have that $B$ is also defined on $c_0(V(v),\widetilde{\mu})$. By Theorem \ref{t-charinfinv}(a) there is then a backward invariant sequence $f\in \ell^\infty(V(v), \widetilde{\mu})$ with $f(v)=1$, and $f$ is a fixed point for $B$. Using again that $(\mu_u/\widetilde{\mu}_u)_u\in c_0(V(v))$, we deduce that $f\in c_0(V(v),\mu)$. Since $v\in V$ was arbitrary, Theorem \ref{chaos-fixedpoint} tells us that $B$ is chaotic on $c_0(V,\mu)$; note that \eqref{eq-fpleft} holds for all rooted trees.
\end{proof}

By the usual conjugacy argument, see Section \ref{s-notation} and \cite{GrPa21}, we can deduce a characterization of chaos for weighted backward shifts on unweighted spaces. Recall the notation $\lambda(v\to u)$ for weights $\lambda$ and vertices $v\in V$, $u\in V(v)$, see Section \ref{s-notation}.

\begin{theorem}\label{t-charplambda}
Let $V$ be a rooted tree and $\lambda=(\lambda_v)_{v\in V}$ a weight. Suppose that the weighted backward shift $B_\lambda$ is an operator on $\ell^p(V)$, $1< p<\infty$. Then $B_\lambda$ is chaotic on $\ell^p(V)$ if and only if, for every $v\in V$,
\[
c_p\big(V(v),(\tfrac{1}{\lambda(v\to u)})_{u\in V(v)}\big) < \infty.
\]
\end{theorem}

Indeed, the conjugacy \eqref{eq-conjB2} leads to the stated condition but with weight $(\tfrac{1}{\lambda(v_0\to u)}\mu_{v_0})_{u\in V(v)}$, where $v_0$ denotes the root. Since the continued fractions $c_p$ are positively homogeneous in the weight, one may multiply by $\lambda(v_0\to v)/\mu_{v_0}$, see \eqref{eq-lambda}. The same remark applies to the following two results.

\begin{theorem}\label{t-char1lambda}
Let $V$ be a rooted tree and $\lambda=(\lambda_v)_{v\in V}$ a weight. Suppose that the weighted backward shift $B_\lambda$ is an operator on $\ell^1(V)$. Then $B_\lambda$ is chaotic on $\ell^1(V)$ if and only if, for every $v\in V$, there exists a branch $(v,v_1,v_2,\ldots)$ starting from $v$ such that
\[
\sum_{n=1}^\infty \frac{1}{|\lambda(v\to v_n)|} <\infty.
\]
\end{theorem}

This result should be compared with the characterizations of hypercyclicity and mixing of $B_\lambda$ on $\ell^1(V)$, see \cite[Theorem 4.4]{GrPa21}. Interestingly, $B_\lambda$ is mixing if and only if, for every $v\in V$, there are children $v_n\in\Chi^n(v)$, $n\geq 1$, such that
\[
|\lambda(v\to v_n)|\to\infty
\]
as $n\to\infty$, but it is not required that the $v_n$ belong to a single branch.

With Theorem \ref{t-char1lambda}, for rooted trees, the question of chaos for weighted backward shifts on $\ell^1(V)$ can be considered to be settled, both for theoretical and practical purposes. 

\begin{theorem}\label{t-char0lambda}
Let $V$ be a rooted tree and $\lambda=(\lambda_v)_{v\in V}$ a weight. Suppose that the weighted backward shift $B_\lambda$ is an operator on $c_0(V)$. Then the following assertions are equivalent:
\begin{enumerate}
\item[\rm (i)] $B_\lambda$ is chaotic on $c_0(V)$;
\item[\rm (ii)] there is a non-zero sequence $(\varepsilon_v)_{v\in V}\in c_0(V)$ such that, for every $v\in V$, 
\[
c_\infty\big(V(v),(\tfrac{1}{\varepsilon_{u}\lambda(v\to u)})_{u\in V(v)}\big) < \infty;
\]
\item[\rm (iii)] for every $v\in V$ there is a non-zero sequence $(\varepsilon_u)_{u\in V(v)}\in c_0(V(v))$ such that
\[
c_\infty\big(V(v),(\tfrac{1}{\varepsilon_{u}\lambda(v\to u)})_{u\in V(v)}\big) < \infty.
\]
\end{enumerate}
\end{theorem}

\section{A characterization of chaos on $\ell^p$ and $c_0$ over unrooted trees}\label{s-chaosunrootedlc}

Let us turn to weighted backward shifts on unrooted trees. We need to translate Theorem \ref{t-chunrooted} into the setting of the spaces $\ell^p$ and $c_0$. But this will be easy with the results of Section~\ref{s-backwinv}.

\begin{theorem}\label{t-charpmubil}
Let $V$ be an unrooted tree and $\mu=(\mu_v)_{v\in V}$ a weight on $V$. Suppose that the backward shift $B$ is an operator on $\ell^p(V,\mu)$, $1< p<\infty$. Then $B$ is chaotic on $\ell^p(V,\mu)$ if and only if, for every $v\in V$,
\[
c_p(V(v),\mu) < \infty
\]
and, for every $\eps>0$, there is some $N\geq 1$ such that
\[
c_p(V_-^N(v),\mu)<\eps.
\]
\end{theorem}

\begin{proof} This is a direct consequence of Theorem \ref{t-chunrooted} together with Theorem \ref{t-charpinv}(a). Note that, by hypothesis, each power $B^N$, $N\geq 1$, is an operator on $\ell^p(V,\mu)$. Now, by \eqref{eq-childN}, for any $v\in V$ and $N\geq 1$, the sets of children for vertices in $V_-^N(v)$ differ by at most two elements from sets of descendants of degree $N$ for vertices in $V$; hence the backward shift $B$ is defined on $\ell^p(V_-^N(v),\mu)$, so that the hypothesis of Theorem \ref{t-charpinv}(a) is satisfied for $V_-^N(v)$.
\end{proof}

\begin{remark} By Theorem \ref{t-rec2} and the definition of the tree $V_-^N(v)$, see Figure \ref{fig-v_N} and \eqref{eq-WnNN}, the value $c_p(V_-^N(v),\mu)$ is an almost classical continued fraction of generalized continued fractions. In fact, we see that
\begin{align*}
&c_p(V_-^N(v),\mu)=\\
&=\frac{1}{\Bigg( \frac{1}{c'_p(W_{-N}^N,\mu)^{p^*}} + \frac{1}{\Bigg( |\mu_{\prt^N(v)}|^p+\frac{1}{\bigg( \frac{1}{c'_p(W_{-2N}^N,\mu)^{p^*}}+\frac{1}{\Big( |\mu_{\prt^{2N}(v)}|^p+\frac{1}{\big( \frac{1}{c'_p(W_{-3N}^N,\mu)^{p^*}} + \ldots \big)^{p/p^{\ast}}} \Big)^{p^{\ast}/p}}\bigg)^{p/p^{\ast}}} \Bigg)^{p^{\ast}/p}}\Bigg)^{1/p^{\ast}}},
\end{align*}
where the limit is taken along the even-numbered convergents. Note the special interpretation of $c_p'$ explained in \eqref{eq-cp'}.

A corresponding remark applies to the constants for $V_-^N(v)$ in the following theorems.
\end{remark} 

In the same way, using Theorem \ref{t-char1inv}(a), we see that the previous result holds true also for $p=1$. But in that case the constants $c_1$ have a particularly simple form. We thus obtain the following.

\begin{theorem}\label{t-char1mubil}
Let $V$ be an unrooted tree and $\mu=(\mu_v)_{v\in V}$ a weight on $V$. Suppose that the backward shift $B$ is an operator on $\ell^1(V,\mu)$. 
Then $B$ is chaotic on $\ell^1(V,\mu)$ if and only if, for every $v\in V$, there exists a branch $(v,v_1,v_2,\ldots)$ starting from $v$ such that
\[
\sum_{n=1}^\infty |\mu_{v_n}| <\infty
\]
and, for every $\eps>0$, there is some $N\geq 1$ such that, either
\[
\sum_{n=1}^\infty |\mu_{\prt^{nN}(v)}| <\eps,
\]
or there is a branch $(v_0,v_1,v_2,\ldots)$ starting from $v_0:=\prt^{N}(v)$ with $v_N\neq v$ such that
\begin{equation}\label{eq-condalt}
\sum_{n=1}^{\infty} |\mu_{v_{nN}}| <\eps.
\end{equation}
\end{theorem}

In fact, the last condition initially reads as follows: there is some $j\geq 0$ and a branch $(\prt^{(j+1)N}(v),v_1,v_2,\ldots)$ starting  from $\prt^{(j+1)N}(v)$ with $v_N\neq \prt^{jN}(v)$ such that
\begin{equation}\label{eq-condalt2}
\sum_{n=1}^{j} |\mu_{\prt^{nN}(v)}|+ \sum_{n=1}^\infty |\mu_{v_{nN}}| <\eps.
\end{equation}
Then \eqref{eq-condalt2} implies \eqref{eq-condalt} by considering $v_0:=\prt^{(j+1)N}(v)$. On the other hand, \eqref{eq-condalt} implies \eqref{eq-condalt2} with $j=0$.

\begin{example}
Neither alternative of the second condition in Theorem \ref{t-char1mubil} can be dropped in general. For the first alternative, this follows from considering the tree $V=\ZZ$.

On the other hand, consider once more the comb $V=\ZZ\times\NN_0$ of Example \ref{ex-comb}. For the weights we take $\mu_{(n,0)} = \tfrac{1}{2^n}$, $n\geq 0$; $\mu_{(n,0)} = \tfrac{1}{|n|}$, $n\leq -1$; and $\mu_{(n+m,m)} = \tfrac{1}{2^m}\mu_{(n,0)}$ for $m\geq 1$ (recall that $(n+m,m)$ is a descendant of degree $m$ of $(n,0)$). Then the first condition and the second alternative of the second condition are satisfied for every $v\in V$, so that $B$ is chaotic. But the first alternative is never satisfied.
\end{example}

The case of the space $c_0(V,\mu)$ is again more complicated. 

\begin{theorem}\label{t-char0mubil}
Let $V$ be an unrooted tree and $\mu=(\mu_v)_{v\in V}$ a weight on $V$. Suppose that the backward shift $B$ is an operator on $c_0(V,\mu)$. Then the following assertions are equivalent:
\begin{enumerate}
\item[\rm (i)] $B$ is chaotic on $c_0(V,\mu)$;
\item[\rm (ii)] there is a weight $\widetilde{\mu}$ on $V$ with $(\mu_v/\widetilde{\mu}_v)_v\in c_0(V)$ such that, for every $v\in V$, 
\[
c_\infty(V(v),\widetilde{\mu}) < \infty
\]
and, for any $\eps>0$, there is some $N\geq 1$ such that
\[
c_\infty(V_-^N(v),\widetilde{\mu}) < \eps;
\]
\item[\rm (iii)] for every $v\in V$, there is a weight $\widetilde{\mu}$ on $V$ with $(\mu_u/\widetilde{\mu}_u)_u\in c_0(V)$ such that
\[
c_\infty(V(v),\widetilde{\mu}) < \infty
\]
and, for any $\eps>0$, there is some $N\geq 1$ such that, 
\[
c_\infty(V_-^N(v),\widetilde{\mu}) < \eps.
\]
\end{enumerate}
\end{theorem}

For the proof we will need the following simple lemma.

\begin{lemma}\label{l-char0mubil}
Let $V$ be a countable set, $\mu=(\mu_v)_{v\in V}$ a weight on $V$, and $f_n\in c_0(V,\mu)$, $n\geq 1$, such that $\|f_n\|_{c_0(V,\mu)}\to 0$ as $n\to \infty$. Then there is a weight $\widetilde{\mu}$ on $V$ with $(\mu_v/\widetilde{\mu}_v)_v\in c_0(V)$ such that $f_n\in c_0(V,\widetilde{\mu})$, $n\geq 1$, and $\|f_n\|_{c_0(V,\widetilde{\mu})}\to 0$ as $n\to \infty$.
\end{lemma} 

\begin{proof}
Let $(\eps_n)_{n\geq 1}$ be a decreasing sequence tending to zero such that $\|f_n\|_{c_0(V,\mu)}\leq\eps_n$ for $n\geq 1$. Since every sequence $f_n$ belongs to $c_0(V,\mu)$, we can construct an increasing sequence $(F_m)_{m\geq 1}$ of finite subsets of $V$ with $\bigcup_{m=1}^\infty F_m=V$ such that, for all $m\geq 1$, $1\leq n\leq m$ and $v\notin F_m$,
\[
|f_n(v)\mu_v|\leq \eps_{m+1}.
\]
For $m\geq 0$, let $\widetilde{\mu}_v = \tfrac{1}{\sqrt{\eps_{m+1}}} \mu_v$ for $v\in F_{m+1}\setminus F_{m}$, where $F_0=\varnothing$. Then $(\mu_v/\widetilde{\mu}_v)_v\in c_0(V)$.

Now let $n\geq 1$ and $v\in V$. Then there is some $m\geq 0$ such that $v\in F_{m+1}\setminus F_{m}$. If $m\geq n$, then
\[
|f_n(v)\widetilde{\mu}_v| = |f_n(v)\mu_v|\Big|\frac{\widetilde{\mu}_v}{\mu_v}\Big|\leq \sqrt{\eps_{m+1}}\leq \sqrt{\eps_{n}}.
\]
And if $m<n$ then
\[
|f_n(v)\widetilde{\mu}_v| = |f_n(v)\mu_v|\Big|\frac{\widetilde{\mu}_v}{\mu_v}\Big|\leq \frac{\eps_n}{\sqrt{\eps_{m+1}}}\leq \sqrt{\eps_{n}}.
\]
These inequalities show that $f_n\in c_0(V,\widetilde{\mu})$, $n\geq 1$, and $\|f_n\|_{c_0(V,\widetilde{\mu})}\to 0$ as $n\to \infty$.
\end{proof}

\begin{proof}[Proof of Theorem \ref{t-char0mubil}]
We start by showing that (i) implies (ii). Thus let $B$ be chaotic on $c_0(V,\mu)$. Let $(v_n)_{n\geq 1}$ be a sequence from $V$ in which every $v\in V$ appears infinitely often. It then follows from Theorem \ref{t-chunrooted} that, for each $n\geq 1$, 
there is a fixed point $f_n$ for $B$ in $c_0(V(v_n),\mu)$ such that $f_n(v_n) \neq 0$, $\|f_n\|_{c_0(V(v_n),\mu)}<\frac{1}{n}$, and there is some $N_n\geq 1$ and a backward invariant sequence $g_n$ in $c_0(V_-^{N_n}(v_n),\mu)$ such that $g_n(v_n)=1$ and 
\[
\|g_n-e_{v_n}\|_{c_0(V_-^{N_n}(v_n),\mu)}<\frac{1}{n}.
\]
We extend the $f_n$ and $g_n$ to all of $V$ by assigning the values 0 outside the stated trees. Applying Lemma \ref{l-char0mubil} to the sequences $f_n$ and $g_n-e_{v_n}$, $n\geq 1$, we obtain a weight $\widetilde{\mu}$ on $V$ with $(\mu_v/\widetilde{\mu}_v)_v\in c_0(V)$ such that $f_n, g_n \in c_0(V,\widetilde{\mu})$, $n\geq 1$, and 
\[
\|g_n-e_{v_n}\|_{c_0(V_-^{N_n}(v_n),\widetilde{\mu})}\to 0\text{ as $n\to \infty$}.
\]

With Theorem \ref{t-charinfinv}(a) we deduce that, for any $v\in V$,
\[
c_\infty(V(v),\widetilde{\mu})<\infty.
\]
On the other hand, it follows as in the proof of Theorem \ref{t-charpmubil} that $B$ is defined on each $c_0(V_-^{N_n}(v_n),\mu)$. Since $(\mu_v/\widetilde{\mu}_v)_v\in c_0(V)$, the weight $\mu$ may be replaced here by the weight $\widetilde{\mu}$. Theorem \ref{t-charinfinv}(a) now implies that
\[
c_\infty(V_-^{N_n}(v_n),\widetilde{\mu}) \to 0 \text{ as $n\to \infty$},
\]
which shows that, for any $v\in V$ and $\eps>0$, there is some $N\geq 1$ such that 
\[
c_\infty(V_-^{N}(v),\widetilde{\mu}) <\eps.
\]

Next, (ii) trivially implies (iii).

Assume finally that (iii) holds. Let $v\in V$ and $\widetilde{\mu}$ be the associated weight on $V$ with $(\mu_u/\widetilde{\mu}_u)_u\in c_0(V)$.
Then Theorem \ref{t-charinfinv}(a) implies that there exists a backward invariant sequence $f\in \ell^\infty(V(v),\widetilde{\mu})$ with $f(v)=1$, hence $f\in c_0(V(v),\mu)$; moreover, for any $\eps>0$ there is some $N\geq 1$ and a backward invariant sequence $f\in \ell^\infty(V_-^N(v),\widetilde{\mu})$ with $f(v)=1$ such that
\[
\|f-e_v\|_{\ell^\infty(V_-^N(v),\widetilde{\mu})}< \eps,
\] 
hence $f\in c_0(V_-^N(v),\mu)$ with
\[
\|f-e_v\|_{\ell^\infty(V_-^N(v),\mu)}< M\eps,
\]
where $M=\sup_{u\in V} \big|\mu_u/\widetilde{\mu}_u\big|$. Note that in both cases the hypothesis in Theorem \ref{t-charinfinv}(a) is satisfied by the argument given in the proof of Theorem \ref{t-charpmubil} and since $(\mu_u/\widetilde{\mu}_u)_u\in c_0(V)$. 

Theorem \ref{t-chunrooted} then implies that $B$ is chaotic on $c_0(V,\mu)$, hence (i).
\end{proof}

We use once more a conjugacy to transfer the weight from the space to the operator.

\begin{theorem}\label{t-charplambdabil}
Let $V$ be an unrooted tree and $\lambda=(\lambda_v)_{v\in V}$ a weight. Suppose that the weighted backward shift $B_\lambda$ is an operator on $\ell^p(V)$, $1< p<\infty$. Then $B_\lambda$ is chaotic on $\ell^p(V)$ if and only if, for every $v\in V$,
\[
c_p\big(V(v),(\tfrac{1}{\lambda(v\to u)})_{u\in V(v)}\big) < \infty
\]
and, for every $\eps>0$, there is some $N\geq 1$ such that
\[
c_p\big(V_-^N(v),(\tfrac{\lambda(w\to v)}{\lambda(w\to u)})_{u\in V_-^N(v)}\big)<\eps,
\]
where $w$ is a common ancestor of $u$ and $v$.
\end{theorem}

In fact, the conjugacy \eqref{eq-conjB2} initially gives other weights. Then an argument as in \eqref{eq-conjug} and the positive homogeneity of the continued fractions $c_p$ in the weight allow to arrive at the stated conditions. 

In the same manner one obtains the following two results.

\begin{theorem}\label{t-char1lambdabil}
Let $V$ be an unrooted tree and $\lambda=(\lambda_v)_{v\in V}$ a weight. Suppose that the weighted backward shift $B_\lambda$ is an operator on $\ell^1(V)$. Then $B_\lambda$ is chaotic on $\ell^1(V)$ if and only if, for every $v\in V$, there exists a branch $(v,v_1,v_2,\ldots)$ starting from $v$ such that
\[
\sum_{n=1}^\infty \frac{1}{|\lambda(v\to v_n)|} <\infty
\]
and, for every $\eps>0$, there is some $N\geq 1$ such that, either
\[
\sum_{n=1}^\infty |\lambda(\prt^{nN}(v)\to v)| <\eps,
\]
or there is a branch $(v_0,v_1,v_2,\ldots)$ starting from $v_0:=\prt^{N}(v)$ with $v_N\neq v$ such that
\begin{equation}\label{eq-cone}
\sum_{n=1}^\infty \Big|\frac{\lambda(\prt^{N}(v)\to v)}{\lambda(\prt^{N}(v)\to v_{nN})}\Big| <\eps.
\end{equation}
\end{theorem}

With Theorem \ref{t-char1lambdabil}, the question of chaos for weighted backward shifts on $\ell^1(V)$ can be considered to be also closed on unrooted trees. 

\begin{theorem}\label{t-char0lambdabil}
Let $V$ be an unrooted tree and $\lambda=(\lambda_v)_{v\in V}$ a weight. Suppose that the weighted backward shift $B_\lambda$ is an operator on $c_0(V)$. Then the following assertions are equivalent:
\begin{enumerate}
\item[\rm (i)] $B_\lambda$ is chaotic on $c_0(V)$;
\item[\rm (ii)] there is a non-zero sequence $(\varepsilon_v)_{v\in V}\in c_0(V)$ such that, for every $v\in V$, 
\[
c_\infty\big(V(v),(\tfrac{1}{\varepsilon_{u}\lambda(v\to u)})_{u\in V(v)}\big) < \infty
\]
and, for any $\eps>0$, there is some $N\geq 1$ such that
\[
c_\infty\big(V_-^N(v),(\tfrac{\lambda(w\to v)}{\varepsilon_{u}\lambda(w\to u)})_{u\in V_-^N(v)}\big) < \eps,
\]
where $w$ is a common ancestor of $u$ and $v$;
\item[\rm (iii)] for every $v\in V$, there is a non-zero sequence $(\varepsilon_u)_{u\in V}\in c_0(V)$ such that,
\[
c_\infty\big(V(v),(\tfrac{1}{\varepsilon_{u}\lambda(v\to u)})_{u\in V(v)}\big) < \infty
\]
and, for any $\eps>0$, there is some $N\geq 1$ such that
\[
c_\infty\big(V_-^N(v),(\tfrac{\lambda(w\to v)}{\varepsilon_{u}\lambda(w\to u)})_{u\in V_-^N(v)}\big) < \eps,
\]
where $w$ is a common ancestor of $u$ and $v$.
\end{enumerate}
\end{theorem}

We have seen in Theorem \ref{t-chunrootedsimplified} that for certain shifts on unrooted trees without a free left end the characterizations simplify; see also Remark \ref{rem-sym}(a). And for unrooted trees with a free left end we note Theorem \ref{chaos-fixedpoint}; see also Remark \ref{rem-GMM}(b). As a consequence we have the following result, which we only state for weighted backward shifts on unweighted spaces. Recall that $\gen(v)$ denotes the generation containing the vertex $v\in V$.

Note that condition \eqref{eq-simpl} below is satisfied for symmetric weights if $p\neq 1$ and if the tree has no free left end.

\begin{theorem}\label{t-charlambdabil_simple}
Let $V$ be an unrooted tree, $\lambda=(\lambda_v)_{v\in V}$ a weight, and $X=\ell^p(V)$, $1\leq p<\infty$, or $X=c_0(V)$. Suppose that the weighted backward shift $B_\lambda$ is an operator on $X$. 

\emph{(a)} Suppose that, for every $v\in V$,
\begin{equation}\label{eq-simpl}
\Big(\frac{\lambda(w\to u)}{\lambda(w\to v)}\Big)_{u\in \emph{\gen}(v)} \notin \ell^{p^\ast}(\emph{\gen}(v)),
\end{equation}
where $w$ is a common ancestor of $u$ and $v$, and where $p^\ast=1$ if $X=c_0(V)$. Then $B_\lambda$ is chaotic on $X$ if and only if the following condition is satisfied:

\emph{(C)} if $X=\ell^1(V)$: for every $v\in V$, there exists a branch $(v,v_1,v_2,\ldots)$ starting from $v$ such that
\[
\sum_{n=1}^\infty \frac{1}{|\lambda(v\to v_n)|} <\infty;
\]
if $X=\ell^p(V)$, $1< p<\infty$: for every $v\in V$, 
\[
c_p\big(V(v),(\tfrac{1}{\lambda(v\to u)})_{u\in V(v)}\big) < \infty;
\]
if $X=c_0(V)$: there is a non-zero sequence $(\varepsilon_v)_{v\in V}\in c_0(V)$ such that, for every $v\in V$, 
\[
c_\infty\big(V(v),(\tfrac{1}{\varepsilon_{u}\lambda(v\to u)})_{u\in V(v)}\big) < \infty
\]
or, equivalently, for every $v\in V$, there is a non-zero sequence $(\varepsilon_u)_{u\in V(v)}\in c_0(V(v))$ such that,
\[
c_\infty\big(V(v),(\tfrac{1}{\varepsilon_{u}\lambda(v\to u)})_{u\in V(v)}\big) < \infty.
\]

\emph{(b)} Suppose that $V$ has a free left end. Then $B_\lambda$ is chaotic on $X$ if and only if condition \emph{(C)} is satisfied and, for some $v\in V$,
\[
\sum_{n=1}^\infty \lambda(\prt^n(v)\to v)e_{\prt^{n}(v)}\in X.
\]
\end{theorem}

\section{Particular shifts on rooted or unrooted trees}\label{s-special_trees}

When $X=\ell^p(V)$, $1<p<\infty$, or $X=c_0(V)$, the characterizations of chaotic weighted backward shifts obtained in the previous two sections require to evaluate a continued fraction, which is in general rather difficult. Only in one specific case are we able to turn them into a simple explicit form, the case when the shift and the tree are symmetric. This will be considered in Subsection \ref{subs-symsym}. As a special case we will then look at Rolewicz operators, that is, operators of the form $\lambda B$. 

In other situations, we will start again from scratch and use the results in Sections \ref{s-periodic} and \ref{s-unrooted} to obtain a new, easy characterization for chaos under certain assumptions, see Subsection \ref{subs-sym}. In the final subsection we study a particular family of weighted backward shifts that all turn out to be chaotic.

\subsection{Symmetric shifts on symmetric trees}\label{subs-symsym}
There is one important special case in which the continued fractions can be considerably simplified. Recall the notions of symmetric weighted backward shifts and symmetric trees given in Section \ref{s-notation} and the related numbers $\lambda_n$ of common weights and $\gamma_n$ of number of children in generation $n$.

\begin{theorem}\label{t-symsym}
Let $V$ be a symmetric tree with $|\Chi(v)|<\infty$ for all $v\in V$, and let $\lambda=(\lambda_n)_{n}$ be a symmetric weight. Let $X=\ell^p(V)$, $1 < p<\infty$, or $X=c_0(V)$, and suppose that the weighted backward shift $B_\lambda$ is an operator on $X$. 

\emph{(a)} Let the tree be either rooted, or unrooted without a free left end. Then $B_\lambda$ is chaotic if and only if 
\begin{align*}
\sum_{n=1}^\infty \frac{1}{(\gamma_0\cdots\gamma_{n-1})^{p/p*}|\lambda_1\cdots\lambda_n|^{p} } < \infty,  &\text{ if $X=\ell^p(V)$,}\\
\lim_{n\to\infty}\gamma_0\cdots\gamma_{n-1}|\lambda_1\cdots\lambda_n|= \infty,  &\text{ if $X=c_0(V)$.}
\end{align*}

\emph{(b)} If the tree is unrooted with a free left end, then $B_\lambda$ is chaotic if and only if the condition in \emph{(a)} holds and 
\begin{align*}
\sum_{n=1}^\infty |\lambda_{-n+1}\cdots\lambda_0|^{p} < \infty,  &\text{ if $X=\ell^p(V)$,}\\
\lim_{n\to\infty} \lambda_{-n+1}\cdots\lambda_0 =0,  &\text{ if $X=c_0(V)$.}
\end{align*}
\end{theorem}

\begin{proof}
Let $X=\ell^p(V)$, $1<p<\infty$. If the tree is rooted, then Theorem \ref{t-charplambda} shows that $B_\lambda$ is chaotic if and only if, for every $v\in V$, $v\in\gen_n$ say,
\[
c_p\big(V(v),(\tfrac{1}{\lambda(v\to u)})_{u\in V(v)}\big)=\lim_{m\to\infty}c_p\big(V(v)_{m-n},(\tfrac{1}{\lambda(v\to u)})_{u\in V(v)}\big) < \infty;
\]
see \eqref{eq-Vmv} for the meaning of $V(v)_{m-n}=V_m(v)$. 

Then, for $m\geq n+1$, one sees that $c_p\big(V(v)_{m-n},(\tfrac{1}{\lambda(v\to u)})_{u\in V(v)}\big)$ equals
\[
\frac{1}{\Bigg( \gamma_n\frac{1}{\Bigg( \frac{1}{|\lambda_{n+1}|^p}+\frac{1}{+\big( \ddots \gamma_{m-2}\frac{1}{\big(\frac{1}{|\lambda_{n+1}\ldots\lambda_{m-1}|^p}+\frac{1}{ \gamma_{m-1}^{p/p^{\ast}}|\lambda_{n+1}\ldots\lambda_{m}|^{p}} \big)^{p^{\ast}/p}}\big)^{p/p^{\ast}} } \Bigg)^{p^{\ast}/p}}\Bigg)^{1/p^{\ast}}},
\]
see Remark \ref{rem-contfr}(a), which turns into
\[
\Big(\frac{1}{\gamma_n^{p/p*}|\lambda_{n+1}|^{p} }+\frac{1}{(\gamma_{n}\gamma_{n+1})^{p/p*}|\lambda_{n+1}\lambda_{n+2}|^{p} }+\ldots +\frac{1}{(\gamma_n\cdots\gamma_{m-1})^{p/p*}|\lambda_{n+1}\cdots\lambda_{m}|^{p} }\Big)^{1/p}.
\]
From here the assertion follows.

If the tree is unrooted, then the assertion follows in the same way from Theorem \ref{t-charlambdabil_simple}. Note that if the tree has no free left end and the weight $\lambda$ is symmetric, then each generation is infinite and $\tfrac{\lambda(w\to u)}{\lambda(w\to v)}=1$ for all $u$ and $v$ in the same generation and any common ancestor $w$, so that \eqref{eq-simpl} is automatically satisfied. 

Now let $X=c_0(V)$. If the tree is rooted, then it follows from Theorem \ref{t-char0lambda} that $B_\lambda$ is chaotic if and only if, for every $v\in V$, $v\in\gen_n$ say, there is a non-zero sequence $(\varepsilon_u)_{u\in V(v)}\in c_0(V(v))$ such that, 
\[
c_\infty\big(V(v),(\tfrac{1}{\varepsilon_{u}\lambda(v\to u)})_{u\in V(v)}\big)=\lim_{m\to\infty}c_\infty\big(V(v)_{m-n},(\tfrac{1}{\varepsilon_{u}\lambda(v\to u)})_{u\in V(v)}\big) < \infty.
\]
Since, by hypothesis, every generation in $V(v)$ only has a finite number of members, we may replace each $\varepsilon_u$, $u\in V(v)_k$, $k\geq 0$, by 
\[
\max_{w\in V(v)_k} |\varepsilon_w|;
\]
the new sequence is still in $c_0(V(v))$. In other words, we may assume that the $\varepsilon_u$ are strictly positive and only depend on the generation of $u$, so that we can write
\[
\varepsilon_n= \varepsilon_u\quad \text{if $u\in \gen_m$, $m\geq n$.}
\]

Then we have, for any $m\geq n+1$, 
\begin{align*}
c_\infty\big(V(v)_{m-n},&(\tfrac{1}{\varepsilon_{u}\lambda(v\to u)})_{u\in V(v)}\big)\\
&=\frac{1}{ \gamma_n\tfrac{1}{ \max\bigg(\tfrac{1}{\varepsilon_{n+1}|\lambda_{n+1}|},\tfrac{1}{  \ddots \gamma_{m-2}\tfrac{1}{\max\Big(\frac{1}{\varepsilon_{m-1}|\lambda_{n+1}\ldots \lambda_{m-1}|},\frac{1}{\gamma_{m-1}\varepsilon_{m}|\lambda_{n+1}\ldots \lambda_m|}\Big)}  }\bigg) }}\\
& = \max \Big(\frac{1}{\gamma_n\varepsilon_{n+1}|\lambda_{n+1}|}, \frac{1}{\gamma_n\gamma_{n+1}\varepsilon_{n+2}|\lambda_{n+1}\lambda_{n+2}|},\ldots, \frac{1}{\gamma_n\cdots\gamma_{m-1}\varepsilon_{m}|\lambda_{n+1}\cdots\lambda_{m}|}\Big).
\end{align*}
Thus, $B_\lambda$ is chaotic on $c_0(V)$ if and only if there is a strictly positive null sequence $(\varepsilon_n)_n$ such that
\[
\sup_n \frac{1}{\gamma_0\cdots\gamma_{n-1}\varepsilon_{n}|\lambda_{1}\cdots\lambda_{n}|} <\infty,
\]
which is equivalent to the condition in the assertion.

If the tree is unrooted then we finish the proof by using Theorem \ref{t-charlambdabil_simple}, just as we did in the case when $X=\ell^p(V)$, $1<p<\infty$.
\end{proof}

We mention that the result recovers the well-known characterizations of chaos in the case of the classical trees $\NN_0$ and $\ZZ$, see \cite[Examples 4.9, 4.15]{GrPe11}.

Comparing, for $X=c_0(V)$, the conditions for chaos with those for the mixing property obtained in \cite[Corollaries 4.7 and 5.5]{GrPa21}, we have the following; note that $\gamma_{-n}\cdots \gamma_{-1}\to \infty$ as $n\to\infty$ if the tree is unrooted without a free left end.

\begin{corollary} 
Let $V$ be a symmetric tree with $|\Chi(v)|<\infty$ for all $v\in V$, and let $\lambda=(\lambda_n)_{n}$ be a symmetric weight. Let the weighted backward shift $B_\lambda$ be an operator on $c_0(V)$. Then $B_\lambda$ is chaotic on $c_0(V)$ if and only if it is mixing.
\end{corollary} 

The much simpler case when $X=\ell^1(V)$ will be treated, for symmetric shifts on arbitrary trees, in Subsection \ref{subs-sym}.

\subsection{Rolewicz operators}\label{subs-rol}
Rolewicz operators are particular symmetric weighted backward shifts where the weight $\lambda=(\lambda)_{v\in V}$ is a constant sequence. In other words, we regard the multiple $\lambda B$, $\lambda\in\KK\setminus\{0\}$, of the unweighted backward shift; see also \cite{GrPa21}. They are named after Rolewicz \cite{Rol} who identified them as hypercyclic operators, the first ever in the context of Banach spaces.

The case of $X=\ell^1(V)$ can be treated completely; note that, by \cite[Example 2.5]{GrPa21}, any Rolewicz operator is continuous on this space. Thus, Theorem \ref{t-char1lambda} together with \cite[Corollaries 4.5 and 5.4]{GrPa21} implies the following; see also Theorem \ref{t-l1sym} below.

\begin{theorem} 
Let $V$ be a tree and $\lambda\in\KK\setminus\{0\}$. Let $\lambda B$ be the corresponding Rolewicz operator. Then the following assertions are equivalent:
\begin{enumerate}
\item[\emph{(i)}] $\lambda B$ is chaotic on $\ell^1(V)$;
\item[\emph{(ii)}] $\lambda B$ is mixing on $\ell^1(V)$;
\item[\emph{(iii)}] $\lambda B$ is hypercyclic on $\ell^1(V)$;
\item[\emph{(iv)}] the tree is rooted and $|\lambda|>1$.
\end{enumerate}
\end{theorem}

If $X=\ell^p(V)$, $1<p<\infty$, or $X=c_0(V)$, then the continuity of $\lambda B$ implies that $\sup_{v\in V}|\Chi(v)|<\infty$, see \cite[Example 2.5]{GrPa21}. In that case, Theorem \ref{t-symsym} provides a characterization of chaos for the Rolewicz operator on these spaces when the underlying tree is symmetric. In the special case of trees in which each vertex has the same number of children, we then have the following in view of the discussions after \cite[Corollary 4.5]{GrPa21} and after \cite[Corollary 5.4]{GrPa21}.

\begin{theorem} Let $V$ be a tree in which each vertex has exactly $N\geq 1$ children. Let $X=\ell^p(V)$, $1< p<\infty$, or $X=c_0(V)$, and let $\lambda\in\KK\setminus\{0\}$. Let $\lambda B$ be the corresponding Rolewicz operator. Then the following assertions are equivalent:
\begin{enumerate}
\item[\emph{(i)}] $\lambda B$ is chaotic;
\item[\emph{(ii)}] $\lambda B$ is mixing;
\item[\emph{(iii)}] $\lambda B$ is hypercyclic;
\item[\emph{(iv)}] the tree is rooted, or unrooted without a free left end, and\\
$|\lambda|>N^{-1/p^*}  \text{ if $X=\ell^p(V)$, $1<p<\infty$};$\\
$|\lambda|>N^{-1} \text{ if $X=c_0(V)$}.$
\end{enumerate}
\end{theorem}

In \cite[Example 4.6]{GrPa21} we have given a symmetric tree on which some Rolewicz operator is hypercyclic but not mixing. But there are also Rolewicz operators that are mixing and not chaotic.

\begin{example}
There exists a symmetric rooted tree $V$ on which the unweighted backward shift $B$ is mixing on $X=\ell^p(V)$, $1<p<\infty$, but not chaotic. To see this, fix an integer $m\geq 2^{p/p^\ast}$. Then consider the symmetric rooted tree with $\gamma_{m^n}=2$, $n\geq 1$, and $\gamma_n=1$ otherwise. Since, for all $v\in V$, $|\Chi^n(v)|\to\infty$ as $n\to\infty$, \cite[Corollary 4.5]{GrPa21} implies that $B$ is mixing on $X$. On the other hand, Theorem \ref{t-symsym} shows that $B$ is not chaotic on $\ell^p(V)$. 
\end{example}

We refer to Corollary \ref{co-rol} and Example \ref{ex-nonsuff} for further discussions on Rolewicz operators.

\subsection{Symmetric shifts}\label{subs-sym}
In this subsection we consider symmetric weighted backward shifts on arbitrary trees. The case of $X=\ell^1(V)$ is easily treated.

\begin{theorem}\label{t-l1sym}
Let $V$ be a tree and $\lambda=(\lambda_n)_{n}$ a symmetric weight. Suppose that the weighted backward shift $B_{\lambda}$ is an operator on $\ell^1(V)$. Then $B_\lambda$ is chaotic if and only if
\[
\sum_{n=1}^\infty \frac{1}{|\lambda_1\cdots\lambda_n|} < \infty
\]
and, if the tree is unrooted, then
\[
\sum_{n=1}^\infty |\lambda_{-n+1}\cdots\lambda_0| < \infty.
\]
\end{theorem}

This is a direct consequence of Theorems \ref{t-char1lambda} and \ref{t-char1lambdabil}. In the unrooted case one needs to note that the first term in the series in \eqref{eq-cone} is 1, so that only the first alternative of the second characterizing condition in Theorem  \ref{t-char1lambdabil} can hold for $\varepsilon<1$. Moreover, by the continuity of $B_\lambda$, we have that $\sup_n |\lambda_n|<\infty$, see \cite[Proposition 2.3]{GrPa21}, which then leads to the second condition in the theorem above.

So let us turn to the spaces $\ell^p(V)$, $1<p<\infty$, and $c_0(V)$. The simplification of the continued fractions in Subsection \ref{subs-symsym} relied heavily on the fact that both the tree and the weight are symmetric. When we drop the symmetry of the tree, the structure of the continued fractions becomes quite involved.

So we will here go back to the results in Sections \ref{s-periodic} and \ref{s-unrooted} in order to obtain a characterization of chaos under certain additional assumptions. In particular, we will obtain a new, simpler proof of Theorem \ref{t-symsym}.

First, it is easy to arrive at a necessary condition for chaos.

\begin{theorem}\label{pr-symneccond}
Let $V$ be a tree with $|\Chi(v)|<\infty$ for all $v\in V$, and let $\lambda=(\lambda_n)_{n}$ be a symmetric weight.  Let $X=\ell^p(V)$, $1< p<\infty$, or $X=c_0(V)$, and suppose that the weighted backward shift $B_{\lambda}$ is an operator on $X$. If $B_\lambda$ is chaotic then, for all $v\in V$, 
\begin{align*}
\sum_{n=1}^\infty \frac{1}{|\Chi^n(v)|^{p/p*}|\lambda_{m+1}\cdots\lambda_{m+n}|^{p} } < \infty,  &\text{ if $X=\ell^p(V)$,}\\
\lim_{n\to\infty}|\Chi^n(v)||\lambda_{m+1}\cdots\lambda_{m+n}|= \infty,  &\text{ if $X=c_0(V)$,}
\end{align*}
where $m$ is such that $v\in\emph{\gen}_m$.
\end{theorem}

\begin{proof} 
Let $X=\ell^p(V)$, $1<p<\infty$. Let $v\in \gen_m$. Since $B_\lambda$ is chaotic, there is by Theorem \ref{chaos-fixedpoint} a fixed point $f\in X$ with $f(v)=1$. In particular we have that
\[
1=f(v)=\lambda_{m+1}\cdots\lambda_{m+n}\sum_{u\in\Chi^n(v)} f(u)
\]
and hence by H\"older's inequality
\[
\frac{1}{|\Chi^n(v)|^{p/p*}||\lambda_{m+1}\cdots\lambda_{m+n}|^p}\leq \| f\chi_{\Chi^n(v)}\|^p.
\]
Summing over $n\geq 1$ yields the result since $f\in X$.

The proof for $X=c_0(V)$ is similar.
\end{proof}

For $X=c_0(V)$, the necessary condition also follows from \cite[Theorem 4.4]{GrPa21} and the fact that chaotic weighted backward shifts are mixing.

We will give an example below to show that these conditions are in general not sufficient, even when the tree is rooted. Before doing so we will show that we arrive indeed at a characterization under any of the following assumptions on the geometry of the tree and the weights:
\begin{enumerate}
\item[(A1)] for all $v\in V$, there are $N\in\NN$ and $C>0$ such that, for all $n\geq 1$ and each path $(v,v_1,\ldots,v_{nN})$ of length $nN$ in $V$,
\[
|\Chi^{nN}(v)|\leq C|\Chi^N(v)|\ |\Chi^{N}(v_{N})|\ |\Chi^{N}(v_{2N})|\cdots |\Chi^N(v_{(n-1)N})|;
\]
\item[(A2)] for all $v\in V$, there are $N\in\NN$ and $\rho>1$ such that, for all $u\in \Chi^{nN}(v)$, $n\geq 0$, 
\[
|\Chi^N(u)|^{1/p^\ast}|\lambda_{m+nN+1}\cdots\lambda_{m+(n+1)N}| \geq\rho,
\]
where $m$ is such that $v\in\gen_m$; here, $p^\ast=1$ if $X=c_0(V)$;
\item[(A3)] $X=c_0(V)$ and $|\lambda_n|\geq 1$ for all $n$.
\end{enumerate}

\begin{theorem}\label{pr-symcond}
Let $V$ be a tree with $|\Chi(v)|<\infty$ for all $v\in V$, and let $\lambda=(\lambda_n)_{n}$ be a symmetric weight. Let $X=\ell^p(V)$, $1< p<\infty$, or $X=c_0(V)$, and suppose that the weighted backward shift $B_{\lambda}$ is an operator on $X$. 

Assume that one of the conditions \emph{(A1)}, \emph{(A2)} or \emph{(A3)} holds.

\emph{(a)} Let the tree be either rooted, or unrooted without a free left end. Then $B_\lambda$ is chaotic if and only if, for all $v\in V$, 
\begin{equation}\label{eq-simpl3}
\begin{split}
\sum_{n=1}^\infty \frac{1}{|\Chi^n(v)|^{p/p*}|\lambda_{m+1}\cdots\lambda_{m+n}|^{p} } < \infty,  &\text{ if $X=\ell^p(V)$,}\\
\lim_{n\to\infty}|\Chi^n(v)||\lambda_{m+1}\cdots\lambda_{m+n}|= \infty,  &\text{ if $X=c_0(V)$,}
\end{split}
\end{equation}
where $m$ is such that $v\in\emph{\gen}_m$.

\emph{(b)} If the tree is unrooted with a free left end, then $B_\lambda$ is chaotic if and only if the condition in \emph{(a)} holds and 
\begin{align*}
\sum_{n=1}^\infty |\lambda_{-n+1}\cdots\lambda_0|^{p} < \infty,  &\text{ if $X=\ell^p(V)$,}\\
\lim_{n\to\infty} \lambda_{-n+1}\cdots\lambda_0 =0,  &\text{ if $X=c_0(V)$.}
\end{align*}
\end{theorem}
 
\begin{proof} 
By Theorems \ref{pr-symneccond} and \ref{chaos-fixedpoint} the conditions are necessary.

Conversely, suppose that the conditions are satisfied and one of (A1), (A2) or (A3) holds. By Theorems \ref{t-chunrootedsimplified} and \ref{chaos-fixedpoint}, it then suffices to show that, for any $v\in V$, there is a fixed point $f$ for $B_\lambda$ in $X$ over $V(v)$ with $f(v)=1$. Note that the additional assumption in Theorem \ref{t-chunrootedsimplified}, equivalently condition \eqref{eq-simpl}, is satisfied by the symmetry of the weight, see the proof of Theorem \ref{t-symsym}. Also, by Theorem \ref{chaos-fixedpoint}, it suffices if $f$ is just a periodic point of some period $N\geq 1$.

Let $v\in V$. We first assume that (A1) or (A2) holds, and we choose $N\in \NN$ accordingly. Let $X=\ell^p(V)$, $1< p<\infty$. We define $f$ on $V(v)$ in a canonical way. Indeed, let $f(v)=1$, and if $u\in \Chi^{nN}(v)$ for some $n\geq 1$, then we set
\[
f(u)= \frac{1}{ |\Chi^N(v)|\ |\Chi^{N}(v_{N})|\ |\Chi^{N}(v_{2N})|\cdots |\Chi^N(v_{(n-1)N})|\lambda_{m+1}\lambda_{m+2}\cdots\lambda_{m+nN}},
\]
where $m$ is such that $v\in\gen_m$ and $(v,v_1,\ldots,v_{nN})$ is the path leading from $v$ to $v_{nN}:=u$. Any other vertex in $V(v)$ is given the value 0.

It is then easily checked that, on $V(v)$, $B_\lambda^{N}f=f$. Moreover, 
\[
\|f\|^p = 1 + \sum_{n=1}^\infty \sum_{v_{nN} \in \Chi^{nN}(v)} \frac{1}{ |\Chi^N(v)|^p\ |\Chi^{N}(v_{N})|^p\cdots |\Chi^N(v_{(n-1)N})|^p|\lambda_{m+1}\lambda_{m+2}\cdots\lambda_{m+nN}|^p}.
\]

Under (A1), for the given $C$ we have that
\begin{align*}
\|f\|^p&\leq 1+ \sum_{n=1}^\infty \sum_{v_{nN} \in \Chi^{nN}(v)} \frac{C^p}{ |\Chi^{nN}(v)|^p|\lambda_{m+1}\lambda_{m+2}\cdots\lambda_{m+nN}|^p}\\
&= 1+ \sum_{n=1}^\infty  \frac{C^p}{ |\Chi^{nN}(v)|^{p/p\ast}|\lambda_{m+1}\lambda_{m+2}\cdots\lambda_{m+nN}|^p} <\infty,
\end{align*}
where we have used that $p-1=\tfrac{p}{p^*}$.

Under (A2), for the given $\rho$ we obtain that
\begin{align*}
\|f\|^p&\leq 1+ \sum_{n=1}^\infty \sum_{v_{nN} \in \Chi^{nN}(v)} \frac{1}{ |\Chi^N(v)|\ |\Chi^{N}(v_{N})|\cdots |\Chi^N(v_{(n-1)N})|}\frac{1}{\rho^{pn}}\\
&= 1+ \sum_{n=1}^\infty \frac{1}{\rho^{pn}} <\infty,
\end{align*}
where we have used that $p-\tfrac{p}{p^*}=1$ and 
\[
 \sum_{v_{nN} \in \Chi^{nN}(v)} \frac{1}{ |\Chi^N(v)|\ |\Chi^{N}(v_{N})|\cdots |\Chi^N(v_{(n-1)N})|}=1,\ n\geq 1.
\]

Under both conditions we can thus conclude that $f\in \ell^p(V(v))$ is a periodic point for $B_\lambda$ over $V(v)$ with $f(v)=1$.

The proof for $c_0(V)$ is very similar.

Now suppose that (A3) holds, thus $X=c_0(V)$. Let $v\in V$, $v\in\gen_m$ say. By (a), we find recursively a strictly increasing sequence $(M_n)_{n\geq 0}$ of positive integers with $M_0=0$ such that, for all $u\in \Chi^{M_{n-1}}(v)$, $n\geq 1$,
\begin{equation}\label{eq-mn}
|\Chi^{M_{n}-M_{n-1}}(u)||\lambda_{m+M_{n-1}+1}\cdots\lambda_{m+M_{n}}|\geq 2.
\end{equation}
Then we define $f$ on $V(v)$ in the following way. We start with $f(v)=1$. If $u\in \Chi^{M_n}(v)$, $n\geq 1$, we set
\[
f(u)= \frac{1}{ |\Chi^{M_1}(v)|\ |\Chi^{M_2-M_1}(v_{M_1})| \cdots |\Chi^{M_{n}-M_{n-1}}(v_{M_{n-1}})|\lambda_{m+1}\lambda_{m+2}\cdots\lambda_{m+M_n}},
\]
where $(v,v_1,\ldots,v_{M_n})$ is the path leading from $v$ to $u:=v_{M_n}$. It follows from \eqref{eq-mn} that
\[
|f(u)|\leq \frac{1}{2^n}.
\] 

Finally, if $u\in \Chi^j(v)$, $M_{n}<j<M_{n+1}$, $n\geq 0$, then we define
\[
f(u)= f(v_{M_{n}}) \frac{|\Chi^{M_{n+1}-j}(u)|}{|\Chi^{M_{n+1}-M_{n}}(v_{M_{n}})|}\frac{1}{\lambda_{m+M_{n}+1}\cdots\lambda_{m+j}},
\]
where $v_{M_{n}}=\prt^{j-M_{n}}(u) \in \Chi^{M_{n}}(v)$. Since the second factor on the right is at most 1, we have with (A3) that
\[
|f(u)|\leq \frac{1}{2^{n}}.
\]

Altogether we have defined $f$ on $V(v)$, and we see that $f\in c_0(V(v))$. It is easily verified that $f$ is a fixed point for $B_\lambda$ over $V(v)$.
\end{proof}

We note that condition (A1) is satisfied, with $N=1$ and $C=1$, for any symmetric tree. Thus Theorem \ref{pr-symcond} gives us a second, and easier proof of Theorem \ref{t-symsym}. 

Condition (A2) is particularly interesting for $X=c_0(V)$ because it implies that $B_\lambda$ is chaotic on that space if
\[
|\Chi^n(v)||\lambda_{m+1}\cdots\lambda_{m+n}|\to \infty
\]
holds uniformly for $v\in V$ as $n\to\infty$, where $m$ is such that $v\in\gen_m$, and $\lambda_{-n+1}\cdots\lambda_0 \to 0$ as $n\to\infty$ if $V$ has a free left end.

Incidentally, in view of (b), condition (A3) cannot hold for chaotic symmetric shifts on unrooted trees with a free left end.

Condition (A2) is also satisfied for any Rolewicz operator $\lambda B$ with $|\lambda|>1$, and condition (A3) is satisfied for any Rolewicz operator $\lambda B$ on $c_0(V)$ with $|\lambda|\geq 1$. By \cite[Example 2.5]{GrPa21}, continuity of $\lambda B$ on the stated spaces is equivalent to $\sup_{v\in V}|\Chi(v)|<\infty$. Let us call a branch starting from some $v\in V$ a \textit{free right end} if $|\Chi^n(v)|=1$ for all $n\geq 1$. Then we have the following.

\begin{corollary}\label{co-rol}
Let $V$ be a rooted tree, or an unrooted tree without a free left end, with $\sup_{v\in V}|\Chi(v)|<\infty$. 
 
\emph{(a)} Every Rolewicz operator $\lambda B$ with $|\lambda|>1$ is chaotic on $X=\ell^p(V)$, $1< p<\infty$, and on $X=c_0(V)$.

\emph{(b)} If $V$ has no free right end, then every Rolewicz operator $\lambda B$ with $|\lambda|\geq 1$ is chaotic on $X=c_0(V)$. 
\end{corollary}

We now show, as previously announced, that the necessary condition for chaos in Theorem \ref{pr-symneccond} is not, in general, sufficient, even for rooted trees. Our counter-examples are even Rolewicz operators, and they show that the conditions on $\lambda$ in the corollary are optimal.

\begin{example}\label{ex-nonsuff}
(a) Let $1<p< \infty$.  Then there exists a rooted tree $V$ with $|\Chi(v)| \leq 2$ for all $v\in V$ so that the unweighted backward shift $B$ is an operator on $\ell^p(V)$ that satisfies \eqref{eq-simpl3} (and is therefore mixing \cite[Corollary 4.5]{GrPa21}) but that is not chaotic.

We start the construction by considering the rooted binary tree $V_{\text{bin}}$, that is, the tree with root $v_0$ in which every vertex has exactly two children. Let
\[
M_n=\frac{n(n+1)}{2}, n\geq 0,
\]
and we will call the $M_n$-th generation
\[
I_n = \gen_{M_n}, n\geq 0.
\]
Since, for $n\geq 1$, $|I_{n-1}|=2^{M_{n-1}}$ and since there are $n$ generations from $I_{n-1}$ to $I_{n}$, we can partition $I_{n}$ into $2^{M_{n-1}}$ sets 
\[
I_{n,j}, j=1,\ldots, 2^{M_{n-1}}
\]
of cardinality $2^n$ according to their ancestors of degree $n$. 

Now, in any set $I_{n,j}$, $n\geq 1$, $j=1,\ldots, 2^{M_{n-1}}$, we choose (exactly) one vertex, which we will call \textit{distinguished}. Then, for any $v\in I_{n,j}$ that is not distinguished, we add $N_n$ vertices between $v$ and $\prt(v)$, which we choose so large that
\begin{equation}\label{eq-Nn}
\frac{2^n}{N_n^{1/p}}\leq \frac{1}{n}\Big(\frac{1}{2^{M_{n-1}}}-\frac{1}{2^{M_{n}}}\Big), n\geq 1,
\end{equation}
see Figure \ref{fig-symweight}.

\begin{figure}
\begin{tikzpicture}
\draw[fill] (0,0) circle (.5pt);

\draw[->,>=latex] (0,0) -- (1,.9);\draw[fill] (1,.9) circle (.5pt);
\draw[->,>=latex] (0,0) -- (1,-.9);\draw[fill] (1,-.9) circle (.5pt);

\draw[->,>=latex] (1,.9) -- (2,1.5);\draw[fill] (2,1.5) circle (.5pt);
\draw[->,>=latex] (1,.9) -- (2,.5);\draw[fill] (2,.5) circle (.5pt);
\draw[->,>=latex] (1,-.9) -- (2,-1.5);\draw[fill] (2,-1.5) circle (.5pt);
\draw[->,>=latex] (1,-.9) -- (2,-.5);\draw[fill] (2,-.5) circle (.5pt);

\draw[->,>=latex] (2,1.5) -- (2.25,1.575);\draw[fill] (2.25,1.575) circle (.5pt);
\draw[->,>=latex] (2.25,1.575) -- (2.5,1.65);\draw[fill] (2.5,1.65) circle (.5pt);
\draw[->,>=latex] (2.5,1.65) -- (2.75,1.725);\draw[fill] (2.75,1.725) circle (.5pt);
\draw[->,>=latex] (2.75,1.725) -- (3,1.8);\draw[fill] (3,1.8) circle (.5pt);

\draw[->,>=latex] (2,1.5) -- (2.25,1.425);\draw[fill] (2.25,1.425) circle (.5pt);
\draw[->,>=latex] (2.25,1.425) -- (2.5,1.35);\draw[fill] (2.5,1.35) circle (.5pt);
\draw[->,>=latex] (2.5,1.35) -- (2.75,1.275);\draw[fill] (2.75,1.275) circle (.5pt);
\draw[->,>=latex] (2.75,1.275) -- (3,1.2);\draw[fill] (3,1.2) circle (.5pt);

\draw[->,>=latex] (2,.5) -- (2.25,.575);\draw[fill] (2.25,.575) circle (.5pt);
\draw[->,>=latex] (2.25,.575) -- (2.5,.65);\draw[fill] (2.5,.65) circle (.5pt);
\draw[->,>=latex] (2.5,.65) -- (2.75,.725);\draw[fill] (2.75,.725) circle (.5pt);
\draw[->,>=latex] (2.75,.725) -- (3,.8);\draw[fill] (3,.8) circle (.5pt);

\draw[->,>=latex] (2,.5) -- (2.25,0.425);\draw[fill] (2.25,.425) circle (.5pt);
\draw[->,>=latex] (2.25,.425) -- (2.5,0.35);\draw[fill] (2.5,.35) circle (.5pt);
\draw[->,>=latex] (2.5,.35) -- (2.75,0.275);\draw[fill] (2.75,.275) circle (.5pt);
\draw[->,>=latex] (2.75,.275) -- (3,0.2);\draw[fill] (3,.2) circle (.5pt);

\draw[->,>=latex] (2,-1.5) -- (2.25,-1.425);\draw[fill] (2.25,-1.425) circle (.5pt);
\draw[->,>=latex] (2.25,-1.425) -- (2.5,-1.35);\draw[fill] (2.5,-1.35) circle (.5pt);
\draw[->,>=latex] (2.5,-1.35) -- (2.75,-1.275);\draw[fill] (2.75,-1.275) circle (.5pt);
\draw[->,>=latex] (2.75,-1.275) -- (3,-1.2);\draw[fill] (3,-1.2) circle (.5pt);

\draw[->,>=latex] (2,-.5) -- (2.25,-.575);\draw[fill] (2.25,-.575) circle (.5pt);
\draw[->,>=latex] (2.25,-.575) -- (2.5,-.65);\draw[fill] (2.5,-.65) circle (.5pt);
\draw[->,>=latex] (2.5,-.65) -- (2.75,-.725);\draw[fill] (2.75,-.725) circle (.5pt);
\draw[->,>=latex] (2.75,-.725) -- (3,-.8);\draw[fill] (3,-.8) circle (.5pt);

\draw[->,>=latex] (2,-.5) -- (2.25,-0.425);\draw[fill] (2.25,-.425) circle (.5pt);
\draw[->,>=latex] (2.25,-.425) -- (2.5,-0.35);\draw[fill] (2.5,-.35) circle (.5pt);
\draw[->,>=latex] (2.5,-.35) -- (2.75,-0.275);\draw[fill] (2.75,-.275) circle (.5pt);
\draw[->,>=latex] (2.75,-.275) -- (3,-0.2);\draw[fill] (3,-.2) circle (.5pt);

\draw[->,>=latex] (2,-1.5) -- (3,-1.8);\draw[fill] (3,-1.8) circle (.5pt);
\end{tikzpicture}
\caption{Adding vertices to a binary tree (set $I_{n,j}$ with ancestors)}%
\label{fig-symweight}
\end{figure}
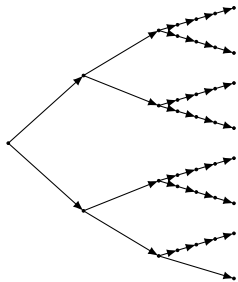

The new tree will be denoted by $V$. Since any vertex in $V$ has at most two children, the unweighted backward shift $B$ is an operator on $\ell^p(V)$, see \cite[Example 2.5]{GrPa21}.

We first show that \eqref{eq-simpl3} holds. Let $v\in V$. We construct a branch starting from $v$ in the following way. First, there is a path, of length $l$ say, from $v$ to some vertex in $I_{n}$ for some $n\geq 0$. From there we continue the branch in a unique way through successive distinguished vertices in $I_{k}$, $k\geq n+1$. Note that from the vertex in $I_{n}$ on we follow an original branch in $V_\text{bin}$.

Now, any vertex in $I_{k-1}$, $k\geq 1$, has $2^{k}$ descendants of degree $k$. Hence, for $k\geq n+1$,
\[
|\Chi^j(v)| \geq 2^k\quad\text{if $l+M_k-M_{n}\leq j <l+M_{k+1}-M_{n}$},
\]
and therefore
\[
\sum_{j=l+M_k-M_{n}}^{l+M_{k+1}-M_{n}-1} \frac{1}{|\Chi^j(v)|^{p/p^\ast}} \leq \frac{k+1}{2^{kp/p^\ast}},
\]
which implies that \eqref{eq-simpl3} holds for the operator $B$ on $\ell^p(V)$.

We will now show that $B$ is not chaotic. Otherwise there is a fixed point $f\in\ell^p(V)$ for $B$ with $f(v_0)=1$. We fix an integer
\[
n\geq 2\|f\|.
\]
Since the additional vertices have no influence on the action of $B$, we have that
\[
1=f(v_0)=\sum_{v\in I_{n}}f(v),
\]
and since $|I_{n}|=2^{M_{n}}$, there is some $w=:w_{n}\in I_{n}$ such that
\[
|f(w_n)|\geq \frac{1}{2^{M_{n}}}.
\]
Now, the set of children of $w_{n}$ of degree $n+1$ in the original tree $V_{\text{bin}}$ is given by some $I_{n+1,j}$, $1\leq j\leq 2^{M_{n}}$. Let $w_{n+1}$ be its distinguished vertex. Since $f$ is a fixed point for $B$, we have that
\[
|f(w_{n})|= \Big|\sum_{v\in I_{n+1,j}} f(v)\Big|\leq |f(w_{n+1})| + \sum_{\substack{v\in I_{n+1,j}\\v\neq w_{n+1}}} |f(v)|.
\]
But any of the $2^{n+1}-1$ vertices $v\in I_{n+1,j}\setminus\{w_{n+1}\}$ is non-distinguished, so that the added $N_{n+1}$ vertices all share the value of $f(v)$, whence
\[
\|f\|^p\geq N_{n+1} |f(v)|^p.
\]
Thus we have that 
\[
|f(w_{n+1})| \geq |f(w_{n})| - \frac{2^{n+1}}{N_{n+1}^{1/p}}\|f\|.
\]
Repeating this argument, we find successively distinguished vertices $w_k$, $k\geq n+1$, such that
\[
|f(w_{k})| \geq |f(w_{n})| - \sum_{j=n+1}^k\frac{2^{j}}{N_{j}^{1/p}}\|f\|\geq \frac{1}{2^{M_{n}}} - \sum_{j=n+1}^\infty\frac{2^{j}}{N_{j}^{1/p}}\|f\|.
\]
Using \eqref{eq-Nn}, we obtain that, for all $k\geq n+1$, 
\begin{align*}
|f(w_{k})| &\geq \frac{1}{2^{M_{n}}} - \sum_{j=n+1}^\infty\frac{1}{j}\Big(\frac{1}{2^{M_{j-1}}}-\frac{1}{2^{M_{j}}}\Big)\|f\|\\
& \geq \frac{1}{2^{M_{n}}} - \frac{1}{n+1}\frac{1}{2^{M_{n}}}\|f\|\geq \frac{1}{2^{M_n+1}},
\end{align*}
which contradicts the fact that $f\in\ell^p(V)$.

(b) Let $0<\lambda<1$. There exists a rooted tree $V$ with $|\Chi(v)| \leq 2$ for all $v\in V$ so that the Rolewicz operator $\lambda B$ is an operator on $c_0(V)$ that satisfies \eqref{eq-simpl3} (and is therefore mixing \cite[Corollary 4.5]{GrPa21}) but that is not chaotic. Note that by Corollary \ref{co-rol}(b) we cannot take the unweighted shift here.

We construct a counter-example for $\tfrac{1}{2^{1/3}}<\lambda<1$. For smaller values, the same ideas work when, in the initial tree, each vertex has $N$ children, where $N>\tfrac{1}{\lambda^3}$.

Thus let $\tfrac{1}{2^{1/3}}<\lambda<1$. We perform the same modifications on the rooted binary tree $V_\text{bin}$ as above, but this time with
\[
M_n=2^n, n\geq 0,
\]
and $N_n$ so large that
\begin{equation}\label{eq-Nn0}
\lambda^{N_n} \leq \frac{1}{n 2^{M_{n}-M_{n-1}}}\Big(\frac{1}{2^{M_{n-1}}}-\frac{1}{2^{M_{n}}}\Big), n\geq 1.
\end{equation}
We then consider the operator $\lambda B$ on the modified tree $V$, which is an operator on $c_0(V)$, see \cite[Example 2.5]{GrPa21}.

In order to show that \eqref{eq-simpl3} holds, let $v\in V$, and choose as before a branch starting from $v$ that first passes through some vertex in $I_{n}$ for some $n\geq 0$ and then successively through distinguished vertices in $I_{k}$, $k\geq n+1$. We need here a slightly stronger lower estimate on the number of descendants: for $k\geq n+1$,
\[
|\Chi^j(v)| \geq \max\big(2^{M_k-M_{k-1}},2^{j-(l+M_{k}-M_n)}\big) \quad\text{if $l+M_k-M_{n}\leq j <l+M_{k+1}-M_{n}$}.
\]
Since $2\lambda>1$, $j\to\max(2^{M_k-M_{k-1}}\lambda^j,2^{j-(l+M_{k}-M_n)}\lambda^j)$ achieves its minimum at $j=l+M_{k}-M_{n}+M_k-M_{k-1}$ so that, for $l+M_k-M_{n}\leq j <l+M_{k+1}-M_{n}$,
\[
|\Chi^j(v)|\lambda^j \geq 2^{M_k-M_{k-1}}\lambda^{l+M_{k}-M_{n}+M_k-M_{k-1}}= \lambda^{l-2^n} \big(2\lambda^3\big)^{2^{k-1}}.
\]
Since $\lambda>\tfrac{1}{2^{1/3}}$, the right-hand side tends to infinity as $k\to\infty$, so that \eqref{eq-simpl3} holds for the Rolewicz operator $\lambda B$ on $c_0(V)$.

On the other hand, suppose that $\lambda B$ is chaotic. Then there is a fixed point $f\in c_0(V)$ for $\lambda B$ with $f(v_0)=1$. Fix
\[
n\geq 2\|f\|.
\]
Since
\[
1=f(v_0)=\sum_{v\in I_{n}}\lambda^{l(v)}f(v),
\]
where $l(v)$ is the length of the path from $v_0$ to $v$ in $V$, and since $|\lambda|\leq 1$, there is some $w=:w_{n}\in I_{n}$ such that
\[
|f(w_n)|\geq \frac{1}{2^{M_{n}}}.
\]
We consider again the branch starting from $w_n$ that passes successively through distinguished vertices $w_k$ in $I_k$, $k\geq n+1$. Choose $j$ such that $w_{n+1}\in I_{n+1,j}$. Then
\begin{align*}
|f(w_{n})|&= \Big|\lambda^{M_{n+1}-M_n}f(w_{n+1})+ \sum_{\substack{v\in I_{n+1,j}\\v\neq w_{n+1}}} \lambda^{M_{n+1}-M_n+N_{n+1}}f(v)\Big|\\
&\leq |f(w_{n+1})| + \sum_{\substack{v\in I_{n+1,j}\\v\neq w_{n+1}}} \lambda^{N_{n+1}}\|f\|\leq |f(w_{n+1})| + 2^{M_{n+1}-M_n} \lambda^{N_{n+1}}\|f\|,
\end{align*}
and thus 
\[
|f(w_{n+1})| \geq |f(w_{n})| - 2^{M_{n+1}-M_n} \lambda^{N_{n+1}}\|f\|.
\]
We find inductively and then using \eqref{eq-Nn0} that, for $k\geq n+1$,
\begin{align*}
|f(w_{k})| &\geq |f(w_{n})| - \sum_{j=n+1}^k 2^{M_{j}-M_{j-1}} \lambda^{N_{j}}\|f\|\geq \frac{1}{2^{M_{n}}} - \sum_{j=n+1}^\infty2^{M_{j}-M_{j-1}}\lambda^{N_{j}}\|f\|\\
&\geq \frac{1}{2^{M_{n}}} - \sum_{j=n+1}^\infty\frac{1}{j}\Big(\frac{1}{2^{M_{j-1}}}-\frac{1}{2^{M_{j}}}\Big)\|f\|\geq \frac{1}{2^{M_{n}}} - \frac{1}{n+1}\frac{1}{2^{M_{n}}}\|f\|\geq \frac{1}{2^{M_n+1}},
\end{align*}
in contradiction to the fact that $f\in c_0(V)$.
\end{example}

In conclusion, if $|\lambda|\leq 1$ for $\ell^p(V)$, $1<p<\infty$, or if $|\lambda|<1$ for $c_0(V)$, we do not have handy characterizations of chaos for Rolewicz operators, unless the tree is symmetric. The best we have to offer are the general characterizations of Theorems
\ref{t-charplambda}, \ref{t-char0lambda}, and \ref{t-charlambdabil_simple}.

\subsection{The Bergman shift}\label{subs-Bergm}
In this short subsection we return to a weighted backward shift that we have already studied in  \cite{GrPa21}.

\begin{example}
Let $V$ be an arbitrary locally finite rooted tree with root $v_0$, and let $q\geq 1$ be a real constant. We consider the weighted backward shift $B_{\lambda,q}$ on $V$ of parameter $q$ that is given by
\[
\lambda_{v,q}=\frac{1}{\sqrt{|\Chi(\prt(v))|}}\sqrt{\frac{n_v+q-1}{n_v}}, \ v\neq v_0,
\]
where $n_v\geq 1$ is such that $v\in\gen_{n_v}$. This is the adjoint of the Dirichlet shift, see \cite{CPT17} and \cite[Example 4.8]{GrPa21}. By analogy with the special case when $V=\NN_0$ we call $B_{\lambda,q}$ a \textit{Bergman shift} on $V$.

In \cite{GrPa21} we have shown that $B_{\lambda,q}$ is mixing on $\ell^2(V)$ for $q>1$, while it is not hypercyclic if $q=1$. We will show here that it is even chaotic for $q>1$. For this it suffices to construct, for any $v\in V$, a fixed point $f$ for the operator $B_{\lambda,q}$ in $\ell^2(V(v))$ with $f(v)=1$. This is straightforward. We define $f(v)=1$ and recursively, for any $u\in V(v)$, $u\neq v$,
\[
f(u) = \frac{f(\prt(u))}{\lambda_{u,q}|\Chi(\prt(u))|}.
\]
We then obviously have a fixed point for $B_{\lambda,q}$ over $V(v)$. Now, for any $u\in V(v)$, $u\neq v$,
\[
f(u)= \frac{1}{\sqrt{|\Chi(\prt(u))}|\sqrt{|\Chi(\prt^2(u))}|\cdots\sqrt{|\Chi(v)}|}\sqrt{\frac{n_u(n_u-1)\cdots(m+1)}{(n_u-1+q)(n_u-2+q)\cdots(m+q)}},
\]
where $m$ is such that $v\in\gen_m$. It follows that, for all $k\geq 1$,
\[
\sum_{u\in \Chi^k(v)}|f(u)|^2 = \frac{(m+k)(m+k-1)\cdots(m+1)}{(m+k-1+q)(m+k-2+q)\cdots(m+q)}\sim C \frac{1}{k^{q-1}}
\]
for some constant $C>0$ depending on $m$ and $q$, which implies that $f\in \ell^2(V(v))$ since $q>1$. Altogether we have shown that $B_{\lambda,q}$ is chaotic.
\end{example}

\section{Appendix 1: A characterization of hypercyclic and mixing weighted backward shifts on trees}\label{s-appendix}
In \cite{GrPa21} we have characterized hypercyclic and mixing weighted backward shifts on the spaces $\ell^p(V)$, $1\leq p<\infty$, and $c_0(V)$ over a tree $V$. Our aim there was to obtain a characterization directly in terms of the weights involved. 

Using very much the same ideas one can also arrive at a characterization of hypercyclic and mixing weighted backward shifts on very general Fr\'echet sequence spaces over the tree $V$. The conditions here are much less explicit and require to find a suitable sequence of elements in the underlying space; that is, our results are in the spirit of Theorems \ref{chaos-fixedpoint} and \ref{t-chunrooted}. The characterizations are well adapted to showing that any chaotic weighted backward shift is mixing, see Remark \ref{rem-GMM}(a).

We continue to assume our trees to be leafless, which is, anyway, necessary for having a weighted backward shift hypercyclic.

For the results in this section we can relax the hypothesis of unconditionality of the basis $(e_v)_{v\in V}$. Let $V$ be a tree and $X$ a Fr\'echet sequence space over $V$. Let the generations be enumerated according to some vertex $v_0$. We will demand the following condition (B):

\begin{enumerate}
\item[(B0)] for any $v\in V$, $e_v\in X$;
\item[(B1)] for any $f\in X$ and any $n$, the series
\[
f\chi_{\gen_n} = \sum_{v\in \gen_n} f(v) e_v
\]
converge unconditionally in $X$;
\item[(B2)] for any $f\in X$,
\[
f= \sum_{n} f\chi_{\gen_n} = \sum_{n<0} f\chi_{\gen_n} + \sum_{n\geq 0} f\chi_{\gen_n}
\]
converges in $X$, where the sums extend over the $n$ with non-empty generations $\gen_n$;
\item[(B3)] for any $\varepsilon>0$ there is some $\delta>0$ such that, for any $f\in X$, any $n$ and any $W\subset \gen_n$,
\[
\text{if $\|f\chi_{\gen_n}\|<\delta$ then $\|f\chi_{W}\|<\varepsilon$.}
\]
\end{enumerate}

Condition (B) seems quite natural in view of the fact that there is a natural ordering between the generations, but none within the generations. For example, for the classical trees $\NN_0$ and $\ZZ$, (B) simply says that $(e_n)_n$ is a basis in $X$.

Note that condition (B3) expresses a kind of uniform unconditionality of the convergence in (B1). 

Clearly, any unconditional basis $(e_v)_{v\in V}$ satisfies condition (B), and under condition (B), the set $\{e_v : v\in V\}$ has a dense linear span in $X$.

In the following, the \textit{support} of a sequence $f\in\KK^V$ is the set $\{v\in V: f(v)\neq 0\}$.

\begin{theorem}\label{t-charHCroot}
Let $V$ be a rooted tree and $\lambda=(\lambda_v)_{v\in V}$ a weight. Let $X$ be a Fr\'echet sequence space over $V$ in which $(e_v)_{v\in V}$ satisfies \emph{(B)}, and suppose that the weighted backward shift $B_\lambda$ is an operator on $X$. 

{\rm (a)} The following assertions are equivalent:
\begin{enumerate}[label={\rm (\roman*)}]
    \item $B_\lambda$ is hypercyclic;
    \item $B_\lambda$ is weakly mixing;
    \item there is an increasing sequence $(n_k)_k$ of positive integers such that, for each $v\in V$, there is some $f_{v,n_k}\in X$ of support in $\Chi^{n_k}(v)$ such that, for all $v\in V$,
		\[
		f_{v,n_k}\to 0 \text{ in } X \quad\text{and}\quad \sum_{u\in \Chi^{n_k}(v)} \lambda(v\to u) f_{v,n_k}(u) \to 1
		\]
		as $k\to\infty$.
\end{enumerate}

{\rm (b)} $B_\lambda$ is mixing if and only if condition {\rm (iii)} holds for the full sequence $(n_k)_k=(n)_n$.
\end{theorem}

\begin{proof} (a) (iii) $\Longrightarrow$ (ii). We apply the Hypercyclicity Criterion, see \cite[Theorem 1.1]{GrPa21} for the notation used here; it also holds for separable Fr\'echet spaces. Let $X_0=Y_0=\vect\{e_v:v\in V\}$, which is dense in $X$ by assumption. Since $B_\lambda$ is a weighted backward shift on a rooted  tree, we have that for any $f\in X_0$ that
\[
B_\lambda^n f= 0
\]
if $n$ is sufficiently large. Next, using the elements given by (iii), we define 
\[
R_{n_k} : Y_0\to X, \ k\geq 1,
\]
by
\[
R_{n_k} e_v = f_{v,n_k}, \ v\in V,
\]
followed by linear extension to $Y_0$. By hypothesis we then have that
\[
R_{n_k} g\to 0
\]
for all $g\in Y_0$. Finally, since $f_{v,n_k}$ has its support in $\Chi^{n_k}(v)$, we find that
\[
B_\lambda^{n_k} R_{n_k} e_v = \sum_{u\in\Chi^{n_k}(v)} \lambda(v\to u) f_{v,n_k}(u) e_v \to e_v
\]
as $k\to\infty$, and hence 
\[
B_\lambda^{n_k} R_{n_k} g \to g
\]
for all $g\in Y_0$. Thus, $B_\lambda$ satisfies the Hypercyclicity Criterion and is therefore weakly mixing.

The implication (ii) $\Longrightarrow$ (i) holds for any operator.

(i) $\Longrightarrow$ (iii). Let $\|\cdot\|$ be an F-norm inducing the topology of $X$.

Let $F\subset V$ be a finite set, $N$ a positive integer and $\varepsilon >0$. Since the embedding of $X$ into $\KK^V$ is continuous and $F$ is finite, there is some $\delta > 0$ such that, for any $f\in X$,
\begin{equation}\label{eq-FK}
\text{if } \|f-\chi_F\|<\delta \text{ then } |f(v)|\geq \tfrac{1}{2} \text{ for all } v\in F.
\end{equation}

Next, it follows from (B2) and the Banach-Steinhaus theorem that, for any $\delta >0$, there is some $\eta>0$ such that, for any $f\in X$,
\[
\text{if } \|f\|<\eta \text{ then }  \|f\chi_{\gen_n}\|<\delta \text{ for all } n.
\]
Together with (B3) we obtain that, for the given $\varepsilon>0$, there is some $\eta >0$ such that, for any $f\in X$,
\begin{equation}\label{eq-BS}
\text{if } \|f\|<\eta \text{ then }  \|f\chi_W\|<\varepsilon \text{ whenever $W\subset \gen_n$ for some $n$}.
\end{equation}

Now, (i) implies that there is a hypercyclic vector $f\in X$ for $B_\lambda$ with $\|f\|<\eta$, and hence some $n\geq N$ such that
\[
\|B_\lambda^n f -\chi_F\| < \delta,
\]
from which we get by \eqref{eq-FK} that, for all $v\in F$,
\[
|(B_\lambda^n f)(v)| \geq \tfrac{1}{2}.
\] 

Let us consider, for $v\in F$,
\[
\alpha_{v}=(B_\lambda^n f)(v)=\sum_{u\in \Chi^n(v)} \lambda(v\to u) f(u),
\]
and set
\[
f_{v,F,N,\varepsilon} = \frac{1}{\alpha_{v}} f \chi_{\Chi^n(v)},\ v\in F.
\]
Then we have that each sequence $f_{v,F,N,\varepsilon}$, $v\in F$, has its support in $\Chi^n(v)$ and  
\[
\sum_{u\in \Chi^n(v)} \lambda(v\to u) f_{v,F,N,\varepsilon}(u) =1.
\]
Moreover, since $\|f\|<\eta$, it follows from \eqref{eq-BS} that
\[
\|f_{v,F,N,\varepsilon}\| \leq  (\tfrac{1}{|\alpha_v|} +1)\|f \chi_{\Chi^n(v)}\| < 3\varepsilon,\ v\in F;
\]
here we have used the property of F-norms that $\|\alpha x\|\leq (|\alpha|+1)\|x\|$ for each scalar $\alpha$ and vector $x$, see \cite[Section 2.1]{GrPe11}.

In order to deduce (iii) it now suffices to apply this construction to an increasing sequence $(F_k)_k$ of finite subsets of $V$ whose union is $V$, an increasing sequence $(N_k)_k$ of positive integers and a null sequence $(\varepsilon_k)_k$ of strictly positive numbers.

(b) The proof is very similar to that of (a) if one notes that whenever the Hypercyclicity Criterion holds for the full sequence of positive integers then the operator is mixing; in the opposite direction, we replace the existence of a hypercyclic vector by the mixing property.
\end{proof}

As the proof shows, for the implication (iii) $\Longrightarrow$ (ii) it suffices that $\{e_v:v\in V\}$ has a dense linear span in $X$. 

The result has an analogue for unrooted  trees.

\begin{theorem}\label{t-charHCunroot}
Let $V$ be an unrooted  tree and $\lambda=(\lambda_v)_v$ a weight. Let $X$ be a Fr\'echet sequence space over $V$ in which $(e_v)_{v\in V}$ satisfies \emph{(B)}, and suppose that the weighted backward shift $B_\lambda$ is an operator on $X$. 

{\rm (a)} The following assertions are equivalent:
\begin{enumerate}[label={\rm (\roman*)}]
    \item $B_\lambda$ is hypercyclic;
    \item $B_\lambda$ is weakly mixing;
    \item there is an increasing sequence $(n_k)_k$ of positive integers such that, for each $v\in V$, there are $f_{v,n_k}\in X$ of support in $\Chi^{n_k}(v)$ and $g_{v,n_k}\in X$ of support in $\Chi^{n_k}(\prt^{n_k}(v))$ such that, for all $v\in V$,
		\[
		f_{v,n_k}\to 0 \text{ and } g_{v,n_k}\to 0 \text{ in } X,
		\]
		\[		
		\sum_{u\in \Chi^{n_k}(v)} \lambda(v\to u) f_{v,n_k}(u) \to 1,
		\]
		and 
		\[		
		\Big(\lambda(\prt^{n_k}(v)\to v)+\sum_{u\in \Chi^{n_k}(\prt^{n_k}(v))} \lambda(\prt^{n_k}(v)\to u) g_{v,n_k}(u)\Big) e_{\prt^{n_k}(v)}  \to 0 \text{ in } X
		\]
		as $k\to\infty$.
\end{enumerate}

{\rm (b)} $B_\lambda$ is mixing if and only if condition {\rm (iii)} holds for the full sequence $(n_k)_k=(n)_n$.
\end{theorem}

\begin{proof}
The proof is similar to that of Theorem \ref{t-charHCroot}, so we only indicate the necessary modifications. But in order to avoid a minor technical problem in the proof of the second implication below we note that it suffices to do the proof for the unweighted backward shift $B$. In fact, a simple but slightly tedious verification shows that the characterizing condition in (iii) is invariant under the passage form the dynamical system $B_\lambda:X\to X$ to the conjugate dynamical system $B:X_\mu\to X_\mu$, where the weight $\mu$ on $V$ depends on $\lambda$ as described in Section \ref{s-notation} and $X_\mu=\{f\in\KK^V : (f(v)\mu_v)_{v\in V}\in X\}$ carries its natural topology. Of course, this reduction could also have been done in the proof of Theorem \ref{t-charHCroot}.

Thus we now assume that $\lambda_v=1$ for all $v\in V$.

(a) (iii) $\Longrightarrow$ (ii). We apply here the variant of the Hypercyclicity Criterion as given in \cite[Proposition 5.1]{GrPa21}, which also holds for separable Fr\'echet spaces. As before we set $X_0=Y_0=\vect\{e_v:v\in V\}$, and we define $I_{n_k}:X_0\to X$ and $R_{n_k}:Y_0\to X$, $k\geq 1$, by
\[
I_{n_k} e_v = e_v+g_{v,n_k},
\]
\[
R_{n_k} e_v = f_{v,n_k},
\]
with linear extension to $X_0=Y_0$. Then the hypotheses of \cite[Proposition 5.1]{GrPa21} are satisfied; for this, note in particular that
\[
B^{n_k} I_{n_k} e_v = \Big(1+\sum_{u\in \Chi^{n_k}(\prt^{n_k}(v))} g_{v,n_k}(u)\Big) e_{\prt^{n_k}(v)}.
\]

(i) $\Longrightarrow$ (iii). Let $\|\cdot\|$ be an F-norm inducing the topology of $X$.

Let $F\subset V$ be a finite set, $N$ a positive integer and $\varepsilon >0$. We choose $\delta>0$ and $\eta>0$ with $\eta\leq\delta$ so that \eqref{eq-FK} and \eqref{eq-BS} hold. We next form a subset $G$ of $F$ by choosing exactly one vertex from $F$ per generation. Then, by topological transitivity of $B$, there is some $f\in X$ and some $n\geq N$ such that
\begin{equation}\label{eq-tt}
\big\|f-\chi_G\big\|< {\eta}\quad\text{and}\text \quad \big\|B^nf-\chi_F\big\|<\eta,
\end{equation}
\begin{equation}\label{eq-tt2}
\Chi^n(F)\cap G=\varnothing\quad\text{and}\text \quad \prt^n(G)\cap F=\varnothing
\end{equation}
and, for any $v,w\in F$ that belong to the same generation,
\begin{equation}\label{eq-tt3}
\prt^n(v)=\prt^n(w).
\end{equation}

It follows from \eqref{eq-tt2} that, for any $v\in F$, $\chi_G\chi_{\Chi^n(v)}=0$. Arguing as in the proof of Theorem \ref{t-charHCroot} and using \eqref{eq-FK}, \eqref{eq-BS}, and \eqref{eq-tt} we find, for any $v\in F$, some $f_{v,F,N,\varepsilon}\in X$ of support in $\Chi^n(v)$ such that
\[
\sum_{u\in \Chi^n(v)} f_{v,F,N,\varepsilon}(u) =1
\]
and
\[
\|f_{v,F,N,\varepsilon}\| <3\varepsilon.
\]

Next, let again $v\in F$. By construction, there is a unique $w\in G$ in the same generation as $v$, and by \eqref{eq-tt3} we have that $\prt^n(v)=\prt^n(w)$. Then we set
\[
g_{v,F,N,\varepsilon} = f\chi_{\Chi^n(\prt^n(w))}-e_w.
\]
Thus $g_{v,F,N,\varepsilon}\in X$ has its support in $\Chi^n(\prt^n(w))=\Chi^n(\prt^n(v))$; moreover, \eqref{eq-BS}, \eqref{eq-tt}, and the fact that $G$ contains no vertex from $\Chi^n(\prt^n(v))$ other than $w$ implies that 
\[
\|g_{v,F,N,\varepsilon}\|<\eps.
\]
On the other hand, it follows from \eqref{eq-tt2} that $\chi_F\chi_{\{\prt^n(w)\}}=0$; we can then deduce from \eqref{eq-BS} and \eqref{eq-tt} that
\begin{align*}
\Big\|\Big(1+\sum_{u\in \Chi^{n}(\prt^{n}(v))}g_{v,F,N,\varepsilon}(u)\Big) e_{\prt^{n}(v)}\Big\|&=\Big\|\Big(1+\sum_{u\in \Chi^{n}(\prt^{n}(w))}g_{v,F,N,\varepsilon}(u)\Big) e_{\prt^{n}(w)}\Big\|\\
&=\big\| \big(B^n f\big)\chi_{\{\prt^n(w)\}}\big\|<\eps.
\end{align*}

To finish the proof we now perform this construction with an increasing sequence $(F_k)_k$ of finite subsets of $V$ whose union is $V$, an increasing sequence $(N_k)_k$ of positive integers and a null sequence $(\varepsilon_k)_k$ of strictly positive numbers.

(b) This follows with the ideas from the proof of (a) just like in the proof of Theorem \ref{t-charHCroot}.
\end{proof}

The two results generalize the corresponding known results for the trees $\NN_0$ and $\ZZ$, see \cite[Theorems 4.8 and 4.13]{GrPe11}. Recall that, in the case of $\NN_0$, it suffices to have the characterizing condition only for the root $v_0=0$, see \cite[Lemma 4.2]{GrPe11}, which is not the case for general rooted trees, see \cite[Example 4.6]{GrPa21}.

\section{Appendix 2: Chaos and the boundary of a tree}\label{s-chaosPoisson}
We will discuss here a relationship of our work on chaotic weighted backward shifts with harmonic functions on trees. The following discussion was motivated and influenced by recent work of Abakumov, Nestoridis and Picardello \cite{ANP17}, \cite{ANP21} on certain universality phenomena on trees, which continue to be leafless.

Harmonic functions on rooted trees were first considered by Cartier \cite{Car72}, \cite{Car73}. For this purpose he defined the  boundary of a tree, also identified as its Poisson boundary, see \cite[Example 14.30]{LyPe16}. We will give here a characterization of chaos for weighted backward shifts on rooted trees in terms of harmonic functions on the tree, and we will also provide a probabilistic characterization on the boundary. 

The relationship between chaotic weighted backward shifts and analysis on the tree is more transparent for the unweighted backward shift $B$. This will not restrict the generality of our discussion because of the by now familiar conjugacy of any weighted backward shift $B_\lambda$ on a Fr\'echet sequence space $X$ over a tree $V$ with the unweighted backward shift $B$ on a weighted space $X_\mu$, see Section \ref{s-notation}. So, throughout this section, we will only consider the unweighted backward shift.

We first recall the setup. Given a rooted tree $T=(V,E)$ with root $v_0$, the generations $\gen_n$, $n\geq 0$, define a projective system together with the powers of the map $\prt$,
\[
\gen_m\to \gen_n, \quad v\to \prt^{m-n}(v)
\]
for $m\geq n\geq 0$. The corresponding projective (or inverse) limit
\[
\partial T = \varprojlim_{n\geq 0} \gen_n
\]
is called the \textit{boundary} of the tree. It consists of all possible branches or \textit{infinite geodesics} starting from the root $v_0$, 
\begin{equation}\label{eq-branch}
\zeta=(v_0,v_1,v_2,\ldots).
\end{equation}

The boundary $\partial T$ is endowed with the projective limit topology, based on the discrete topology on the generations $\gen_n$; it is the topology inherited from $\prod_{n=0}^\infty \gen_n$ under its product topology. A base of the topology is given by the sets
\[
\partial T_v, \ v\in V
\]
of all infinite geodesics passing through $v$; note that these sets coincide with the boundary of the subtree $V(v)$. In this way, $\partial T$ turns into a separable completely metrizable topological space (that is, a Polish space) that is totally disconnected. If the tree is locally finite then $\partial T$ is compact.

Now suppose that, for any vertex $v\in V$, there are transition probabilities 
\[
p(v,u)>0,\ u\in \Chi(v), \text{ with $\sum_{u\in\Chi(v)} p(v,u)=1$}.
\]
Let $P$ denote the family $(p(v,u))_{v\in V, u\in \Chi(v)}$. Note that the existence of such a family requires that the tree is leafless. Then a function $h$ on $V$ is called \textit{harmonic} or \textit{$P$-harmonic} if, for any $v\in V$,
\[
h(v) = \sum_{u\in\Chi(v)} p(v,u) h(u).
\]
We remark that many texts also allow a transition probability $p(v,\prt(v))>0$, see for example \cite{Car72}, \cite{Car73}, \cite{Ana11}, \cite{LyPe16}. This is quite natural when the tree is not considered to be directed. We follow here \cite{PiWo19} and \cite{ANP21}; they call $P$ a \textit{forward-only transition operator}.

Now, for any $n\geq 0$, define transition probabilities of degree $n$ from the root $v_0$ to vertices $v\in V$ by setting
\[
\mu^P_v := \prod_{k=1}^n p(v_{k-1},v_k),
\]
where $v\in \gen_n$ and $(v_0,\ldots,v_n)$ is the path from $v_0$ to $v=v_n$; note that $\mu^P_{v_0}=1$. Then we see that $h$ is $P$-harmonic if and only if, for any $v\in V$,
\[
\mu^P_v h(v) = \sum_{u\in\Chi(v)}\mu^P_u h(u), 
\]
which is equivalent to saying that $(\mu^P_v h(v))_{v\in V}$ is a fixed point for the (unweighted) backward shift $B$. 

For the definition of the weighted space $X_\mu$ we refer to Section \ref{s-notation}. In view of Theorem \ref{chaos-positive} we have the following.

\begin{theorem}\label{chaos-harmonic}
Let $T=(V,E)$ be a rooted tree, $X$ be a Fr\'echet sequence space over $V$ in which $(e_v)_{v\in V}$ is an unconditional basis, and suppose that the backward shift $B$ is an operator on $X$. Let $P$ be a forward-only transition operator on $V$. 

Then $B$ is chaotic on $X$ if and only if there exists a strictly positive $P$-harmonic function $h$ on $V$ such that
\[
h\in X_{\mu^P}.
\]
\end{theorem}

Note that one may choose the transition operator $P$ arbitrarily. It might be natural, for example, to take for a locally finite tree
\[
p(v,u) = \frac{1}{|\Chi(v)|}, \ v\in V, u\in \Chi(v).
\]
Let us mention that if $X=\ell^p(V)$, $1\leq p<\infty$, then the space of $P$-harmonic functions in $X_{\mu^P}$ is a particular \textit{Bergman space} as introduced in \cite{CCPS18}.

The previous characterization can also be expressed in terms of martingales. For this we endow the boundary $\partial T$ with its Borel $\sigma$-algebra, denoted by $\mathcal{A}=\mathcal{B}(\partial T)$. It is the $\sigma$-algebra generated by the sets $\partial T_v$, $v\in V$. We will also need the sub-$\sigma$-algebras
\[
\mathcal{A}_n = \sigma (\{\partial T_v: v\in \gen_n\}),\ n\geq 0,
\] 
that is, the $\sigma$-algebras generated by the sets $\partial T_v$ with vertices $v$ from generation $n$. Then $(\mathcal{A}_n)_{n\geq 0}$ is a filtration. Finally, let $\mu$ be a probability measure on $(\partial T,\mathcal{A})$ such that $\mu(\partial T_v)>0$ for all $v\in V$.

To say that $(\varphi_n)_{n\geq 0}$ is a martingale for $(\mathcal{A}_n)_{n\geq 0}$ means that, for each $n\geq 0$, $\varphi_n$ is $\mu$-integrable, 
and that for each $v\in \gen_n$ it takes on $\partial T_v$ the constant value
\begin{align*}
\varphi_n|_{\partial T_v} &= \frac{1}{\mu(\partial T_v)}\int_{\partial T_v} \varphi_{n+1} d\mu\\
&= \frac{1}{\mu(\partial T_v)}\sum_{u\in\Chi(v)}\mu(\partial T_u) \varphi_{n+1}|_{\partial T_u};
\end{align*}
note that $(\partial T_u)_{u\in\Chi(v)}$ is a partition of $\partial T_v$. This is equivalent to saying that
\[
(\mu(\partial T_v) \varphi_{|v|}|_{\partial T_v})_{v\in V}
\]
is a fixed point for the backward shift $B$ on $V$; here, $|v|$ denotes the distance of $v$ from the root, that is, $v\in \gen_{|v|}$. Together with Theorem \ref{chaos-positive} this proves one implication of the following characterization, the other one being obtained in the same spirit.

Recall that a Borel measure is said to be of full topological support if it is strictly positive on every non-empty open set. In our context it means that $\mu(\partial T_v)>0$ for all $v\in V$.

\begin{theorem}\label{chaos-martingale}
Let $T=(V,E)$ be a rooted tree, $X$ be a Fr\'echet sequence space over $V$ in which $(e_v)_{v\in V}$ is an unconditional basis, and suppose that the backward shift $B$ is an operator on $X$. Let $\mu$ be a probability measure on $(\partial T,\mathcal{A})$ of full topological support.

Then $B$ is chaotic on $X$ if and only if there exists a strictly positive martingale $(\varphi_n)_{n\geq 0}$ for the filtration $(\mathcal{A}_n)_{n\geq 0}$ such that
\[
(\varphi_{|v|}|_{\partial T_v})_{v\in V} \in X_{(\mu(\partial T_v))_v}.
\]
\end{theorem}

We add a third characterization of chaos directly in terms of probabilities on the boundary $\partial T$. Let $\mu$ be a probability measure on $(\partial T,\mathcal{A})$. Then the sequence $(\mu(\partial T_v))_{v\in V}$ is a fixed point for the backward shift $B$. Indeed, for any $v\in V$, 
\[
\mu(\partial T_v) = \sum_{u\in\Chi(v)}\mu(\partial T_u)
\]
because $(\mu(\partial T_u))_{u\in \Chi(v)}$ is a partition of $\partial T_v$. Conversely, if $(f(v))_v$ is a positive fixed point for $B$ with value 1 at the root, then, for any $n\geq 0$,
\[
\mu_n(\partial T_v):=f(v),\ v\in \gen_n
\]
defines a probability measure on $\mathcal{A}_n$. By the Daniell-Kolmogorov consistency (or extension) theorem, the $\mu_n$, $n\geq 0$, can be extended to a probability measure $\mu$ on $(\partial T,\mathcal{A})$, see \cite{Par05} (for a more general result see Proposition \ref{p-miflows} below). Thus, again by Theorem \ref{chaos-positive}, we have the following.

\begin{theorem}\label{chaos-proba}
Let $T=(V,E)$ be a rooted tree, $X$ be a Fr\'echet sequence space over $V$ in which $(e_v)_{v\in V}$ is an unconditional basis, and suppose that the backward shift $B$ is an operator on $X$. 

Then $B$ is chaotic on $X$ if and only if there exists a probability measure $\mu$ on $(\partial T,\mathcal{A})$ of full topological support such that
\[
(\mu(\partial T_v))_{v\in V} \in X.
\]
\end{theorem}

\section{Epilogue}\label{s-epi}

We have kindly been informed by Nikolaos Chalmoukis that some of our results in Section \ref{s-backwinv} on rooted trees are strongly linked with investigations of potential theory on trees. We explain this link, which is provided by the notion of capacity, in the first subsection. We then discuss, in the same spirit, the bearing of our results on potential theory on unrooted trees, before returning to rooted trees once more.

\subsection{Potential theory on rooted trees I}\label{subs-rootI}

Let $T=(V,E)$ be a rooted directed tree, with or without leaves. The boundary $\partial T$ of $T$, its subsets $\partial T_v$, $v\in V$, and its Borel sets are defined as in Appendix~2, with the only difference that the branches starting from $v_0$, which define the points in $\partial T$, see \eqref{eq-branch}, may be of finite length.

The notion of capacity can be defined in the following way.

\begin{definition}\label{def-cap}
Let $T=(V,E)$ be a rooted tree with root $v_0$, $w=(w_v)_{v\in V}$ a positive weight on $V$, and $1<p<\infty$. Then the \textit{$p$-capacity} of the boundary $\partial T$ with respect to $w$ is given by 
\[
\capac_p(\partial T,w)= \sup\Big\{ \mu(\partial T)^p : \text{ $\mu$ a positive Borel measure on $\partial T$ with $\sum_{v\in V}(\mu(\partial T_v)w_v)^{p^\ast}\leq 1$}\Big\}.
\]
\end{definition}

For locally finite, leafless rooted trees, this appears in \cite[Theorem 1]{ChLe19} and \cite[Theorem A]{ArLe23} if $1<p<\infty$ and the weight $w=1$.

In potential theory, a \textit{flow} is a function $f:V\to \KK$ such that, for any $v\in V$ that is not a leaf,
\[
f(v) = \sum_{u\in\Chi(v)} f(u),
\]
where the series are supposed to be unconditionally convergent; in other words, a flow is a backward invariant sequence in the sense of Definition \ref{d-backwinv}. As explained at the end of Appendix~2, for leafless trees, there is a one-to-one correspondence between positive Borel measures $\mu$ on $\partial T$ and positive flows $f$ on $V$ given by
\[
f(v)= \mu(\partial T_v),\ v\in V.
\]
This extends to trees with leaves; it suffices to replace any leaf by an infinite branch and apply the result for leafless trees. Note that if $f(v_0)=1$ then $\mu$ is a probability measure.

Therefore we have that
\[
\capac_p(\partial T,w)= \sup\Big\{ f(v_0)^p : \text{ $f$ a positive flow on $V$ with $\sum_{v\in V}(f(v)w_v)^{p^\ast}\leq 1$}\Big\}.
\]
Now, by a standard argument, we find that
\[
\capac_p(\partial T,w)= \Big(\inf\Big\{ \sum_{v\in V}(f(v)w_v)^{p^\ast} : \text{ $f$ a positive flow on $V$ with $f(v_0)=1$}\Big\}\Big)^{-p/p^\ast};
\]
this is the definition given in \cite[Section 16.1]{LyPe16} for $p=2$. The value
\[
\mathcal{E}_p(f,w)= \sum_{v\in V}(f(v)w_v)^{p^\ast}
\]
is called the \textit{$p$-energy} of a positive flow (or the corresponding measure), see \cite{ChLe19}. A flow $f$ with $f(v_0)=1$ is called a \textit{unit flow}.
  	
Theorem \ref{t-charpinv}(a) determines the minimal norm for (real or complex) unit flows. Since it turns out that a unit flow with minimal norm is positive, one may restrict attention to such flows. Thus we get the following result from Theorem \ref{t-charpinv}, where for $\zeta=(v_0,v_1,v_2,\ldots)\in \partial T$, which may be a finite or infinite sequence, we write
\[
\lim_{v\to\zeta} f(v)= \lim_{n\nearrow} f(v_n).
\]

\begin{theorem}\label{t-charpinvbis}
Let $T=(V,E)$ be a rooted tree with root $v_0$, $w=(w_v)_{v\in V}$ a positive weight on $V$, and $1< p<\infty$. Suppose that $B$ is defined on $\ell^p(V,w)$. Then
\[
\capac_{p^\ast}(\partial T,w)= r_{p}(V,w)^{-p^\ast}.
\]

If $\capac_{p^\ast}(\partial T,w)>0$, there exists a unique positive unit flow $f$ on $V$ of minimal energy, and for very boundary point $\zeta\in\partial T$ we have that
\[
\lim_{v\to\zeta} f(v)= \mu(\{\zeta\})\in [0,1],
\]
where $\mu$ is the corresponding probability measure on $\partial T$.
\end{theorem}

Theorem \ref{t-charpinv}(b) provides an explicit formula for the minimal flow $f$, and $(f(v_n))_{n\geq 0}$ is decreasing for every branch $(v_0, v_1,v_2,\ldots)$ starting from $v_0$, see Remark \ref{r-decr2}.

The measure corresponding to the unique positive unit flow on $V$ of minimal energy is the harmonic measure on $\partial T$, see \cite[Theorem 16.2]{LyPe16}, called the equilibrium measure in \cite{ArLe23}. In \cite[Section 7]{ArLe23} one can also find an expression for $\capac_p(\partial T,w)$ as a continued fraction when $p=2$ and $w=1$; see also \cite[p.\ 57]{Soa94}.

We can now express chaos on rooted trees in terms of capacity, see Theorem \ref{t-charpmu}; note that $\partial T_v$, $v\in V$, coincides with the boundary of the tree $V(v)$ of descendants of $v$.

\begin{corollary}
Let $T=(V,E)$ be a rooted tree with root $v_0$, $w=(w_v)_{v\in V}$ a (real or complex) weight on $V$, and $1< p<\infty$. Suppose that the backward shift $B$ is an operator on $\ell^p(V,w)$. Then $B$ is chaotic on $\ell^p(V,w)$ if and only if, for every $v\in V$, 
\[
\capac_{p^\ast}(\partial T_v,|w|)>0.
\]
\end{corollary}

\subsection{The boundary of an unrooted tree}\label{subs-boundaryunrooted}

We pass to unrooted trees, again with or without leaves. These do not seem to have been investigated in potential theory yet. We first need a notion of boundary. By slightly modifying the terminology in Sections \ref{s-notation} and \ref{s-chaosPoisson}, we will call a \textit{branch} any maximal sequence
\[
\zeta = (v_n)_{n< m}
\]
in the tree with $v_{n-1}=\prt(v_{n})$ for $n<m$, where $m\in \ZZ\cup\{\infty\}$; for a finite $m$ this branch terminates in a leaf.

\begin{definition} 
Let $T=(V,E)$ be an unrooted tree. The \textit{boundary} $\partial T$ of $T$ is defined as the set of all branches.
\end{definition}

Like in Appendix 2, we endow $\partial T$ with the topology inherited from $\prod_{n\in\ZZ} \gen_n$ under the product topology, with the discrete topology on the generations. Thus $\partial T$ turns into a separable completely metrizable topological space. A base of the topology is again given by the sets 
\[
\partial T_v,\ v\in V
\]
of branches passing through $v$. These sets also generate the Borel $\sigma$-algebra $\mathcal{B}(\partial T)$ of $\partial T$.

\begin{example}
For the classical tree $T$ given by $V=\ZZ$, $\partial T$ can be identified with $\{\infty\}$. For the comb tree of Example \ref{ex-comb}, $\partial T$ can be identified with $\ZZ\cup \{\infty\}$. For a tree with a free left end, $\partial T= \partial T_{v_0}$, where $v_0$ is such that all its ancestors only have one child.
\end{example}

\begin{remark} 
Our choice of boundary reflects the order of the directed tree. One might also think of adding a supplementary boundary point $-\infty$, which can be understood as the point obtained by following the tree ``backwards.'' More formally, it can be represented by any reversed branch $(\ldots, v_1, v_0, v_{-1}, \ldots)$, where $v_{n-1}=\prt(v_{n})$ for all $n$. Note than any two such reversed branches have a common tail.

This point of view would have the advantage of making $\partial T\cup \{-\infty\}$ compact if the tree is locally finite. However, as we will see, if one looks at measures on the boundary that define flows on the tree, then $-\infty$ is irrelevant. The point $-\infty$ will reappear when we follow our flow to its limit.
\end{remark}

Similarly to the rooted case, we have a correspondence between (real or complex) Borel measures on $\partial T$ and (real or complex) flows on $V$. However, there is a restriction. The following was obtained in \cite[Proposition 2.2]{ChLe19} for locally finite, leafless rooted trees, but it holds for arbitrary trees. Let $L(V)$ denote the set of leaves of the tree.

\begin{proposition}\label{p-miflows}
Let $T=(V,E)$ be a directed (rooted or unrooted) tree with generations \emph{$\gen_n$}, $n\in\NN$ or $n\in\ZZ$. Let $f:V\to \KK$ be a flow. Then there is a Borel measure $\mu$ on $\partial T$ such that $f(v)=\mu(\partial T_v)$ for all $v\in V$ if and only if 
\[
\sum_{v\in L(V)} |f(v)|<\infty \quad\text{and}\quad \sup_{n\geq 0} \sum_{v\in \emph{$\gen_n$}} |f(v)|  <\infty.
\] 
In that case, $\mu$ is uniquely determined by $f$.
\end{proposition}

\begin{proof}
The necessity is obvious: Let the flow $f$ be induced by a Borel measure $\mu$ on $\partial T$. Since both the sets $\partial T_v$, $v\in L(V)$, and the sets $\partial T_v$, $v\in\gen_n$, form disjoint subsets of $\partial T$, we have that
\[
\sum_{v\in L(V)} |f(v)| = \sum_{v\in L(V)} |\mu(\partial T_v)| \leq \|\mu\|
\]
and
\[
\sum_{v\in\gen_n} |f(v)| = \sum_{v\in\gen_n} |\mu(\partial T_v)| \leq \|\mu\|,\,\, n\geq 0,
\]
where $\|\mu\|$ is the total variation of the measure, which is finite. 

As for sufficiency, we note that, by treating separately the real and imaginary parts of complex flows, we need only consider real flows. In addition, we may assume that the tree is leafless. Indeed, if $v$ is a leaf, we may simply extend the tree beyond $v$ by an infinite branch along which $f$ carries the (constant) value $f(v)$. This essentially leaves $\partial T$ unchanged, the tree becomes leafless, and the hypothesis on the initial tree implies the hypothesis on the modified tree.

Thus suppose that $f$ is a real flow on a leafless tree for which $\sup_{n\geq 0} \sum_{v\in \gen_n} |f(v)|<\infty$. We need to construct a real Borel measure $\mu$ on $\partial T$ that induces $f$. We begin as in \cite{ChLe19} by defining, for any $n\geq 0$, a measure $\mu_n$ on $\partial T$ by 
\[
\mu_n = \sum_{v\in \gen_n} f(v) \delta_{\zeta_v},
\]
where $\zeta_v$ is an arbitrarily chosen but fixed boundary point in $\partial T_v$ and $\delta$ denotes Dirac measure; recall that $\zeta_v$ is an (infinite) branch passing through $v$. The hypothesis implies that $\mu_n$ is well-defined.

A trivial but crucial property for the sequel is that, for any $v\in V$,
\begin{equation}\label{eq-flow1}
|f(v)|\leq \sum_{u\in \Chi(v)} |f(u)|,
\end{equation}
note that the tree is leafless; in particular, $\sum_{v\in \gen_n} |f(v)|$ is increasing in $n$. Let us write
\begin{equation}\label{eq-flow}
M=\sup_{n\geq 0} \sum_{v\in \gen_n} |f(v)|= \lim_{n\to\infty} \sum_{v\in\gen_n} |f(v)|.
\end{equation}

We will show that the sequence $(\mu_n)_{n\geq 0}$ of measures is uniformly tight. Thus, let $\varepsilon>0$. Then there is some $N\geq 0$ such that
\[
\sum_{v\in \gen_{N}} |f(v)| > M-\varepsilon,
\]
and hence there is a finite subset $V_0\subset \gen_{N}$ such that
\[
\sum_{v\in V_0} |f(v)| > M-\varepsilon.
\]
We write $\Chi(V_0)=\bigcup_{v\in V_0}\Chi(v)$. Since $\sum_{u\in \Chi(V_0)} |f(u)| \geq \sum_{v\in V_0} |f(v)|$ by \eqref{eq-flow1}, there is a finite subset $V_1\subset \Chi(V_0)$ such that 
\[
\sum_{v\in V_1} |f(v)| > M-\varepsilon.
\]
By making $V_1$ larger, if necessary, we can assume that, for any $v\in V_0$, the branch $\zeta_v$ starting from $v$ passes through a vertex in $V_1$. 

Continuing in this way, we find finite subsets $V_k\subset \gen_{N+k}$ such that, for any $k\geq 0$,
\begin{equation}\label{eq-tree2}
V_{k+1}\subset \Chi(V_k),\quad \sum_{v\in V_k} |f(v)| > M-\varepsilon
\end{equation}
and, for every $v\in V_k$, the branch $\zeta_v$ starting from $v$ passes through any $V_{k+j}$, $j\geq 0$.

Now let $W_v$, $v\in V_0$, be the rooted leafless subtree of $V$ with root $v$ whose vertices of generation $k\geq 0$ are given by $V_k\cap \Chi^k(v)$; thus $W_v$ is locally finite, so that its boundary
\[
K_v= \partial W_v, v\in V_0
\]
is compact. We can identify $K_v$ with a subset of $\partial T$, which is then also compact in $\partial T$. Finally consider the compact set $K\subset \partial T$ given by
\[
K=\bigcup_{v\in V_0} K_v.
\]

Let $n=N+k$, $k\geq 0$. By construction we have that $\zeta_v\in K$ for every $v\in V_k$. Thus
\begin{align*}
|\mu_n|(\partial T \setminus K) &= \sum_{v\in \gen_{N+k}\setminus V_k} |f(v)| \delta_{\zeta_v}(\partial T \setminus K) + \sum_{v\in V_k} |f(v)| \delta_{\zeta_v}(\partial T \setminus K)\\
& \leq \sum_{v\in \gen_{N+k}\setminus V_k} |f(v)| <\varepsilon,
\end{align*}
where the last inequality comes from \eqref{eq-tree2} with \eqref{eq-flow}. 

Finally, it follows from the definition of the measures $\mu_n$ and the hypothesis that there is a finite set $F\subset \partial T$ such that $|\mu_n|(\partial T\setminus F)<\varepsilon$ for $0\leq n <N$. 

Altogether we have shown that the sequence $(\mu_n)_{n\geq 0}$ of measures on $\partial T$ is uniformly tight. Also, by hypothesis, their total variations are uniformly bounded. Since $\partial T$ is a separable complete metric space, Prokhorov's theorem implies that there is a real Borel measure $\mu$ on $\partial T$ so that some subsequence $(\mu_{n_k})_k$ converges weakly to $\mu$, see \cite[Theorem 8.6.2]{Bog07}. 

We now show that $\mu$ induces the flow $f$. Let $v\in V$, and let $n\in\NN$ or $n\in\ZZ$ be such that $v\in\gen_n$. Since $f$ is a flow, we have whenever $n_k\geq n$
\begin{align*}
f(v)= \sum_{u\in \Chi^{n_k-n}(v)} f(u) &= \sum_{u\in \Chi^{n_k-n}(v)} \mu_{n_k}(\partial T_u)= \mu_{n_k}\Big(\bigcup_{u\in \Chi^{n_k-n}(v)}\partial T_u\Big)\\& = \mu_{n_k}(\partial T_v)= \int_{\partial T} \chi_{\partial T_v} \mathrm{d}\mu_{n_k}\to \int_{\partial T} \chi_{\partial T_v} \mathrm{d}\mu = \mu(\partial T_v)
\end{align*}
as $k\to\infty$, where we have used that $\partial T_v$ is open and closed in $\partial T$, so that $\chi_{\partial T_v}$ is continuous.

Thus $f$ is indeed induced by the Borel measure $\mu$. For the uniqueness, suppose that $f$ is induced by real measures $\mu_1$ and $\mu_2$. Using their Jordan decompositions we find that, for all $v\in V$, $\mu_1^+(\partial T_v)-\mu_1^-(\partial T_v)=f(v)=\mu_2^+(\partial T_v)-\mu_2^-(\partial T_v)$ and hence $\mu_1^+(\partial T_v)+\mu_2^-(\partial T_v)=\mu_2^+(\partial T_v)+\mu_1^-(\partial T_v)$. By the uniqueness theorem for (positive finite) measures we have that $\mu_1^++\mu_2^-=\mu_2^++\mu_1^-$ on $\mathcal{B}(\partial T)$ and hence that $\mu_1=\mu_2$.
\end{proof}

\begin{remark} Instead of using Prokhorov's theorem, one may also prove the proposition by applying the Riesz representation theorem as in \cite{ChLe19} to a metrizable compactification of $\partial T$, more precisely the one we get by embedding the boundary in $[0,1]^V$ using the functions $\chi_{\partial T_v}$, $v \in V$.
\end{remark}

We will call flows that arise from Borel measures on $\partial T$ \textit{measure-induced flows}, briefly \textit{m-i.\ flows}. It follows, in particular, that any positive and any negative flow on a rooted tree is measure-induced (by a positive or negative Borel measure); but this was already obtained by the construction in Appendix 2.

Of course, the real and imaginary parts of a complex flow are real flows. But it is an intriguing consequence of the Jordan decomposition theorem that any real m-i.\ flow is the difference of two positive flows:

\begin{corollary}
Let $T=(V,E)$ be a directed (rooted or unrooted) tree. Let $f:V\to \RR$ be a real flow. Then $f$ is an m-i.\ flow if and only if it can be written as the difference of two positive flows.
\end{corollary}

\begin{example}
To give an example of a real flow $f$ that is not measure-induced, consider once more the comb tree of Example \ref{ex-comb}. Define, for $n\in\ZZ$, the value of $f$ as 1 at the vertex $(2n,0)$, as $0$ at $(2n+1,0)$, as 1 at the tooth emanating from $(2n,0)$, and as $-1$ at the tooth emanating from $(2n+1,0)$. Then $f$ does not satisfy the condition of Proposition \ref{p-miflows}. Note that it can clearly not be decomposed as the difference of two positive flows.
\end{example}

\subsection{Potential theory on unrooted trees}\label{subs-unrooted}

After the preliminaries of the previous subsection we can now define capacity of the boundary for an unrooted tree. Since there is no root, we need to single out, arbitrarily, a vertex $v_0$. Also, as we will see, we have to accept real Borel measures on $\partial T$. Thus, in the sequel, measures and flows are supposed to be real (there is no need to allow them to be complex).

\begin{definition}\label{d-caproot}
Let $T=(V,E)$ be an unrooted tree, $v_0\in V$, $w=(w_v)_{v\in V}$ a positive weight on $V$, and $1<p<\infty$. Then the \textit{$p$-capacity} of the boundary $\partial T$ with respect to $v_0$ and $w$ is given by 
\[
\capac_p(\partial T,v_0,w)= \sup\Big\{ |\mu(\partial T_{v_0})|^p : \text{ $\mu$ a Borel measure on $\partial T$ with $\sum_{v\in V}|\mu(\partial T_v)w_v|^{p^\ast}\leq 1$}\Big\}.
\]
\end{definition}

The definition can again be reformulated in terms of measure-induced flows. We have that
\[
\capac_p(\partial T,v_0,w)= \sup\Big\{ |f(v_0)|^p : \text{ $f$ an m-i.\ flow on $V$ with $\sum_{v\in V}|f(v)w_v|^{p^\ast}\leq 1$}\Big\}
\]
and then
\[
\capac_p(\partial T,v_0,w)= \Big(\inf\Big\{ \sum_{v\in V}|f(v)w_v|^{p^\ast} : \text{ $f$ an m-i.\ flow on $V$ with $f(v_0)=1$}\Big\}\Big)^{-p/p^\ast},
\]
where we can call
\[
\mathcal{E}_p(f,w)= \sum_{v\in V}|f(v)w_v|^{p^\ast}
\]
the \textit{$p$-energy} of a flow.

We may ask again whether, given a fixed vertex $v_0\in V$, an m-i.\ flow  $f$ of finite energy with $f(v_0)=1$ exists, that is, if $\capac_p(\partial T,v_0,w)>0$. And in that case we are interested in a flow of minimal energy. The answer is provided by our discussion in Section \ref{s-unrooted}: indeed, an m-i.\ flow  $f$ of finite energy with $f(v_0)=1$ exists if and only if the restriction $f_+$ of $f$ to the rooted subtree $V_+(v_0)=V(v_0)$ and the sequence $f_-$ defined by \eqref{f-} on the rooted tree $V_-(v_0)$ (under the modified parent-child relationship) are m-i.\ flows of finite energy with $f_+(v_0)=f_-(v_0)=1$. The latter is characterized by Theorem \ref{t-charpinv}(a). Moreover, since m-i.\ unit flows of minimal energy on rooted trees with positive weights are positive, which correspond to positive Borel measures, we obtain with \eqref{f-} the following result, where for $\zeta=(v_n)_{n<m}\in \partial T$, $m\in\ZZ\cup\{\infty\}$, we write
\[
\lim_{v\to\zeta} f(v)= \lim_{n\to m} f(v_n)
\]
and
\[
\lim_{v\to-\infty} f(v)= \lim_{n\to -\infty} f(v_n),
\]
the latter being independent of the sequence $(v_n)_n$.

\begin{theorem}\label{t-charpinvunrootedbis}
Let $T=(V,E)$ be an unrooted tree, $v_0\in V$, $w=(w_v)_{v\in V}$ a positive weight on $V$, and $1< p<\infty$. Suppose that $B$ is defined on $\ell^p(V,w)$. 

\emph{(a)} The following assertions are equivalent:
\begin{enumerate}
	\item[\rm (i)] there exists an m-i.\ flow $f$ on $V$ of finite energy with $f(v_0)=1$;
	\item[\rm (ii)] $\capac_{p^\ast}(\partial T,v_0,w)>0$;
	\item[\rm (iii)] $r_p(V(v_0),w)<\infty$ and $r_p(V_-(v_0),w)<\infty$.
\end{enumerate}
In that case, 
\[
\capac_{p^\ast}(\partial T,v_0,w)= \big(r_p(V(v_0),w)^p+ r_p(V_-(v_0),w)^p - w_{v_0}^p\big)^{-p^\ast/p}.
\]

\emph{(b)} If $\capac_{p^\ast}(\partial T,v_0,w)>0$, there exists a unique m-i.\ flow $f$ on $V$ with $f(v_0)=1$ of minimal energy. Its corresponding measure $\mu$ satisfies $\mu\geq 0$ on $\partial T_{v_0}$ with $\mu(\partial T_{v_0})=1$, and $\mu\leq 0$ on $(\partial T_{v_0})^c$ with $\mu((\partial T_{v_0})^c)\in [-1,0]$. Moreover,
\[
\lim_{v\to -\infty} f(v)= 1+ \mu((\partial T_{v_0})^c) \in [0,1],
\]
and for every boundary point $\zeta\in\partial T$ we have that
\[
\lim_{v\to\zeta} f(v)= \mu(\{\zeta\})\in [-1,1].
\]
\end{theorem}

More precisely, by Remark \ref{r-decr2} and the discussion in Section \ref{s-unrooted} we have the following: For any branch $(v_n)_{n<m}$ that passes through $v_0$, the values $f(v_n)$ are positive and increasing for $n\leq 0$, then positive and decreasing for $n\geq 0$. For any branch $(v'_n)_{n<m}$ with $v'_{k}=\prt^{-k}(v_0)$ and $v'_{k+1}\neq \prt^{-k-1}(v_0)$ for some $k\leq -1$, the values $f(v'_n)$ are positive and increasing for $n\leq k$, then negative and increasing for $n\geq k+1$.

\begin{example}
Let $V$ be the unrooted binary tree, that is, the unrooted tree in which every vertex has exactly two children. We fix arbitrarily a vertex $v_0$, and choose all weights to be 1. The unique m-i.\ flow $f$ of minimal norm with $f(v_0)=1$ is shown in Figure \ref{fig-flow}. 
\end{example}

\begin{figure}
\begin{tikzpicture}[scale=1.7]
\draw[fill] (0,0) circle (.5pt);

\draw[->,>=latex] (-2.5,0.75) -- (-2,.6);\draw[fill] (-2,.6) circle (.5pt);

\draw[->,>=latex] (-2,0.6) -- (-1,.3);\draw[fill] (-1,.3) circle (.5pt);
\draw[->,>=latex] (-2,0.6) -- (-1,.9);\draw[fill] (-1,.9) circle (.5pt);
\draw[->,>=latex] (-1,0.9) -- (0,1.2);\draw[fill] (0,1.2) circle (.5pt);
\draw[->,>=latex] (-1,0.9) -- (0,1);\draw[fill] (0,1) circle (.5pt);
\draw[->,>=latex] (-1,0.3) -- (0,0);\draw[fill] (0,0) circle (.5pt);
\draw[->,>=latex] (-1,0.3) -- (0,0.5);\draw[fill] (0,0.5) circle (.5pt);

\draw[-] (0,1.2) -- (.3,1.3);
\draw[-] (0,1.2) -- (.3,1.2);
\draw[-] (0,1) -- (.3,1.05);
\draw[-] (0,1) -- (.3,.95);

\draw[->,>=latex] (0,0) -- (1,-.3);\draw[fill] (1,-.3) circle (.5pt);
\draw[->,>=latex] (0,0) -- (1,.2);\draw[fill] (1,.2) circle (.5pt);
\draw[->,>=latex] (0,0.5) -- (1,.8);\draw[fill] (1,.8) circle (.5pt);
\draw[->,>=latex] (0,0.5) -- (1,.6);\draw[fill] (1,.6) circle (.5pt);

\draw[-] (1,.8) -- (1.3,.9);
\draw[-] (1,.8) -- (1.3,.8);
\draw[-] (1,.6) -- (1.3,.65);
\draw[-] (1,.6) -- (1.3,.55);

\draw[->,>=latex] (1,-.3) -- (2,-.6);\draw[fill] (2,-.6) circle (.5pt);
\draw[->,>=latex] (1,-.3) -- (2,-.2);\draw[fill] (2,-.2) circle (.5pt);
\draw[->,>=latex] (1,.2) -- (2,.2);\draw[fill] (2,.2) circle (.5pt);
\draw[->,>=latex] (1,.2) -- (2,.4);\draw[fill] (2,.4) circle (.5pt);

\draw[-] (2,-.2) -- (2.3,-.15);
\draw[-] (2,-.2) -- (2.3,-.25);
\draw[-] (2,.2) -- (2.3,.15);
\draw[-] (2,.2) -- (2.3,.25);
\draw[-] (2,.4) -- (2.3,.5);
\draw[-] (2,.4) -- (2.3,.4);

\draw[-] (2,-.6) -- (2.3,-.7);

\node at (-2,0.4) {\footnotesize{$\tfrac{1}{4}$}};
\node at (-1,0.1) {\footnotesize{$\tfrac{1}{2}$}};
\node at (-.2,-.2) {\footnotesize{$f(v_0)=1$}};
\node at (1,-.5) {\footnotesize{$\tfrac{1}{2}$}};
\node at (2,-.8) {\footnotesize{$\tfrac{1}{4}$}};

\node at (-1,0.7) {\footnotesize{$-\tfrac{1}{4}$}};
\node at (-.1,1.4) {\footnotesize{$-\tfrac{1}{8}$}};
\node at (-.1,.8) {\footnotesize{$-\tfrac{1}{8}$}};
\node at (-.1,.3) {\footnotesize{$-\tfrac{1}{2}$}};
\node at (1,0) {\footnotesize{$\tfrac{1}{2}$}};
\node at (.9,0.4) {\footnotesize{$-\tfrac{1}{4}$}};
\node at (.9,1) {\footnotesize{$-\tfrac{1}{4}$}};
\node at (2,-.4) {\footnotesize{$\tfrac{1}{4}$}};
\node at (2,.0) {\footnotesize{$\tfrac{1}{4}$}};
\node at (2,.65) {\footnotesize{$\tfrac{1}{4}$}};

\end{tikzpicture}
\caption{A binary tree with a flow of minimal energy}%
\label{fig-flow}
\end{figure}
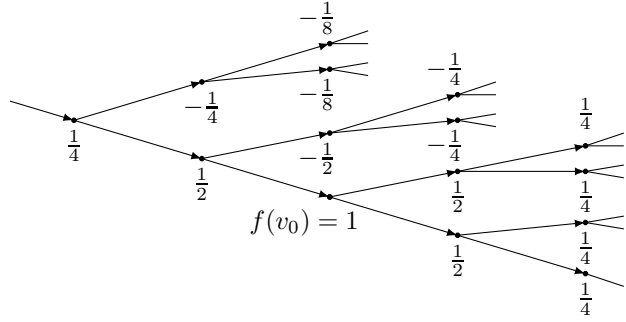

One might again call the unique Borel measure $\mu$ on $\partial T$ corresponding to the flow $f$ of minimal energy with $f(v_0)=1$ a \textit{harmonic measure} or an \textit{equilibrium measure}.

\subsection{Potential theory on rooted trees II}\label{subs-rootII}

The fact that, in an unrooted tree, one has both the necessity and the freedom to choose a distinguished point suggests to do likewise in a rooted tree. 

Thus, in this subsection, let $T=(V,E)$ be a rooted tree and $v_0\in V$ a vertex that is not necessarily the root; the latter will here be denoted by $\rt$. We are now interested in finding flows of minimal energy with value 1 at $v_0$. For this we define capacity with respect to $v_0$, where we now have to admit real Borel measures and real flows.

\begin{definition} 
Let $T=(V,E)$ be a rooted tree, $v_0\in V$, $w=(w_v)_{v\in V}$ a positive weight on $V$, and $1<p<\infty$. Then the \textit{$p$-capacity} of the boundary $\partial T$ with respect to $v_0$ and $w$ is given by 
\[
\capac_p(\partial T,v_0,w)= \sup\Big\{ |\mu(\partial T_{v_0})|^p : \text{ $\mu$ a Borel measure on $\partial T$ with $\sum_{v\in V}|\mu(\partial T_v)w_v|^{p^\ast}\leq 1$}\Big\}.
\]
\end{definition}

Since the extremal measure in the case of $v_0=\rt$ is necessarily a positive measure or a negative measure, we have that $\capac_p(\partial T,\rt,w)$ coincides with the capacity of Definition \ref{def-cap}. 

We find again that
\begin{align*}
\capac_p(\partial T,v_0,w)&= \sup\Big\{ |f(v_0)|^p : \text{ $f$ an m-i.\ flow on $V$ with $\sum_{v\in V}|f(v)w_v|^{p^\ast}\leq 1$}\Big\}\\
&= \Big(\inf\Big\{ \sum_{v\in V}|f(v)w_v|^{p^\ast} : \text{ $f$ an m-i.\ flow on $V$ with $f(v_0)=1$}\Big\}\Big)^{-p/p^\ast}.
\end{align*}
We repeat the procedure of Section \ref{s-unrooted} of decomposing a flow $f$ on $V$ with $f(v_0)=1$ as a flow $f_+$ on the subtree $V_+(v_0)=V(v_0)$ and a flow $f_-$ defined by \eqref{f-} on the modified tree $V_-(v_0)$. The only difference is that the lower branch in Figure \ref{fig-v_1b}, coming from the ancestors of $v_0$, is finite and ends with $\rt$. This also implies that $r_p(V_-(v_0),w)<\infty$, see Corollary \ref{c-subtrees}(a). 

\begin{theorem}\label{t-charpinvunrootedter}
Let $T=(V,E)$ be a rooted tree, $v_0\in V$, $w=(w_v)_{v\in V}$ a positive weight on $V$, and $1< p<\infty$. Suppose that $B$ is defined on $\ell^p(V,w)$. 

\emph{(a)} The following assertions are equivalent:
\begin{enumerate}
	\item[\rm (i)] there exists an m-i.\ flow $f$ on $V$ of finite energy with $f(v_0)=1$;
	\item[\rm (ii)] $\capac_{p^\ast}(\partial T,v_0,w)>0$;
	\item[\rm (iii)] $r_p(V(v_0),w)<\infty$.
\end{enumerate}
In that case, 
\[
\capac_{p^\ast}(\partial T,v_0,w)= \big(r_p(V(v_0),w)^p+ r_p(V_-(v_0),w)^p - w_{v_0}^p\big)^{-p^\ast/p}.
\]

\emph{(b)} If $\capac_{p^\ast}(\partial T,v_0,w)>0$, there exists a unique m-i.\ flow $f$ on $V$ with $f(v_0)=1$ of minimal energy. Its corresponding measure $\mu$ satisfies $\mu\geq 0$ on $\partial T_{v_0}$ with $\mu(\partial T_{v_0})=1$, and $\mu\leq 0$ on $(\partial T_{v_0})^c$ with $\mu((\partial T_{v_0})^c)\in [-1,0]$. Moreover,
\[
f(\rt)= 1+ \mu((\partial T_{v_0})^c) \in [0,1],
\]
and for every boundary point $\zeta\in\partial T$ we have that
\[
\lim_{v\to\zeta} f(v)= \mu(\{\zeta\})\in [-1,1].
\]
\end{theorem}

One has again the monotonicity structure described after Theorem \ref{t-charpinvunrootedbis}. 

The similarity of the results in the last two subsections suggests to view unrooted trees as trees with root at $-\infty$.

\end{document}